\documentclass[10pt]{article}
\input{package.tex}

\def\proto{\xslashedrightarrow{}}

\newcommand{\Nb}{\mathbb{N}}

\newcommand{\Ib}{\mathbb{I}}

\renewcommand{\H}{\mathbf{H}}

\def\-{\raisebox{.75pt}{-}}

\newcommand{\uvar}{\_}
\newcommand{\C}{\textbf{č}}

\newcommand{\Db}{\mathbf{D}}

\newcommand{\Lb}{\mathbf{L}} 
\newcommand{\Fb}{\mathbf{F}}

\DeclareMathOperator{\N}{N}

\DeclareMathOperator{\Pair}{Pair}
\DeclareMathOperator{\Cube}{Cb}

\DeclareMathOperator*{\Woperator}{W}

\DeclareMathOperator*{\Wsegoperator}{W_{Seg}}
\DeclareMathOperator*{\Wsatoperator}{W_{Sat}}

\newcommand{\W}{\Woperator}

\newcommand{\Wseg}{\Wsegoperator}
\newcommand{\Wsat}{\Wsatoperator}

\newcommand{\TW}{\mathrm{Tw}}

\DeclareMathOperator*{\M}{M}

\DeclareMathOperator*{\Mseg}{M_{Seg}}
\DeclareMathOperator*{\Msat}{M_{Sat}}

\DeclareMathOperator*{\I}{I}
\DeclareMathOperator*{\F}{F}
\DeclareMathOperator*{\IF}{IF}

\DeclareMathOperator\Set{Set}
\DeclareMathOperator\Sp{Sp}
\DeclareMathOperator\Surj{Surj}
\DeclareMathOperator\Hot{Hot}

\DeclareMathOperator*{\Sq}{Sq}

\DeclareMathOperator{\Hom}{Hom}

\DeclareMathOperator{\oppath}{path}
\renewcommand\path{\oppath}

\DeclareMathOperator*{\LFib}{LFib}

\DeclareMathOperator*{\LCartoperator}{LCart}
\DeclareMathOperator*{\BCartoperator}{BCart}

\DeclareMathOperator*{\RCartoperator}{RCart}

\newcommand{\BCart}{\mbox{${\BCartoperator}$}}

\newcommand{\RCart}{\mbox{${\RCartoperator}$}}
\newcommand{\LCart}{\mbox{${\LCartoperator}$}}

\newcommand{\LCartc}{\mbox{${\LCartoperator}^c$}}

\DeclareMathOperator*{\Hor}{Hor}

\DeclareMathOperator*{\Comp}{Comp}

\DeclareMathOperator{\im}{Im}

\DeclareMathOperator*{\ev}{ev}
\DeclareMathOperator*{\Arr}{Arr}

\newcommand{\coend}{\operatornamewithlimits{coend}}
\newcommand{\cocoend}{\operatornamewithlimits{end}}

\newcommand{\colim}{\operatornamewithlimits{colim}}

\newcommand\dirtimes[1]{\underset{#1}{\overset{\to}{\times}}}

\newcommand\zo{(0,\omega)}

\DeclareMathOperator\cst{cst}

\DeclareMathOperator\Fun{Fun}

\DeclareMathOperator\cat{Cat}
\DeclareMathOperator\Dcatoperatore{DblCat}

\newcommand\ocat{\omega\mbox{-$\cat$}}

\newcommand\ocatm{{\omega\mbox{-$\cat_m$}}}

\newcommand\zocat{(0,\omega)\mbox{-$\cat$}}

\newcommand\icat{1\mbox{-$\cat$}}

\DeclareMathOperator{\intoperotor}{int}
\newcommand{\intw}{\intoperotor}

\newcommand{\ncat}{n\mbox{-$\cat$}}
\newcommand{\tcat}{\mbox{-$\cat$}}
\newcommand{\npmcat}{(n+m)\mbox{-$\cat$}}

\newcommand{\npcat}{(n+1)\mbox{-$\cat$}}

\newcommand{\Dcat}{\mbox{-$\Dcatoperatore$}}

\newcommand{\Dcatc}{\mbox{-$\Dcatoperatore^{c}$}}
\newcommand{\Dcatmc}{\mbox{-$\Dcatoperatore_m^{c}$}}
\newcommand{\Dcatm}{\mbox{-$\Dcatoperatore_m$}}
\newcommand{\oDcatm}{(\omega,1)\mbox{-$\Dcatoperatore_m$}}

\newcommand{\DI}{\I_{dbl}}

\newcommand{\nDcat}{(n,1)\Dcat}

\newcommand{\nDcatc}{(n,1)\Dcatc}

\DeclareMathOperator{\Filt}{Filt}
\newcommand{\Filts}{\Filt^{\twoheadrightarrow}}

\title{Effectivity of Generalized Double $\infty$-Categories}
\author{Félix Loubaton\\ Max Planck Institute for Mathematics}
\date{}
\begin{document}
\maketitle

\begin{abstract}
We construct an adjunction between $m$-categories internal to $(\infty,n)$-categories, called $(n,m)$-double $\infty$-categories, and filtrations $A_0\to \dots\to A_m$ where for all $i<m$, $A_i$ is a $(n+i)$-category. We show that this adjunction induces an equivalence between $(n,m)$-double $\infty$-categories admitting enough companions and filtrations such that each morphism $A_i\to A_{i+1}$ is essentially surjective on cells of dimension lower than or equal to $i$.

This result can be seen as a $(\infty,n)$-categorical generalization of the equivalence between internal groupoids and effective epimorphisms in the category of $\infty$-groupoids proven by Rezk and Lurie.

In the case $n = 0$, this recovers the characterization of flagged $m$-categories given by Ayala-Francis in \cite{ayala2018flagged}, and in the case $n = 1$, it allows us to prove some conjectures concerning the square functor and its variants, stated by Gaitsgory-Rozenblyum in the appendix of \cite{Gaitsgory_A_study_on_DAG}.

\end{abstract}

\tableofcontents

\section*{Introduction}
\addcontentsline{toc}{section}{Introduction}

\vspace{0.5cm}
A fundamental property of the category of sets is that it admits a good notion of \textit{quotient}. Indeed, given a transitive, reflexive, and symmetric relation $R\subset E\times E$ on a set $E$, one can form the quotient of $E$ by $R$, denoted $E_{/R}$, and recover $R$ by inspecting the morphism $p:E\to E_{/R}$: given two elements $x,y$ in $E$, $xRy$ if and only if $p(x)=p(y)$. We will then say that transitive, reflexive, and symmetric relations are \textit{effective}.

\vspace{0.5cm}

A similar result holds in the $\infty$-category $\Hot$ of homotopy types. The analogue of the notion of transitive and reflexive relation is that of a Segal object, i.e., a functor $D_\bullet:\Delta^{op}\to \Hot$ such that the canonical morphism 
 $$D_n\to D_1\times_{D_0}...\times_{D_0}D_1$$
 is an equivalence. The quotient of such a Segal object is the colimit $\colim_{\Delta^{op}}D_\bullet$. Informally, this object is obtained from $D_0$ by adding a path between $D_{d^1}(x)$ and $D_{d^0}(x)$ for every element $x$ in $D_1$, as well as coherences between all these paths.

An \textit{internal groupoid} is a Segal object $D_\bullet$ such that the two morphisms
$$D_{[2]}\to D_{\Lambda^{0}[2]}~~~~~~~ D_{[2]}\to D_{\Lambda^{2}[2]}$$
are equivalences. This condition can be thought of as a kind of "symmetry" for the "relation" $D_1\to D_0\times D_0$. By results of Rezk and Lurie, one can once again recover the simplicial object $D_\bullet$ by inspecting the morphism $p:D_0\to \colim_{\Delta^{op}}D_\bullet$. More precisely, $D_\bullet$ is equivalent to the simplicial object $\C_\bullet(p)$, called the \textit{Čech nerve} of $p$, defined by the formula:
$$\C_n(p):=D_0\times_{\colim_{\Delta^{op}}D_\bullet}...\times_{\colim_{\Delta^{op}}D_\bullet}D_0$$
We will then say that internal groupoids are \textit{effective}.

\vspace{0.5cm}

The starting point of the work presented here is to investigate what the $(\infty,n)$-categorical analogue of this notion of effectiveness should and could be. What underlies this question is the observation that the effectiveness of equivalence relations is one of Giraud's axioms characterizing toposes, and, similarly, the effectiveness of internal groupoids is one of the Giraud-Rezk-Lurie axioms characterizing $(\infty,1)$-toposes. We therefore hope that understanding its $(\infty,n)$-categorical generalizations will allow us to state and prove a theorem \textit{à la} Giraud that would characterize $(\infty,n)$-toposes with several axioms, one of which would be this effectiveness. This link between higher effectiveness and (categorically) higher topos theory is successfully used by Street in his definition of $2$-toposes in the article \cite{street1982two}, which is moreover an important source of inspiration for the present work.

\vspace{0.5cm}

Let $D_\bullet$ be a \textit{$(n,1)$-double $\infty$-category}, i.e., a functor $D_\bullet:\Delta^{op}\to (\infty,n)\tcat$ such that the canonical morphism 
 $$D_n\to D_1\times_{D_0}...\times_{D_0}D_1$$
 is an equivalence\footnote{We chose to call $(n,1)$-double $\infty$-categories the categories internal to $(\infty,n)$-categories to stick with the convention of denoting $(n,1)$-categories the categories internal to $n$-homotopy types.}. Informally, the \textit{quotient} of $D_\bullet$ is the $(\infty,n+1)$-category obtained from $D_0$ by adding a morphism $f_x$ between $D_{d^1}(x)$ and $D_{d^0}(x)$ for every element $x$ in $D_1$, a $2$-cell $\alpha_u:f_x\to f_y$ for every $1$-cell $u:x\to y$ in $D_1$, and so on. More formally, the quotient of $D_\bullet$ is the $(\infty,n+1)$-category $\coend_{\Delta} D_n\otimes[n]$, where $\otimes$ is the Gray tensor product.

The first result of this text is to identify precise conditions under which one can recover $D_\bullet$ by inspecting the morphism $D_0\to \coend_{\Delta}D_n\otimes[n]$. Before stating it, we need a few definitions. A vertical $1$-cell in $D_\bullet$ admits a \textit{companion} if there exists a square
\[\begin{tikzcd}
	a & b \\
	b & b
	\arrow["{\overline{f}}", from=1-1, to=1-2]
	\arrow["f"', from=1-1, to=2-1]
	\arrow["{\alpha_f}"{description}, draw=none, from=1-1, to=2-2]
	\arrow[equals, from=1-2, to=2-2]
	\arrow[equals, from=2-1, to=2-2]
\end{tikzcd}\]
invertible in a certain sense\footnote{We use here the classical terminology of the theory of double $\infty$-categories: a vertical $1$-cell is a $1$-cell of $D_0$, and a square is a $1$-cell of $D_1$.}.

We can then recursively define the notion of a \textit{companion} for a $k$-cell of $D_0$, and we will say that $D_\bullet$ is \textit{accompanied} if all the cells of $D_0$ admit companions (definition \ref{defi:accompange}). We denote by $(n,1)\Dcat$ the $\infty$-category of $(n,1)$-double $\infty$-categories and $(n,1)\Dcatc$ the full sub $\infty$-category of accompanied ones. It is important to note that the notion of companion is, beyond the scope of this text, an important concept in the theory of double ($\infty$-)categories (\cite{grandis2004adjoint}, \cite{shulman2007framed}, \cite{ruit2023formal}).

We define a \textit{$(n,1)$-filtration} as a morphism $A_0\to A_1$ with $A_0$ an $(\infty,n)$-category and $A_1$ an $(\infty,n+1)$-category. A $(n,1)$-filtration is \textit{surjective} if the morphism $A_0\to A_1$ is a \textit{$0$-surjection}, i.e., a functor that is essentially surjective on objects. We denote by $\Filt_{n,1}$ the $\infty$-category of $(n,1)$-filtrations and $\Filts_{n,1}$ the sub $\infty$-category of surjective ones.

\begin{itheorem}[\ref{theo:eff of n1 dcat}]
There exists an adjunction, whose left adjoint is called the \emph{realization}, and whose right adjoint is called the \emph{Čech nerve}:
$$\begin{tikzcd}
	{|\uvar|:(n,1)\Dcat} & {\Filt_{n,1}:\C}
	\arrow[""{name=0, anchor=center, inner sep=0}, shift left=2, from=1-1, to=1-2]
	\arrow[""{name=1, anchor=center, inner sep=0}, shift left=2, from=1-2, to=1-1]
	\arrow["\dashv"{anchor=center, rotate=-90}, draw=none, from=0, to=1]
\end{tikzcd}$$
with the realization defined by the formula:
$$|D_\bullet|:=D_0\to \coend_{\Delta}D_n\otimes[n]$$

The realization of a $(n,1)$-double $\infty$-category is always a surjective $(n,1)$-filtration, and the Čech nerve of a $(n,1)$-filtration is always an accompanied $(n,1)$-double $\infty$-category.

The realization and the Čech nerve induce inverse equivalences:
$$|\uvar|:(n,1)\Dcatc\sim \Filts_{n,1}:\C$$
\end{itheorem}

\vspace{0.5cm}

The second result of this text is to generalize the above theorem to a setting where $\Delta$ is replaced by $\Theta_m$, where $m\in \Nb\cup \{\omega\}$. A \textit{$(n,m)$-double $\infty$-category} is a functor $D_\bullet:\Theta_{m}^{op}\to \ncat$ satisfying the Segal condition (Definition \ref{defi:of nmdouble cate}). We then define a notion of \textit{accompanied} $(n,m)$-double $\infty$-categories (Definition \ref{defi:nmaccomp}). We denote by $(n,m)\Dcat$ the $\infty$-category of $(n,1)$-double $\infty$-categories and $(n,m)\Dcatc$ the full sub $\infty$-category of accompanied ones.

An \textit{$(n,m)$-filtration} is a sequence of morphisms $A_0\to A_1\to \dots \to A_m$ if $m<\omega$, or an infinite sequence $A_0\to A_1\dots$ if $m=\omega$, and where for any integer $i\leq m$, $A_i$ is a $(\infty,n+i)$-category. Such a filtration is \textit{surjective} if for any $k<m$, $A_k\to A_{k+1}$ is $k$-surjective, i.e., essentially surjective on objects and locally $(k-1)$-surjective. We denote by $\Filt_{n,m}$ the $\infty$-category of $(n,m)$-filtrations and $\Filts_{n,m}$ the sub $\infty$-category of surjective ones.

\begin{itheorem}[\ref{theo:eff of nm dcat}]
There exists an adjunction, whose left adjoint is called once again the \emph{realization}, and whose right adjoint is called once again the \emph{Čech nerve}:
$$\begin{tikzcd}
	{|\uvar|:(n,m)\Dcat} & {\Filt_{n,m}:\C}
	\arrow[""{name=0, anchor=center, inner sep=0}, shift left=2, from=1-1, to=1-2]
	\arrow[""{name=1, anchor=center, inner sep=0}, shift left=2, from=1-2, to=1-1]
	\arrow["\dashv"{anchor=center, rotate=-90}, draw=none, from=0, to=1]
\end{tikzcd}$$
with the realization defined by the formula:
$$|D_\bullet|:=D_0\to \coend_{\Delta}D_n\otimes[n]\to ...\to \coend_{\Theta_k}D_a\otimes a\to...$$

The realization of a $(n,m)$-double $\infty$-category is always a surjective $(n,m)$-filtration, and the  Čech nerve of a $(n,m)$-filtration is always an accompanied  $(n,m)$-double $\infty$-category.

The realization and the directed marked Čech nerve induce inverse equivalences:
$$|\uvar|:(n,m)\Dcatc\sim \Filts_{n,m}:\C$$
\end{itheorem}

\subsubsection*{Relation to other work}
\addcontentsline{toc}{subsection}{Relation to other work}

Applied in the case $n=0$, theorem \ref{theo:eff of nm dcat} is equivalent to theorem 0.26 of \cite{ayala2018flagged}, which gives an equivalence between flagged $m$-categories and $\Filts_{0,m}$.

The unproven theorems 4.1.3, 4.3.5, 4.6.3, and 5.2.3 of the appendix on $(\infty,2)$-categories of \cite{Gaitsgory_A_study_on_DAG} can be derived from theorem \ref{theo:eff of n1 dcat} applied to $n=1$ (remark \ref{rem:gait}). A proof of theorem 4.1.3 of \cite{Gaitsgory_A_study_on_DAG} was already given by Abellán in \cite{Abellan2023}.

In \cite{ruit2023formal}, several conjectures are stated about double $\infty$-categories. One of them, hypothesis 3.7, is implied by theorem \ref{theo:square commutes with cil}.

In the strict case, a variant of theorem \ref{theo:eff of n1 dcat} is proven by Street in \cite{street1982two}.

\subsubsection*{Organization of the article}
\addcontentsline{toc}{subsection}{Organization of the article}

Section \ref{section:Recalls and preliminaries} is mainly composed of background material. Nonetheless, it includes results of independent interest, such as the preservation of strict $\omega$-categories under the Gray tensor product with globular sums (theorem \ref{theo:strictness of Gray1}), a decomposition of the Gray tensor product of suspensions into simpler elements (proposition \ref{prop:decomposition}), and the construction of a monoidal structure on marked $\omega$-categories extending the Gray tensor product (theorem \ref{theo:Gray tensor product is monoidal on marked}).

\vspace{0.5cm}
Section \ref{section:Two-sided fibrations} is devoted to the notion of two-sided fibration of marked $\omega$-categories. We prove there theorem \ref{theo:eff of n1 dcat marked}, which is a “marked” version of the effectivity of marked $(\omega,1)$-categories.

\vspace{0.5cm}
Finally, Section \ref{section:Effectivity of accompanied $(n,m)$-double categories} is devoted to the proofs of the theorems mentioned in the introduction. We also state a few corollaries  such as corollary \ref{cor:characterization of n surjective}, which shows that $n$-surjective morphisms are characterized by the fact that they have the unique right lifting property against $(n+1)$-fully faithful morphisms, and results on the square functor and its variants (section \ref{section n1 eff}).

\subsubsection*{Acknowledgment}
\addcontentsline{toc}{subsection}{Acknowledgment}

I would like to thank Thomas Blom and Jaco Ruit for helpful discussions.

\section*{Conventions}
\addcontentsline{toc}{section}{Conventions}

\begin{enumerate}
\item[$*$] Unless explicitly stated, we will only manipulate $\infty$-categories and $(\infty,n)$-categories in this text. To simplify the notation, we will drop the symbol $\infty$.
\item[$*$] We denote by $\Hot$ the category of homotopy types. A subhomotopy type of a homotopy type $B$ is a homotopy type $A$ and a monomorphism $A\to B$.
\item[$*$] Given a category $A$, we denote by $\widehat{A}$ the category $\Fun(A^{op},\Hot)$. 
\item[$*$]  A class of morphisms $S$ in a category $C$ is \textit{cocomplete} if it is stable under colimits, compositions, and contains identities. We denote $\widehat{S}$ the smallest cocomplete class of morphisms containing $S$. By \cite[proposition 2.1.2.10]{loubaton2024categorical}, $\widehat{S}$ is stable under colimits, transfinite composition, pushouts, left cancellation, and retracts.
\item[$*$] We denote by $\widehat{A}[S^{-1}]$ the localization of $\widehat{A}$ by $S$, or equivalently, the subcategory of $\widehat{A}$ whose objects are the $S$-local. We then denote by $\Fb_S$ and $\iota_S$ the two adjoints: 
\[\begin{tikzcd}
	{\Fb_S:\widehat{A}} & {\widehat{A}[S^{-1}]:\iota}
	\arrow[""{name=0, anchor=center, inner sep=0}, shift left=2, from=1-1, to=1-2]
	\arrow[""{name=1, anchor=center, inner sep=0}, shift left=2, from=1-2, to=1-1]
	\arrow["\dashv"{anchor=center, rotate=-90}, draw=none, from=0, to=1]
\end{tikzcd}\]
The unit of this adjunction is in $\widehat{S}$ and the counit is the identity.
\item[$*$]Given a functor $G_\bullet:A^{op}\to D$ with values in a complete category and a presheaf $X$ on $A$, we will denote by $G_X$ the element of $D$ corresponding to $\lim_{\Theta_n/X}G$.
\item[$*$] Given a bifunctor $F:C^{op}\times C\to D$, we will denote $\coend_C F(c,c)$ the \textit{coend} of $F$, i.e., the colimit $\colim_{c\to c'\in \TW(C)}F(c,c')$ where $\TW(C)$ denotes the \emph{twisted arrow category} of $C$.
\item[$*$] Given an integer $k$, we set by convention $\omega + k = \omega - k := \omega$.

\end{enumerate}

\section{Recalls and preliminaries}

\label{section:Recalls and preliminaries}

\subsection{$n$-Categories}

\begin{definition}
We denote by $\Theta$ the Joyal globe (strict) category. This category was originally defined in \cite{Joyal_Theta}, but we will use the notation given in Section 1.1.2 of \cite{loubaton2024categorical} and we recall here some of them.

The objects of $\Theta$ are called \textit{globular sums}. Given a sequence $\textbf{a}:=\{a_0,...,a_{n-1}\}$ of globular sums in $\Theta$, we denote by $[\textbf{a},n]$ the globular sum whose objects are $0,1,...,n$ and such that $\hom_{[\textbf{a},n]}(k,k+1)= a_k$ for any $k<n$.

We define the morphism $\Sigma^n:\Theta\to \Theta$ by induction using the formula $\Sigma^0:=id$ and $\Sigma^{n+1}:=[\uvar,1]\circ\Sigma^n$. The \textit{$n$-globe} is the globular sum $\Db_n:=\Sigma^n[0]$.

Given two globular sums $[\textbf{a},n]$ and $[\textbf{b},m]$, we define 
$$[\textbf{a},n]\vee [\textbf{b},m]:=[\textbf{a}\cdot \textbf{b},n+m].$$

We will often reason by induction on the elements of $\Theta$, using the fact that every globular sum is either of the form $[0]$ or of the form $[\textbf{a},n]$.

We define the \textit{dimension of a globular sum $a$}, denoted by $|a|$, recursively as follows:
$$|[0]|:= 0~~~~~~~~|[\textbf{a},n]|:=1+\max_{k<n}|a_k|.$$

We denote by $\Theta_n$ the full subcategory of $\Theta$ whose objects are globular sums of dimension less than or equal to $n$. We then have:
$$\Theta_0\cong \{[0]\}~~~~~~~~~\Theta_1\cong \Delta~~~~~~\Theta_{\omega}\cong \Theta.$$
\end{definition}

\begin{definition}
We define the functor $\Sp:\Theta_n\to \widehat{\Theta_n}$ by induction, setting $$\Sp_{[0]}:=[0]~~~~\Sp_{[\textbf{a},1]}:=\colim_{\Db_n\to \Sp_a }\Db_{n+1}$$ and $\Sp_{[\textbf{a},n]}$ as the colimit of the diagram: \[\begin{tikzcd} {\Sp_{\{1\}}} & {\Sp_{\{2\}}} & {...} \\ {\Sp_{[a_0,1]}} & {\Sp_{[a_1,1]}} & {...} & {\Sp_{[a_{n-1},1]}} \arrow[from=1-1, to=2-1] \arrow[from=1-1, to=2-2] \arrow[from=1-2, to=2-2] \arrow[from=1-2, to=2-3] \arrow[from=1-3, to=2-3] \arrow[from=1-3, to=2-4] \end{tikzcd}\] We denote by $(\Wseg)_n$ the set of maps $\{\Sp_a\to a, a\in \Theta_n\}$. 
We denote by $E_{eq}$ the simplicial set corresponding to the diagram \[\begin{tikzcd} & 0 && 0 \\ 1 && 1 \arrow[equals, from=1-2, to=1-4] \arrow[from=1-2, to=2-3] \arrow[from=2-1, to=1-2] \arrow[equals, from=2-1, to=2-3] \arrow[from=2-3, to=1-4] \end{tikzcd}\] and we denote by $(\Wsat)_n$ the set of maps $\{\Sigma^kE_{eq}\to \Sigma^k[0]=:\Db_k\}$. 	
Finally, we set $$\mbox{$\W_n$}:=(\Wseg)_n\cup (\Wsat)_n.$$
\end{definition}

\begin{definition}
A \textit{$n$-category} is an object of $\widehat{\Theta_n}$ that is local with respect to $\W_n$. We then set 
$$\mbox{$\ncat:= \widehat{\Theta_n}[\W_n^{-1}]$}$$ 
In particular, we have: 
$$0\mbox{-$\cat$}\sim \Hot~~~~~~~ 1\mbox{-$\cat$}\sim \cat$$

Given a $n$-category $C$, we will simply denote by $C_k$ the homotopy type $\hom_{n\text{-Cat}}(\Db_k,C)$ and we will refer to it as the \textit{homotopy type of $k$-cells}. Given $k<l$, we denote by 
$$\pi^{-}_k, \pi^{+}_k:C_l\to C_k$$ the $k$-sources and $k$-targets which are the two functors induced by the inclusion $i_k^-,i_k^+:\Db_k\to \Db_n$.

The \textit{unit} is the canonical functor
$$\Ib:C_n\to C_{n+1}$$
induced by the morphism $\Db_{n+1}\to \Db_n$.
As $n$-categories are local to the map $\Db_n\coprod_{\Db_i}\Db_m\to\Db_n\vee_{\Db_i}\Db_m$, any triplet of integers $n>i$, $m>i$ induces a morphism:
$$\circ_i:C_n\times_{C_i}C_m\to C_{\max(n,m)}$$ called the
\textit{$i$-composition}.

The $i$-compositions are moreover associative, which in particular implies that we have equivalence 
$$(a\circ_i b)\circ_i c\sim a\circ_i (b\circ_i c) $$
and 
$$(a\circ_i b)\circ_j (c\circ_i d)\sim (a\circ_j c)\circ_i (b\circ_j d)$$
whenever these composites are defined.
\end{definition}

\begin{remark}
\label{rem:consequence of associativity of composition}
Let $i+1<k\leq n$ be three integers, $C$ a $n$-category, and $x,y$ two $i$-composable $k$-cells. 
By the associativity of composition, we have the equivalence:
$$(x \circ_i \pi_{i+1}^+y)\circ_{k-1} (\pi_{i+1}^-x\circ_i y)\sim x \circ_i y\sim (\pi_{i+1}^+x\circ_i y)\circ_{k-1}(x \circ_i \pi_{i+1}^-y)$$
\end{remark}

\begin{definition}
For all $n<m$, the inclusion $\Theta_n\to \Theta_m$ induces a triplet of adjoints:
\[\begin{tikzcd}
	{\iota:\ncat} & {m\mbox{-$\cat$}}
	\arrow[""{name=0, anchor=center, inner sep=0}, from=1-1, to=1-2]
	\arrow[""{name=1, anchor=center, inner sep=0}, "{\tau_n}", shift left=3, from=1-2, to=1-1]
	\arrow[""{name=2, anchor=center, inner sep=0}, "{\tau_n^i}"', shift right=3, from=1-2, to=1-1]
	\arrow["\dashv"{anchor=center, rotate=-90}, draw=none, from=0, to=1]
	\arrow["\dashv"{anchor=center, rotate=-90}, draw=none, from=2, to=0]
\end{tikzcd}\]
The left adjoint $\tau^i_n$ is the \textit{intelligent $n$-truncation functor}, and can be understood as the functor inverting any cell of dimension higher than $n$. The right adjoint $\tau_n$ is the \textit{$n$-truncation functor} and can be understood as the functor discarding any cell of dimension higher than $n$.
\end{definition}

\begin{definition}
The functor $[\uvar,1]:\Theta_n\to \Theta_{n+1}$ induces by extension by colimits a functor
$$\begin{array}{rcl}
\ncat&\to&\npcat_{\bullet,\bullet}\\
C&\mapsto & ([C,1],\{0\},\{1\})
\end{array}
$$
called the \textit{suspension},
which admits a right adjoint
$$\begin{array}{rcl}
\npcat_{\bullet,\bullet}&\to&\ncat\\
(C,a,b)&\mapsto & \hom_C(a,b)
\end{array}
$$
By construction $n$-cells of $ \hom_C(a,b)$ correspond to $(n+1)$-cell  of $C$ whose $0$-sources are $a$ and $0$-target are $b$.
\end{definition}

\begin{definition}
\label{defi:dualities non strict case} For any subset $S$ of $\Nb^*$, we denote by $(\uvar)^S:\Theta\to \Theta$ the functor that inverts the direction of globes of dimensions belonging to $S$. These functors induce, by extension by colimits, functors $$(\uvar)^S:\widehat{\Theta}\to \widehat{\Theta},$$ which are called \textit{dualities}. It is easy to see that this functor preserves $\omega$-categories and thus induces functors $$(\uvar)^S:\ocat\to \ocat.$$ In particular, we have the \textit{odd duality} $(\uvar)^{op}$, corresponding to the set of odd integers, the \textit{even duality} $(\uvar)^{co}$, corresponding to the subset of non-negative even integers, and the \textit{full duality} $(\uvar)^{\circ}$, corresponding to $\Nb^*$. Finally, we have the equivalences $$((\uvar)^{co})^{op}\sim (\uvar)^{\circ} \sim ((\uvar)^{op})^{co}.$$
\end{definition}

\begin{definition}
\label{defi:M}
Given $n\in \Nb\cup\{\omega\}$, we denote by $\Delta[\Theta_n]$ the strict category fitting in the cocartesian square:

\[\begin{tikzcd}
	{\Theta_n\times \{0\}} & {\Theta_n\times \Delta} \\
	{\{0\}} & {\Delta[\Theta_n]}
	\arrow[from=1-1, to=1-2]
	\arrow[from=1-1, to=2-1]
	\arrow[from=1-2, to=2-2]
	\arrow[from=2-1, to=2-2]
	\arrow["\lrcorner"{anchor=center, pos=0.125, rotate=180}, draw=none, from=2-2, to=1-1]
\end{tikzcd}\]

We then have a canonical morphism: 
$$\begin{array}{rrcl}
i:&\Delta[\Theta_n]&\to &\Theta_{n+1}\\
&([m],a)&\mapsto& [a,m]:=[\{a,a,\ldots,a\},m].
\end{array}
$$
inducing by left Kan extension a functor:
$$\begin{array}{rrcl}
[\uvar,\uvar]:&\widehat{\Delta[\Theta_n]}&\to &\widehat{\Theta_{n+1}}\\
\end{array}
$$

Eventually, we denote by $(\Mseg)_{n+1}$ the set of maps $\{[\Sp_a,\Sp_m]\to [a,m], [a,m]\in \Delta[\Theta_n]\}$ and $(\Msat)_{n+1}$ the set of maps $\{E_{eq}\to [0]\}\cup\{[\Sigma^kE_{eq},1]\to [\Db_k,1], k<n\}$, and we set 
$$\mbox{$\M_n$}:=(\Mseg)_n\cup (\Msat)_n.$$
\end{definition}

\begin{prop}
\label{prop:delta theta and delta} Let $C$ be a presentable category. The adjunction \[\begin{tikzcd} {\Lb i_!:\Fun(\Delta[\Theta_{n}]^{op},C)[(\M_{n+1}\times C)^{-1}]} & {\Fun(\Theta_{n+1}^{op},C)[(\W_{n+1}\times C)^{-1}]:i^*} \arrow[""{name=0, anchor=center, inner sep=0}, shift left=2, from=1-1, to=1-2] \arrow[""{name=1, anchor=center, inner sep=0}, shift left=2, from=1-2, to=1-1] \arrow["\dashv"{anchor=center, rotate=-90}, draw=none, from=0, to=1] \end{tikzcd}\] induced by left Kan extention along the functor $i:\Delta[\Theta_n]\to \Theta_{n+1}$, is an equivalence.
\end{prop}

\begin{proof}
By the Yoneda lemma, this adjunction is equivalent to 
\[\begin{tikzcd}
	{\Fun^{co}(C^{op},\widehat{\Delta[\Theta_n]}[(\M_{n+1})^{-1}])} & {\Fun^{co}(C^{op},\widehat{\Theta_{n+1}}[(\W_{n+1})^{-1}])}
	\arrow[""{name=0, anchor=center, inner sep=0}, shift left=2, from=1-1, to=1-2]
	\arrow[""{name=1, anchor=center, inner sep=0}, shift left=2, from=1-2, to=1-1]
	\arrow["\dashv"{anchor=center, rotate=-90}, draw=none, from=0, to=1]
\end{tikzcd}\]
where $\Fun^{co}$ is the category of limit-preserving functors. The result follows from the adjoint equivalence
\[\begin{tikzcd}
	{\widehat{\Delta[\Theta_n]}[(\M_{n+1})^{-1}]} & {\widehat{\Theta_{n+1}}[(\W_{n+1})^{-1}]}
	\arrow[""{name=0, anchor=center, inner sep=0}, shift left=2, from=1-1, to=1-2]
	\arrow[""{name=1, anchor=center, inner sep=0}, shift left=2, from=1-2, to=1-1]
	\arrow["\dashv"{anchor=center, rotate=-90}, draw=none, from=0, to=1]
\end{tikzcd}\]
given in construction 2.2.1.2 of \cite{loubaton2024categorical}.
\end{proof}

\begin{remark}
\label{rem:other description of ncat} The previous proposition, applied to the case $C:=\Hot$, implies that we have a canonical equivalence: $$\npcat\sim \widehat{\Delta[\Theta_n]}[(M_{n+1})^{-1}].$$ In particular, $\Delta[\Theta_n]$ is dense in $\npcat$.
\end{remark}

\begin{prop}
\label{Critere necessaire} Let $C$ be a cocomplete category and $$F:C\times \cat\to C$$ a colimit-preserving functor. Then there exists a unique colimit-preserving functor $$F_\omega:C\times \ocat\to C$$ characterized by the fact that for all $c,a,n$, $F_\omega(c,[a,n])$ fits in the pushout: 
\[\begin{tikzcd}
	{\coprod_{k\leq n}F(F_\omega(c,a),\{k\})} & {F(F_\omega(c,a),[n])} \\
	{\coprod_{k\leq n}F(F_\omega(c,[0]),\{k\})} & {F_\omega(c,[a,n])}
	\arrow[""{name=0, anchor=center, inner sep=0}, from=1-1, to=1-2]
	\arrow[from=1-1, to=2-1]
	\arrow[from=1-2, to=2-2]
	\arrow[from=2-1, to=2-2]
	\arrow["\lrcorner"{anchor=center, pos=0.125, rotate=180}, draw=none, from=2-2, to=0]
\end{tikzcd}\]
\end{prop}

\begin{proof}
We first show by induction on $n$ that there exists a unique colimit-preserving functor 
$$F_{n+1}:C\times \npcat\to C$$
characterized by the formula:

\[\begin{tikzcd}
	{\coprod_{k\leq n}F(F_n(c,a),\{k\})} & {F(F_n(c,a),[n])} \\
	{\coprod_{k\leq n}F(F_n(c,[0]),\{k\})} & {F_{n+1}(c,[a,n])}
	\arrow[""{name=0, anchor=center, inner sep=0}, from=1-1, to=1-2]
	\arrow[from=1-1, to=2-1]
	\arrow[from=1-2, to=2-2]
	\arrow[from=2-1, to=2-2]
	\arrow["\lrcorner"{anchor=center, pos=0.125, rotate=180}, draw=none, from=2-2, to=0]
\end{tikzcd}\]
Suppose the functor $F_n$ is constructed. Let $F_{n+1}:\Delta[\Theta_n]\to \Fun(C,C)$ be defined by the previous formula. We also denote by $F_{n+1}:\widehat{\Delta[\Theta_n]}\to \Fun(C,C)$ the functor induced by left Kan extension. By construction, the functor $F_{n+1}$ sends $[\W_n,\W_1]$ onto equivalences, and remark \ref{rem:other description of ncat} then implies that it uniquely factors as a colimit-preserving functor $F_{n+1}:\npcat\to \Fun(C,C)$.

Suppose now that all the functors $F_n$ are constructed. This induces a family of functors $\Theta_n\to\Fun(C,C)$ that is natural in $n$. Since $\Theta_\omega$ is the union of all the $\Theta_n$, this induces a functor $\Theta\to \Fun(C,C)$ and, by left Kan extension, a functor 
$F_\omega:\widehat{\Theta}\to \Fun(C,C)$. As every morphism of $\W$ lives in $\ncat$ for some $n$, the functor $F_\omega$ sends $\W$ to equivalences, and thus uniquely factors as a colimit-preserving functor $F_{\omega}:\ocat\to \Fun(C,C)$.
\end{proof}

\begin{prop}
\label{prop:ncat cartesin closed}
Let $n\in \Nb\cup\{\omega\}$.
The category $\ncat$ is cartesian closed.
\end{prop}

\begin{proof}
As $\ncat\to \ocat$ preserves colimits and the cartesian product, we can reduce to the case $n=\omega$. The result is then \cite[proposition 2.2.1.54]{loubaton2024categorical}.
\end{proof}

\subsection{$n$-Fully faithful and $n$-surjective morphisms}

\begin{definition}
A morphism $f:C\to D$ is \textit{$0$-fully faithful} if it is an equivalence. A morphism $f:C\to D$ is \textit{$(n+1)$-fully faithful} if, for any pair of objects $a,b\in C$, the induced morphism $\hom_C(a,b)\to \hom_D(fa,fb)$ is $n$-fully faithful.
\end{definition}

\begin{notation}
A $1$-fully faithful morphism will simply be called a fully faithful morphism.
\end{notation}

\begin{lemma}
A morphism $f:C\to D$ is $n$-fully faithful if it has the unique left lifting property against $\partial\Db_n\to \Db_m$ for any $m\geq n$.
\end{lemma}

\begin{proof}
This characterization is trivial when $n=0$. Suppose this characterization is true at stage $n$. By adjunction, a functor $f:C\to D$ has the unique left lifting property against $\partial\Db_n\to \Db_m$ for any $m> n$ if and only if for any $a,b\in C$, $\hom_C(a,b)\to\hom_D(f(a),f(b))$ has the unique left lifting property against $\partial\Db_{n-1}\to \Db_{m-1}$. By the induction hypothesis, this is equivalent to saying that $\hom_C(a,b)\to\hom_D(f(a),f(b))$ is an $n$-fully faithful functor for any $a,b$, which concludes the proof.
\end{proof}

\begin{prop}
\label{prop:ff 2} $n$-Fully faithful morphisms are stable under limits.
\end{prop}

\begin{proof}
This follows from the fact that fully faithful morphisms are characterized by a unique right lifting property.
\end{proof}

\begin{lemma}
\label{lemma:ff 2}
Let $p:C\to D$ be a fully faithful morphism. The induced morphism $C_0\to D_0$ is a monomorphism.
\end{lemma}

\begin{proof}
This is lemma 2.2.1.77 of \cite{loubaton2024categorical}.
\end{proof}

\begin{definition}
A morphism is \textit{$n$-surjective} if it has the right lifting property against $(n+1)$-fully faithful functors. The class of $n$-surjective morphisms then corresponds to the closure by colimit and composition of the set of morphisms 
$$\{\partial \Db_{n+1}\to \Db_m~,m>n\}$$
and is then stable under colimits, transfinite composition, pushouts, left cancellation, and retracts. Furthermore, every map uniquely factors as an $n$-surjection followed by a $(n+1)$-fully faithful morphism.
\end{definition}
\begin{remark}
By construction, a morphism that is $n$-surjective and $(n+1)$-fully faithful is an equivalence.
\end{remark}

\begin{definition}
We will say that a morphism $f:C\to D$ is \textit{surjective on objects} if $f:C_0\to D_0$ is $0$-connected.
\end{definition}

\begin{remark}
A morphism $f:C\to D$ is {surjective on objects} if and only if any diagram of shape 
\[\begin{tikzcd}
	\emptyset & C \\
	{[0]} & D
	\arrow[from=1-1, to=1-2]
	\arrow[from=1-1, to=2-1]
	\arrow[from=1-2, to=2-2]
	\arrow[dashed, from=2-1, to=1-2]
	\arrow[from=2-1, to=2-2]
\end{tikzcd}\]
admits a (a priori non-unique) lift.
\end{remark}

\begin{prop}
\label{prop: characterization of 1 surjective2} A morphism $\phi:C\to D$ is a $0$-surjection if and only if it is $0$-surjective on objects.
\end{prop}

\begin{proof}
Suppose first that $f:C\to D$ is surjective on objects. We can factor $f$ as a $0$-surjective morphism $p:C\to C'$ followed by a fully faithful functor $g:C'\to D$. The morphism $g:C'\to D$ is then also $0$-surjective on objects, and by  \cite[proposition 2.2.1.78]{loubaton2024categorical}, it is an equivalence. As a corollary, $f\sim p$ and $f$ is then $0$-surjective.

To show the converse, note that the functor $$\widehat{\Theta}\xrightarrow{\ev_{[0]}}\Hot\xrightarrow{\pi_0}\Set$$ sends every morphism of $\W$ and morphism of shape $\partial[1]\to \Db_n$ for $n>0$ to a surjection. Eventually, as this morphism preserves colimits and as surjections in $\Set$ are closed under colimits, this implies that every $0$-surjection is sent to a surjection in $\Set$. This then implies that $0$-surjections are surjective on objects.
\end{proof}

\begin{remark}
A morphism $f:X\to Y$ between homotopy types is $n$-fully faithful if and only if it is $(n-2)$-truncated and is $n$-surjective if and only if it is $n$-connective.
\end{remark}

\begin{remark}
We will show in corollary \ref{cor:characterization of n surjective} that $n$-surjective morphisms are characterized by the fact that they are $0$-surjective on objects, and that they are locally $(n-1)$-surjective. In corollary \ref{cor:characterization of n surjective pullback}, we will show that $n$-surjections are closed under pullback, and we will give another characterization of $n$-surjective and $n$-fully faithful morphisms in corollary \ref{cor:yet another characterization of fully faithful functor}.
\end{remark}

\begin{remark}
One can define \textit{$\omega$-surjective morphisms} as those morphisms that are $n$-surjective for all $n$, and one may then wonder if there exist $\omega$-surjective morphisms that are not equivalences. Curiously, such morphisms do exist. For example, if we denote by $E_{eq}^{coind}$ the coherent walking (coinductive) $\omega$-equivalence\footnote{A coinductive $\omega$-equivalence is a morphism $f:x\to y$ which is invertible up to a $2$-cell that is itself invertible up to a higher cell and so on.} defined in \cite{hadzihasanovic2024model}, then the morphism $E_{eq}^{coind}\to [0]$ has the (non-unique) left lifting property against the map $\partial\Db_n\to \Db_n$ for any $n$, and by corollary \ref{cor:characterization of n surjective}, this morphism is $\omega$-surjective.
\end{remark}

\begin{conj}
The morphism 
$$\ncat\to \ncat[\Surj_{\omega}^{-1}]$$
is an equivalence for any $n\in \Nb$. In other words, every $\omega$-surjective morphism between $n$-categories is an equivalence, and thus $\ncat$ is \emph{categorically hypercomplete}.

The morphism 
$$\ocat^{coind}\to \ncat[\Surj_{\omega}]$$
is an equivalence, where $\ocat^{coind}$ denotes the subcategory of $\ocat$ whose objects are $\omega$-categories that are local with respect to $\Sigma_nE_{eq}^{coind}\to \Db_n$ for any integer $n\in \Nb$\footnote{This category should not correspond to the limit of the $\ncat$ along the intelligent truncation functors. See \cite{Henry_an_inductive_model_structure_for_infini_categories} for a more detailed discussion.}, i.e., categories whose coinductive $\omega$-equivalences are trivial. In other words, the \emph{categorical hypercompletion} of $\ocat$ consists of $\ocat^{coind}$.
\end{conj}

\subsection{Special colimits}

In this section, we fix a small category $A$ and a set $S$ of maps with representable codomain in $\widehat{A}$. We set $\H := \widehat{A}[S^{-1}]$, and we then have a localization
\[\begin{tikzcd}
	{\Fb:\widehat{A}} & {\H:\iota}
	\arrow[""{name=0, anchor=center, inner sep=0}, shift left=2, from=1-1, to=1-2]
	\arrow[""{name=1, anchor=center, inner sep=0}, shift left=2, from=1-2, to=1-1]
	\arrow["\dashv"{anchor=center, rotate=-90}, draw=none, from=0, to=1]
\end{tikzcd}\]
where the right adjoint is fully faithful, the unit is in $\widehat{S}$, and the counit is the identity.

\begin{definition}
\label{defi: special colimits}
A functor $F:I\to \H$ has a \textit{special colimit} if the canonical morphism 
\begin{equation}
\label{eq:special colimit}
\colim_{i:I}\iota F(i)\to \iota(\colim_{i:I}F(i))
\end{equation}
is an equivalence in $\widehat{A}$.

Similarly, we say that a functor $\psi: I\to \Arr(\H)$ has a \textit{special colimit} if the canonical morphism 
$$\colim_{i:I}\iota \psi(i)\to \iota(\colim_{i:I}\psi(i))$$
is an equivalence in the arrow category of $\widehat{A}$.
\end{definition}

\begin{remark}
A functor $F:I\to \H$ has a special colimit if and only if $\colim_IF$, computed in $\widehat{A}$, already belongs to $\H$.
\end{remark}

\begin{example}
\label{exemple:every iocategory is a special colimit} Let $C$ be an object of $\H$. The canonical diagram $A_{/C} \to \H$ has a special colimit, given by $C$.
\end{example}

\begin{remark}
Let $F,G:I\to \H$ be two functors, and let $\psi:F\to G$ be a cartesian natural transformation admitting a special colimit. Since $\widehat{A}$ is a topos, this implies that for any object $i$ of $I$, the canonical square
\[\begin{tikzcd}
	{F(i)} & {\colim_IF} \\
	{G(i)} & {\colim_IG}
	\arrow[from=1-1, to=1-2]
	\arrow["{\psi(i)}"', from=1-1, to=2-1]
	\arrow["\lrcorner"{anchor=center, pos=0.125}, draw=none, from=1-1, to=2-2]
	\arrow["{\colim_I\psi}", from=1-2, to=2-2]
	\arrow[from=2-1, to=2-2]
\end{tikzcd}\]
is cartesian.
\end{remark}

\begin{prop}
\label{prop:special colimit2}
Let $G:I\to \H$ be a functor admitting a special colimit, and let $X\to \colim_GI$ be any morphism. Then $X$ is the special colimit of the functor $$\begin{array}{rrcl} F:&I&\to &\H\\ &i&\mapsto& G(i)\times_{\colim_GI}X \end{array}$$
\end{prop}

\begin{proof}
As $\widehat{A}$ is a topos, and since the colimit of $G$ is special, we have an equivalence $$\colim_I \iota F \sim X$$ in $\widehat{A}$, which concludes the proof.
\end{proof}

\begin{prop}
\label{prop:special colimit} Let $F,G:I\to \H$ be two functors, and let $\psi:F\to G$ be a natural transformation. If $\psi$ is cartesian and $G$ has a special colimit, then $\psi$ and $F$ have special colimits.
\end{prop}

\begin{proof}
We have to show that $F$ has a special colimit, which will directly imply that $\psi$ also has one. The morphism \eqref{eq:special colimit} is always in $\widehat{S}$. To conclude, one then has to show that $\colim_{i:I}\iota \psi(i)$ is $S$-local. To this end, it is enough to demonstrate that the canonical morphism $$\colim_{i:I}\iota \psi(i): \colim_{i:I}\iota F(i)\to \colim_{i:I}\iota G(i)$$ has the unique right lifting property against $S$. We then consider a square \begin{equation} \label{eq:proof special colimit} \begin{tikzcd} a & {\colim_{i:I}\iota F(i)} \\ b & { \colim_{i:I}\iota G(i)} \arrow[from=1-1, to=1-2] \arrow["f"', from=1-1, to=2-1] \arrow["{\colim_{i:I}\iota \psi(i)}", from=1-2, to=2-2] \arrow[from=2-1, to=2-2] \end{tikzcd} \end{equation} where $f\in S$. Since the domain of $f$ is representable, there always exists $j:I$ such that the bottom horizontal morphism factors through $G(j)$. As $\psi$ is cartesian and $\widehat{A}$ is a topos, the square \eqref{eq:proof special colimit} factors into two squares, where the right one is cartesian. \[\begin{tikzcd} a & {F(i)} & {\colim_{i:I}\iota F(i)} \\ b & {G(i)} & { \colim_{i:I}\iota G(i)} \arrow[from=1-1, to=1-2] \arrow["f"', from=1-1, to=2-1] \arrow[from=1-2, to=1-3] \arrow["{\psi(i)}", from=1-2, to=2-2] \arrow["\lrcorner"{anchor=center, pos=0.125}, draw=none, from=1-2, to=2-3] \arrow["{\colim_{i:I}\iota \psi(i)}", from=1-3, to=2-3] \arrow[from=2-1, to=2-2] \arrow[from=2-2, to=2-3] \end{tikzcd}\] Lifts in the square \eqref{eq:proof special colimit} are then equivalent to lifts in the left square, which exist and are unique since $F(i)\to G(i)$ has the unique right lifting property against $S$.
\end{proof}

\subsection{Gray operations}

\begin{construction}
The \textit{Gray tensor product}, denoted by $\otimes:\ocat\times \icat\to \ocat$, is defined in \cite[definition 2.3.1.6]{loubaton2024categorical}. By proposition \ref{Critere necessaire}, it canonically extends to a colimit preserving functor $$\otimes:\ocat\times \ocat\to \ocat$$ such that for any $a\in \Theta$, we have a natural family of Cartesian squares
\[\begin{tikzcd}
	{\coprod_{k\leq n }(C\otimes a)\otimes\{k\}} & {(C\otimes a)\otimes[n]} \\
	{\coprod_{k\leq n }C} & {C\otimes[a,n]}
	\arrow[from=1-1, to=1-2]
	\arrow[from=1-1, to=2-1]
	\arrow[from=1-2, to=2-2]
	\arrow[from=2-1, to=2-2]
	\arrow["\lrcorner"{anchor=center, pos=0.125, rotate=180}, draw=none, from=2-2, to=1-1]
\end{tikzcd}\]

The definition of the Gray tensor product in \cite[definition 2.3.1.6]{loubaton2024categorical} and corollary 2.3.3.24 of \cite{loubaton2024categorical} induce equivalences:
$$C \sim C \otimes [0]~~~~C \sim [0] \otimes C.$$
\end{construction}

\begin{definition}
Let $A$, $B$ be two $\omega$-categories. We denote by $A^{B}$ the value on $A$ of the right adjoint of the functor $C\mapsto C\otimes B$, and we then have $A^{[0]}\sim A$.

The \textit{directed pullback} of a span $A\to B\leftarrow C$, denoted by $A\dirtimes{B}C$, is the limit of the diagram :
\[\begin{tikzcd}
	A & {B^{\{0\}}} & {B^{[1]}} & {B^{\{1\}}} & C
	\arrow[from=1-1, to=1-2]
	\arrow[from=1-3, to=1-2]
	\arrow[from=1-3, to=1-4]
	\arrow[from=1-5, to=1-4]
\end{tikzcd}\]
\end{definition}

\begin{definition}
A \textit{$\zo$-category}, or a \textit{strict $\omega$-category}, is a $\omega$-category $C$ such that for any $n$, the homotopy type $C_n$ is $0$-truncated. We denote by $\zocat$ the category of strict $\omega$-categories.

In construction 2.2.1.7 of \cite{loubaton2024categorical}, we define an adjunction
\[\begin{tikzcd}
	{\pi_0:\ocat} & {\zocat:\N}
	\arrow[""{name=0, anchor=center, inner sep=0}, shift left=2, from=1-1, to=1-2]
	\arrow[""{name=1, anchor=center, inner sep=0}, shift left=2, from=1-2, to=1-1]
	\arrow["\dashv"{anchor=center, rotate=-90}, draw=none, from=0, to=1]
\end{tikzcd}\]
where the right adjoint is fully faithful.
\end{definition}

\begin{theorem}[Steiner, Ara-Maltsiniotis] There exists a monoidal structure on $\zocat$, denoted by $\otimes$ and called the Gray tensor product. Its unit is $[0]$.
\end{theorem}

\begin{proof}
This is \cite[theorem A.15]{Ara_Maltsiniotis_joint_et_tranche}
\end{proof}

\begin{prop}
\label{prop:pio Gray} The functor $\pi_0:\ocat\to \zocat$ commutes with the Gray product.
\end{prop}

\begin{proof}
This is proposition 2.3.3.1 of \cite{loubaton2024categorical}.
\end{proof}

In \cite{loubaton2024categorical}, the following theorem is proven:

\begin{theorem}
\label{theo:strictness of Gray1} If $C$ is a strict $\omega$-category and $n$ an integer, then $C \otimes [n]$ is also a strict $\omega$-category.
\end{theorem}

\begin{proof}
This is \cite[theorem 2.3.3.22]{loubaton2024categorical}.
\end{proof}

We now want to show that the Gray product with any globular sum preserves strict $\omega$-categories. This will be established in theorem \ref{theo:strictness of Gray}.

\begin{prop}
\label{prop:equation former}
Let $C$ be an $\omega$-category. $[C,1]\otimes[1]$ is the colimit of the diagram:
$$[C,1]\vee[1]\leftarrow [C\otimes\{1\},1]\to [C\otimes[1],1]\leftarrow[C\otimes\{0\},1]\to [1]\vee[C,1].$$
\end{prop}

\begin{proof}
This is proposition 2.3.1.10 of \cite{loubaton2024categorical}.
\end{proof}

\begin{lemma}
\label{lemma:equation 1}
Let $C$ be an $\omega$-category and $n$ an integer. The $\omega$-category $[C,1]\otimes\Db_{n+1}$ is the colimit of the diagram:
\[\begin{tikzcd}
	{[C,1]\vee\Db_{n+1}} & {[C\otimes\{1\}\otimes\Db_n,1]} & {[(C\otimes[1])\otimes\Db_n,1]} & {[C\otimes\{0\}\otimes\Db_n,1]} & {\Db_{n+1}\vee[C,1]}
	\arrow[from=1-2, to=1-1]
	\arrow[from=1-3, to=1-2]
	\arrow[from=1-4, to=1-3]
	\arrow[from=1-4, to=1-5]
\end{tikzcd}\]
\end{lemma}

\begin{proof}
The case $n=0$ is proposition \ref{prop:equation former}. Suppose the result is proven at the stage $n$. By definition, $[C,1]\otimes \Db_{n+1}$ is the colimit of the diagram :
\[\begin{tikzcd}
	{([C,1]\otimes [0])\otimes\partial[1]} & {([C,1]\otimes \Db_n)\otimes\partial[1]} & {([C,1]\otimes\Db_{n})\otimes[1]}
	\arrow[from=1-2, to=1-1]
	\arrow[from=1-2, to=1-3]
\end{tikzcd}\]
Using the induction hypothesis and proposition \ref{prop:equation former}, we deduce that $[C,1]\otimes \Db_{n+1}$ is the colimit of the diagram:
\[\begin{tikzcd}
	{[C,1]\vee[\{1\},1]} & {[C\otimes\{1\}\otimes\{1\},1]} & {[C\otimes[1]\otimes\{1\},1]} & {[C\otimes\{0\}\otimes\{1\},1]} & {[\{1\},1]\vee[C,1]} \\
	{[C,1]\vee\Db_{n+1}} & {[C\otimes\{1\}\otimes\Db_n,1]} & {[(C\otimes[1])\otimes\Db_n,1]} & {[C\otimes\{0\}\otimes\Db_n,1]} & {\Db_{n+1}\vee[C,1]} \\
	{[C,1]\vee[\{0\},1]} & {[C\otimes\{1\}\otimes\{0\},1]} & {[C\otimes[1]\otimes\{0\},1]} & {[C\otimes\{0\}\otimes\{0\},1]} & {[\{0\},1]\vee[C,1]}
	\arrow[from=1-1, to=2-1]
	\arrow[from=1-2, to=1-1]
	\arrow[from=1-2, to=2-2]
	\arrow[from=1-3, to=1-2]
	\arrow[from=1-3, to=1-4]
	\arrow[from=1-3, to=2-3]
	\arrow[from=1-4, to=1-5]
	\arrow[from=1-4, to=2-4]
	\arrow[from=1-5, to=2-5]
	\arrow[from=2-2, to=2-1]
	\arrow[from=2-2, to=2-3]
	\arrow[from=2-4, to=2-3]
	\arrow[from=2-4, to=2-5]
	\arrow[from=3-1, to=2-1]
	\arrow[from=3-2, to=2-2]
	\arrow[from=3-2, to=3-1]
	\arrow[from=3-2, to=3-3]
	\arrow[from=3-3, to=2-3]
	\arrow[from=3-4, to=2-4]
	\arrow[from=3-4, to=3-3]
	\arrow[from=3-4, to=3-5]
	\arrow[from=3-5, to=2-5]
\end{tikzcd}\]
As the diagram of the statement is a final subdiagram, this concludes the proof.
\end{proof}
\begin{lemma}
\label{lemma:weak assoc prequel}
We suppose that $\uvar\otimes a$ preserves strict $\omega$-categories for any globular sum $a$ of dimension less than or equal to $n$.

Then, there exists an equivalence
$$C\otimes (D\otimes E)\to  (C\otimes D)\otimes E$$
natural in $C:\ocat$, $D:\ncat$, and $E:1\tcat$.
\end{lemma}

\begin{proof}
As all these functors preserve colimits, it is sufficient to construct this equivalence for $C$ and $D$ being a globular sum $a$ and $b$. As $\pi_0$ commutes with the Gray tensor product and by hypothesis and theorem \ref{theo:strictness of Gray1}, we have a canonical comparison
$$a\otimes (b\otimes  [m])\to \pi_0(a\otimes b\otimes [m])\sim (a\otimes b)\otimes[m].$$
It remains to show that this is an equivalence, and for this, we can reduce to the case where $a$ and $b$ are  globes $\Db_{k+1}$ and $\Db_{l+1}$ with $l<n$.

Using proposition \ref{prop:equation former} and lemma \ref{lemma:equation 1}, we deduce that  $\Db_{k+1}\otimes (\Db_{l+1} \otimes[1])$ and $(\Db_{k+1}\otimes \Db_{l+1}) \otimes[1]$ are both the colimit of the diagram
\[\begin{tikzcd}
	{ [\Db_k\otimes [1] \otimes \Db_k\otimes\{1\},1]} && { [\Db_k\otimes [1] \otimes \Db_k\otimes\{0\},1]} \\
	{ [\Db_k\otimes [1]\otimes \Db_l,1]\vee[1]} & { [\Db_k\otimes [1] \otimes \Db_k\otimes[1],1]} & {[1]\vee [\Db_k\otimes [1]\otimes \Db_l,1]} \\
	{(\Db_{k+1}\otimes \Db_{l+1})\vee[1]} && {[1]\vee(\Db_{k+1}\otimes \Db_{l+1})}
	\arrow[from=1-1, to=2-1]
	\arrow[from=1-1, to=2-2]
	\arrow[from=1-3, to=2-2]
	\arrow[from=1-3, to=2-3]
	\arrow[from=2-1, to=3-1]
	\arrow[from=2-3, to=3-3]
\end{tikzcd}\]
where $(\Db_{k+1}\otimes \Db_{l+1})\vee[1]$ is obtained by gluing $\{1\}\otimes\{1\}$ onto $\{0\}$ and $[1]\vee(\Db_{k+1}\otimes \Db_{l+1})$ by gluing $\{1\}$ onto $\{0\}\times\{0\}$.
\end{proof}

\begin{lemma}
\label{lemma:retract the true} Let $C$ and $D$ be two $\omega$-categories. If $C$ is a retract of $D$ and if $D$ is strict, then $C$ is strict.
\end{lemma}

\begin{proof}
The assumptions imply that for all $n$, the homotopy type of $n$-cells of $C$ is a retract of a set, and thus is a set.
\end{proof}

\begin{lemma}
\label{lemme:preretract}
Let $n$ be a non-negative integer and $C$ be an $\omega$-category. There exists a morphism 
$$p_C^{n,1}:[C,1]\otimes[n]\to [C\otimes[n],1]$$ 
that admits a section $s^{n,m}_C$. Both of these morphisms are natural in $C$.
\end{lemma}

\begin{proof}
Up to extension by colimits, we can reduce the problem to the case where $C$ is a globular sum and therefore a $\zo$-category. In this case, the two objects $[C,1]\otimes[n]$ and $[C\otimes[1],1]$ are strict by theorem \ref{theo:strictness of Gray1} and by proposition 2.3.3.3 of \cite{loubaton2024categorical}.

We will use Steiner's theory, recalled in Section 1.2.1 of \cite{loubaton2024categorical}. Globular sums are in particular $\zo$-categories with an atomic and loop-free basis, and by \cite[theorem 5.11]{Steiner_omega_categories_and_chain_complexes} and \cite[theorem A.15]{Ara_Maltsiniotis_joint_et_tranche}, it is sufficient to construct the morphism
$$p^{n,1}_A:[A,1]\otimes[n]\to [A\otimes[n],1]$$
and its section, in the category of augmented directed complexes admitting a unitary and loop-free basis.

We then define the functor $p^{n,1}_A$ by the assignment:
$$\begin{array}{rcl}
p^{n,1}_A([a,1]\otimes e_i)&:=& [a\otimes e_{n-i-1},1],\\
p^{n,1}_A([a,1]\otimes\{k\})&:=& [a\otimes\{n-k\},1],\\
p^{n,1}_A(\{i\}\otimes [e_i])&:=& 0,\\
p^{n,1}_A(\{i\}\otimes \{k\})&:=& \{i\}.\\
\end{array}$$
Similarly, we define the functor $s^{n,1}_A: [A\otimes[n],1]\to [A,1]\otimes[n]$ by:
$$\begin{array}{rcll}
s^{n,1}_A([[a\otimes e_i],1]) &:=& [a,1]\otimes e_{n-i},\\
s^{n,1}_A(\{i\})&:=&\{i\}\otimes\{ni\},\\
s^{n,1}_A([[a\otimes \{k\}],1])& := &[a,1]\otimes \{n-k\},&\text{if } |a|>0,\\
s^{n,1}_A([[a\otimes \{k\}],1]) &:=& [a,1]\otimes \{n-k\} +\{1\}\otimes e_{k+1,n-1} + \{0\}\otimes e_{0,k+1},&\text{if } |a|=0.
\end{array}$$
where $e_i$ is the unique element of the basis of $[n]$ such that $\partial e_i:\{i+1\}-\{i\}$ and where $e_{i,j}:=\coprod_{i\leq l<j}e_l$. 
One can then verify that $s^{n,1}_A$ is indeed a section of $p^{n,1}_A$.
\end{proof}

\begin{lemma}
\label{lemme:retract}
Let $n,m$ be two non-negative integers and $C$ an $\omega$-category. There exists a morphism 
$$p_C^{n,m}:[C,m]\otimes[n\times m]\to [C\otimes[n],m]$$ 
that admits a section $s^{n,m}_C$.
\end{lemma}

\begin{proof}
For $k<m$, we denote by
$$p^{n,1,k}_C:[C,1]\otimes[n\times m]\to [C,1]\otimes[n]\to [C\otimes[n],1]$$
the composite where the first morphism is induced by the assignment
$$\begin{array}{rcll}
[n\times m]&\to& [n]\\
l&\mapsto &0&\text{if $l\leq n\times k$}\\
l&\mapsto &l-km&\text{if $km< l \leq n\times (k+1)$}\\
l&\mapsto &n&\text{if $ n\times (k+1)<l$}\\
\end{array}$$
and the second one is the morphism $p^{n,1}_a$ of lemma \ref{lemme:preretract}.
We denote by
$$s^{n,1,k}_C: [C\otimes[n],1] \to [C,1]\otimes[n] \to [C,1]\otimes[n\times m]$$
where the first morphism is the section $s^{n,1}_C$ of lemma \ref{lemme:preretract} and the second one is induced by $[\uvar +n\times k]:[n]\to [n \times m]$. By construction, $s^{n,1,k}_C$ is a section of $p^{n,1,k}_C$.

Eventually, the morphism $p^{n,m}_C:[C,m]\times[n \times m]\to [C\otimes[n],m]$ is defined as the horizontal colimit of the diagram
\[\begin{tikzcd}
	{[C,1]\otimes[n\times m]} & {\{1\}\otimes[n\times m]} & {[C,1]\otimes[n\times m]} & {...} \\
	{[C\otimes[n],1]} & {\{1\}} & {[C\otimes[n],1]} & {...}
	\arrow["{p^{n,1,0}_C}"', from=1-1, to=2-1]
	\arrow[from=1-2, to=1-1]
	\arrow[from=1-2, to=1-3]
	\arrow["{\{1\}\otimes \{n\}}"', from=1-2, to=2-2]
	\arrow["{p^{n,1,1}_C}"', from=1-3, to=2-3]
	\arrow[from=1-4, to=1-3]
	\arrow[from=1-4, to=2-4]
	\arrow[from=2-2, to=2-1]
	\arrow[from=2-2, to=2-3]
	\arrow[from=2-4, to=2-3]
\end{tikzcd}\]
as for section the morphism $s^{n,m}_C$ obtained as the horizontal colimit of the diagram
\[\begin{tikzcd}
	{[C\otimes[n],1]} & {\{1\}} & {[a\otimes[n],1]} & {...} \\
	{[C,1]\otimes[n\times m]} & {\{1\}\otimes[n\times m]} & {[C,1]\otimes[n\times m]} & {...}
	\arrow["{s^{n,1,0}_C}"', from=1-1, to=2-1]
	\arrow[from=1-2, to=1-1]
	\arrow[from=1-2, to=1-3]
	\arrow["{\{1\}\otimes \{n\}}"', from=1-2, to=2-2]
	\arrow["{s^{n,1,1}_C}"', from=1-3, to=2-3]
	\arrow[from=1-4, to=1-3]
	\arrow[from=1-4, to=2-4]
	\arrow[from=2-2, to=2-1]
	\arrow[from=2-2, to=2-3]
	\arrow[from=2-4, to=2-3]
\end{tikzcd}\]
The morphism $s^{n,m}_C$ is then a section of $p^{n,m}_C$, which concludes the proof.
\end{proof}
\begin{lemma}
\label{lemma:propriete de R}
Let $k$ be an integer, and let $R_k$ be the smallest set of strict $k$-categories such that:
\begin{enumerate}
\item $R_k$ contains $[0]$.
\item For all $(k-1)$-categories $C$ in  $R_k$ and all integers $n$, $C \otimes [n]$ is in $R_k$.
\item $R_k$ is stable under retracts.
\end{enumerate}
Then every globular sum of dimension lower than or equal to $k$ is in $R_k$.
\end{lemma}

\begin{proof}
If $k=0$, this is trivially true. We then suppose that $k>0$.
Let $S$ be the set of strict $(k-1)$-categories $C$ such that for all non-negative integers $m$, $[C,m]$ is in $R_k$. 

Let $C$ be a strict $(k-2)$-category in $S$. By construction, for all non-negative integers $m$, $[C,m]\otimes[n\times m]$ is in $R_k$, and by lemma \ref{lemme:retract}, so is $[C\otimes[n],m]$. Hence, $C\otimes[n]$ is in $S$. Moreover, $S$ is closed under retracts. Since $[[0],m]=[m]$ is in $R_k$ for all $m$, we conclude that $S$ contains $[0]$. By definition, this implies that $R_{k-1}$ is included in $S$, in other words, $[R_{k-1},m]\subset R_k$ for any $m$.

Let us now prove by induction on $n$ that $R_k$ contains $\Theta_k$. The case $k=0$ follows from the fact that, by construction, $\Theta_0 := [0]$ is in $R$. Now, suppose that $\Theta_k$ is contained in $R_k$. Let $[\textbf{a},m]$ be a globular sum of dimension less than or equal to $(k+1)$. We set $\vee \textbf{a} := a_0 \vee \dots \vee a_{m-1}$. The globular sum $\vee \textbf{a}$ is thus in $\Theta_k$ and belongs to $R_k$ by the induction hypothesis. By the results proven earlier, $[\vee \textbf{a},n]$ is therefore in $R_{k+1}$. Remark now that for any $i<m$, the canonical inclusion $a_i\to \vee \textbf{a}$ admits a retract. This implies that $[\textbf{a},m]$ is a retract of $[\vee \textbf{a},m]$ and thus belongs to $R_{k+1}$, which concludes the proof.
\end{proof}

\begin{theorem}
\label{theo:strictness of Gray} If $C$ is a strict $\omega$-category and $a$ a globular sum, then so is $C \otimes a$.
\end{theorem}

\begin{proof}
We will proceed by induction on the dimension of the globular sum $a$. If $|a|=0$, then $a=[0]$ and the result is trivially true. Suppose that the Gray tensor product with globular sums of dimension less than or equal to $n$ preserves strictness.

Let $S_{n+1}$ be the set of $(n+1)$-categories $D$ such that for any strict $\omega$-category $C$, $C\otimes D$ is strict. By definition, $S_{n+1}$ includes $[0]$. Since $C\otimes \uvar$ preserves retracts, lemma \ref{lemma:retract the true} implies that $S_{n+1}$ is stable under retracts. Suppose now we are given an $n$-category $D$ in $S_{n+1}$, and let $C$ be any strict $\omega$-category. The induction hypothesis combined with theorem \ref{theo:strictness of Gray1} implies that $(C\otimes D)\otimes [m]$ is strict, and by lemma \ref{lemma:weak assoc prequel}, so is $C\otimes (D\otimes [m])$. This then implies that $D\otimes [m]$ is in $S_{n+1}$.
The set $S_{n+1}$ then verifies all conditions of lemma  \ref{lemma:propriete de R}, and then includes $\Theta_{n+1}$.
\end{proof}

\begin{prop}
\label{prop:otimes and dualities} Let $C$, $D$ be two $\omega$-categories. We have a canonical equivalence: $$(C\otimes D)^{op}\sim D^{op}\otimes C^{op}~~~~~~~~ (C\otimes D)^{\circ}\sim C^{\circ}\otimes D^{\circ}~~~~~~~~ (C\otimes D)^{co}\sim D^{co}\otimes C^{co}$$
\end{prop}

\begin{proof}
As these functors preserve colimits, it is sufficient to construct these equivalences when $C$ and $D$ are globular sums. As the Gray tensor product of globular sums is strict by theorem \ref{theo:strictness of Gray}, we can apply \cite[proposition A.20]{Ara_Maltsiniotis_joint_et_tranche}.
\end{proof}

\begin{prop}
\label{prop:otimes and an} 
There is a natural equivalence between $[C,n]$ and the colimit of the diagram: \[\begin{tikzcd} {\coprod_{k\leq n}\{k\}} & {C\otimes(\coprod_{k\leq n}\{k\})} & {C\otimes[n]} \arrow[from=1-2, to=1-1] \arrow[from=1-2, to=1-3] \end{tikzcd}\] In particular, we have a natural family of cocartesian squares: \[\begin{tikzcd} {C\otimes \partial[1]} & {C\otimes[1]} \\ {\partial[1]} & {[C,1]} \arrow[from=1-1, to=1-2] \arrow[from=1-1, to=2-1] \arrow[from=1-2, to=2-2] \arrow[from=2-1, to=2-2] \end{tikzcd}\]
\end{prop}

\begin{proof}
This is Corollary 2.3.3.24 of \cite{loubaton2024categorical}.
\end{proof}

\begin{prop}
\label{prop:an other description of the suspension} Let $C$ be an $\omega$-category. There is an equivalence between $[C,n]$ and the colimit of the diagram: \[\begin{tikzcd} {\coprod_{k}\{k\}\otimes [0]} & {\coprod_{k}\{k\}\otimes C^{\circ}} & {[n]\otimes C^{\circ}} \arrow[from=1-2, to=1-1] \arrow[from=1-2, to=1-3] \end{tikzcd}\]
\end{prop}

\begin{proof}
Remark that we have an equivalence $$[C,n]\sim [C^{op},n]^{co}.$$ However, by proposition \ref{prop:otimes and an}, $[C^{op},n]$ is the colimit of the diagram: \[\begin{tikzcd} {\coprod_{k\leq n}[0]\otimes \{k\}} & {\coprod_{k\leq n}C^{op}\otimes \{k\}} & {C^{op}\otimes[n]} \arrow[from=1-2, to=1-1] \arrow[from=1-2, to=1-3] \end{tikzcd}\] Applying the even duality to the previous diagram and using proposition \ref{prop:otimes and dualities}, this concludes the proof.
\end{proof}

\begin{theorem}
\label{theo:monoidal} The Gray product extends to a monoidal structure on $\ocat$, and $\pi_0:\ocat\to \zocat$ extends to a strong monoidal functor.
\end{theorem}

\begin{proof}
We denote by $\Theta_{\otimes}$ the smallest subcategory of $\zocat$ that includes the globular sums and is stable under the Gray tensor product. Theorem \ref{theo:strictness of Gray} then implies that the inclusion $\Theta_{\otimes}\to \ocat$ commutes with the Gray tensor product. By Day convolution (proposition 2.14 of \cite{glasman2013day}), the Gray tensor product induces a monoidal structure on $\widehat{\Theta_\otimes}$. Since $\Theta_{\otimes}$ includes $\Theta$, it is a dense subcategory of $\ocat$, and the Gray tensor product on $\widehat{\Theta_\otimes}$ then restricts to a monoidal structure on $\ocat$ by proposition 2.2.1.9 of \cite{LurieHigherAlgebra}.
\end{proof}

\begin{remark}
There are a priori two more monoidal structures on $\ocat$ deserving the name of Gray tensor product. The first is obtained from the fact that $\ocat$ is presented by the model structure of Verity on complicial sets according to theorem 3.3.2.5 of \cite{Loubaton_complicial_sets_as_a_model_of_infini_n_categories}. We can then transfer the Gray tensor product on complicial sets constructed by Verity (\cite{Verity_complicial_set_characterising_the_simplicial_nerve}) to $\ocat$.

Results of Verity imply that $\pi_0:\ocat\to \zocat$ is strong monoidal. The theorem \ref{theo:strictness of Gray} then produces a comparison between the Gray tensor product coming from complicial sets and the one considered here. To verify that these two structures are the same, we can then restrict to the case where the tensor product of a globular sum with an element of $\Delta$ and refer to theorem 4.3.3.26 of \cite{Loubaton_these}. 

The second monoidal structure is constructed by Campion in \cite{campion2023Gray}. The functor $\pi_0:\ocat\to \zocat$ is strong monoidal, and Campion also shows that the tensor product of two globular sums is strict. It therefore coincides with that of theorem \ref{theo:monoidal}. Moreover, one could have deduced the theorem \ref{theo:monoidal} from Campion's results.
\end{remark}

In proposition 2.3.1.10 of \cite{loubaton2024categorical}, it is shown that $[C,1]\otimes[1]$ is the colimit of the following diagram $$[C,1]\vee[1]\leftarrow [C\otimes\{1\},1]\to [C\otimes[1],1]\leftarrow[C\otimes\{0\},1]\to [1]\vee[C,1]$$ We now want to construct a similar decomposition for $[C,1]\otimes[B,1]$. This will be obtained in proposition \ref{prop:decomposition}. Throughout the end of this section, we will use Steiner theory as presented in section 1.2 of \cite{loubaton2024categorical}.

\begin{definition}
Let $K,L$ be two augmented directed complexes. There are two canonical morphisms $$\triangledown:[K\otimes L,1]\to [K,1]\vee[L,1]~~~~~~~ \triangledown:[K\otimes L,1]\to [L,1]\vee[K,1]$$ that are the unique ones fulfilling $$\triangledown(\{0\}):= \{0\},~~~\triangledown(\{1\}):= \{2\},~~~ \triangledown([x\otimes y,1]):=\left\{ \begin{array}{ll} [x,1]+[y,1] & \text{if $|x|=0$ or $|y|=0$,} \\ 0 & \text{otherwise.} \end{array}\right.$$
\end{definition}

\begin{prop}
\label{prop:appendice formula for cda}
Let $K,L$ be two augmented directed complexes. There is a natural transformation between the colimit of the following diagram: $$
\begin{tikzcd}
	{[K,1]\vee [L,1]} & {[K\otimes\{1\}\otimes L,1]} & {[K\otimes [1]\otimes L,1]} & {[K\otimes\{0\}\otimes L,1]} & {[L,1]\vee [K,1]}
	\arrow[from=1-2, to=1-1]
	\arrow[from=1-2, to=1-3]
	\arrow[from=1-4, to=1-3]
	\arrow[from=1-4, to=1-5]
\end{tikzcd}$$ and $[K,1]\otimes [L,1]$.
\end{prop}

\begin{proof}
The cone is induced by 
$$
\begin{array}{r}
 [K,1]\vee[L,1]\sim [K,1]\otimes\{0\}\coprod_{[0]}\{1\}\otimes[L,1]\to [K,1]\otimes[L,1],\\
\text{(resp. } [L,1]\vee[K,1]\sim \{0\}\otimes[L,1]\coprod_{[0]} [K,1]\otimes\{1\}\to [K,1]\otimes[L,1]\text{)}
\end{array}$$

and by the morphism 
$$f:[K\otimes [1]\otimes L,1]\to [K,1]\otimes [L,1]$$
defined by the formula 
$$\begin{array}{rcl}
f([x\otimes [1]\otimes y,1])&:=&[x,1]\otimes[y,1],\\
f([x\otimes \{0\}\otimes y,1])&:=&\left\lbrace
\begin{array}{ll}
[x,1]\otimes\{1\}+\{0\}\otimes [y,1] & \text{if $|x|=0$ or $|y|=0$,} \\
0 & \text{otherwise.}
\end{array} \right. \\
f([x\otimes \{1\}\otimes y,1])&:=&\left\lbrace
\begin{array}{ll}
[x,1]\otimes\{0\}+\{1\}\otimes [y,1] & \text{if $|x|=0$ or $|y|=0$,} \\
0 & \text{otherwise.}
\end{array} \right.
\end{array}
$$
We leave it to the reader to check the compatibilities of these three morphisms.
\end{proof}

\begin{prop}
\label{prop:decomposition}
Let $C$ and $D$ be two $\omega$-categories. There is a natural equivalence between $[C,1]\otimes[D,1]$ and the colimit of the diagram $$[C,1]\vee[D,1]\leftarrow [C\otimes\{1\}\otimes D,1]\to [C\otimes[1]\otimes D,1]\leftarrow[C\otimes\{0\}\otimes D,1]\to [D,1]\vee[C,1].$$
\end{prop}

\begin{proof}
Since all the operations involved preserve colimits, it suffices to prove the result when $a$ and $b$ are globular sums. In this case, theorem \ref{theo:strictness of Gray} implies that all objects involved are strict $\omega$-categories. The comparison between the colimit of the diagram and $[C,1]\otimes[D,1]$ is induced by proposition \ref{prop:appendice formula for cda}. As all these morphisms commute with colimits, we can reduce to the case where $D$ is a globe to show that this comparison is an equivalence, and then apply lemma \ref{lemma:equation 1}.
\end{proof}

\begin{prop}
\label{prop:otimes and surjection} Let $f:A\to B$ be an $n$-surjection. Then $f\otimes C:A\otimes C\to B\otimes C$ is $n$-surjective.
\end{prop}

\begin{proof}
As the Gray tensor product commutes with colimits, it is sufficient to demonstrate that for any integers $n<k$, the map $$\partial \Db_{n+1}\otimes C\to \Db_k\otimes C$$ is $n$-surjective for any $\omega$-category $C$. We will proceed by induction on $n$. The case $-1$ is straightforward with the convention $\partial \Db_{-1}:= \Db_{-1}:=\emptyset$. Suppose the case $(n-1)$ is proven. As the Gray tensor product commutes with colimits, we can reduce to the case where $C$ is a globe $\Db_m$. If $m=0$, this is straightforward. If $m>0$, by proposition \ref{prop:decomposition}, we have pushouts:
\[\begin{tikzcd}
	{\partial\Db_{n+1}\otimes(\{0\}\coprod\{1\})} & {\partial \Db_{n+1}\otimes \Db_m} \\
	{\Db_{k}\otimes(\{0\}\coprod\{1\})} & \bullet & {[\partial\Db_{n}\otimes [1]\otimes\Db_{m-1},1]} \\
	& {\Db_k\otimes \Db_m} & {[\Db_{n}\otimes [1]\otimes\Db_{m-1},1]}
	\arrow[""{name=0, anchor=center, inner sep=0}, from=1-1, to=1-2]
	\arrow[from=1-1, to=2-1]
	\arrow[from=1-2, to=2-2]
	\arrow[from=2-1, to=2-2]
	\arrow[from=2-2, to=3-2]
	\arrow[from=2-3, to=2-2]
	\arrow[from=2-3, to=3-3]
	\arrow[""{name=1, anchor=center, inner sep=0}, from=3-3, to=3-2]
	\arrow["\lrcorner"{anchor=center, pos=0.125, rotate=180}, draw=none, from=2-2, to=0]
	\arrow["\lrcorner"{anchor=center, pos=0.125}, draw=none, from=2-2, to=1]
\end{tikzcd}\]
and as the suspension sends $(n-1)$-surjections to $n$-surjections and by induction hypothesis, this demonstrates the case $n$.
\end{proof}

\subsection{Definition of marked $\omega$-categories}

\begin{definition}
A \textit{marked $\omega$-category} is a pair $(C,tC)$ where $C$ is a $\omega$-category and $tC$ is a subhomotopy type of $\coprod_{n>0}C_n$ whose elements are called the \textit{marked cells}, and such that
\begin{enumerate}
\item units are marked,
\item marked $n$-cells are stable under $k$-composition for any $k<n$.
\end{enumerate}
A \textit{morphism of marked $\omega$-categories} is a morphism between the underlying $\omega$-categories that preserves the marking. We denote by $\ocatm$ the category of marked $\omega$-categories.
\end{definition}

\begin{definition}
Given an $\omega$-category $C$, we denote by $C^\sharp:=(C,C_{>0})$ the marked $\omega$-category where every cell of positive dimension is marked, and $C^\flat:=(C,\Ib(C_{\geq 0}))$ the marked $\omega$-category where only identities are marked.

Given a marked $\omega$-category $C$, we denote by $C^{\natural}$ the underlying $\omega$-category.
\end{definition}

\begin{notation}
Let $C$ be a marked $\omega$-category. We will denote simply by $C^{\flat}$ the marked $\omega$-category $(C^\natural)^\flat$ and $C^\sharp$ the marked $\omega$-category $(C^\natural)^\sharp$.
\end{notation}

\begin{construction}
Let $C$ be a $\omega$-category and $M \subset \coprod_{n>0} C_n$ a subhomotopy type of cells. The \textit{closure by composition of $M$}, denoted by $\overline{M}$, is the smallest subhomotopy type of arrows such that $M \subset \overline{M}$ and such that $(C, \overline{M})$ is a marked $\omega$-category. For example, $C^\flat := (C, \overline{\emptyset})$.
A subhomotopy type of cells $M$ is \textit{closed by composition} if $M \sim \overline{M}$.
\end{construction}

\begin{example}
\label{exe:canonical exe of cloture by comp}
Let $F:I\to \ocat$ be a diagram. The closure by composition of the homotopy type of the cell $\cup_{i:I}\im F(i)_{>0}$ is $(\colim_IF)_{>0}$.
\end{example}

\begin{remark}
\label{remark:other description of cloture by composition}
Using the associativity of $k$-compositions (remark \ref{rem:consequence of associativity of composition}), we can see that a subhomotopy type of cells $M$ is closed by composition if and only if:
\begin{enumerate}
\item units are in $M$,
\item for any integer $n$, and any pair of composable $(n+1)$-cells $x,y$ in $M$, the composite $x\circ_n y$ is in $M$,
\item for any integer $m>k+1$, any pair of $k$-composable cells $x,y$, with $x$ a $(k+1)$-cell and $y$ an $m$-cell in $M$, the composite $x\circ_k y$ is in $M$,
\item for any integer $m>k+1$, any pair of $k$-composable cells $y,x$, with $y$ a $(k+1)$-cell and $x$ an $m$-cell in $M$, the composite $y\circ_k x$ is in $M$.
\end{enumerate}
\end{remark}

\begin{prop}
The category $\ocatm$ is cartesian closed.
\end{prop}

\begin{proof}
This is proposition 3.1.1.19 of \cite{loubaton2024categorical}.
\end{proof}

\begin{definition}
Given an integer $n$, we denote by $(\Db_n)_t$ the marked $\omega$-category $(\Db_n, \overline{\{e_n\}})$ where $e_n$ is the only non-trivial $n$-cell.
Eventually, we define the category \textit{$t\Theta$} as the (strict) category whose objects are marked globular sums of shape $(\Db_n)_t$ or $a^\flat$.
\end{definition}

\begin{remark}
The category $t\Theta$ is dense in $\ocatm$.
\end{remark}

\begin{remark}
The result of section 3.1.1 of \cite{loubaton2024categorical} implies that $\ocatm$ is a presentable category. If $(C_i,tC_i)$ is a diagram of marked $\omega$-categories indexed by a category $I$, then 
$$\lim_I(C_i,tC_i)\sim (\lim_IC_i,\lim_ItC_i)~~~~ \colim_I(C_i,tC_i)\sim (\colim_IC_i, \overline{\im(\colim_ItC_i)})$$
where $\im(\colim_IC_i)$ denotes the image of the functor $\colim_I tC_i\to \coprod_{n>0}(\colim_IC_i)_n$.
\end{remark}

\begin{prop}
Any subcategory of $\ocatm$ closed under colimits and containing $\Db_n^\flat$ and $(\Db_n)_t$ is $\ocatm$.
\end{prop}

\begin{proof}
The claim follows from the fact that $t\Theta$ is dense, and that its elements are colimits of objects of shape $\Db_n^\flat$ or $(\Db_n)_t$.
\end{proof}

\begin{definition}
The \textit{marked suspension} is the colimit-preserving functor $$[\uvar,1]:\ocatm\to \ocatm_{\bullet,\bullet}$$ sending $a^\flat$ to $[a,1]^\flat$ and $(\Db_n)_t$ to $([\Db_n,1])_t$. It admits a right adjoint: $$\begin{array}{lll} \ocatm_{\bullet,\bullet} &\to& \ocatm\\ (C,a,b) &\mapsto& \hom_C(a,b). \end{array}$$
\end{definition}

\begin{definition}
A \textit{marked trivialization over $(A,B)$} is a morphism in the smallest cocomplete class of morphisms that includes $\Ib_n:(\Db_{n+1})_t\to \Db_n^{\flat}$ for any \( n \). 
A \textit{trivialization}  is a morphism in the smallest cocomplete class of morphisms that includes $\Ib_n:\Db_{n+1}\to \Db_n$ for any \( n \). Trivializations are then stable under transfinite composition, pushouts, left cancellation, and retracts.
\end{definition}

\begin{prop}
\label{prop:trivialization is epi}
Trivializations are epimorphisms.
\end{prop}

\begin{proof}
This is \cite[proposition 2.2.1.50]{loubaton2024categorical}.
\end{proof}

\begin{definition}
Let $C$ be an $\omega$-category, and $S \subset \coprod_{n>0} C_n$ a sub homotopy type of cells. We define $C[S^{-1}]$ as the pushout:
\[\begin{tikzcd}
	{\coprod_{a\in S}\Db_{|a|}} & C \\
	{\coprod_{a\in S}\Db_{|a|-1}} & {C[S^{-1}]}
	\arrow[from=1-1, to=1-2]
	\arrow[from=1-1, to=2-1]
	\arrow[from=1-2, to=2-2]
	\arrow[from=2-1, to=2-2]
	\arrow["\lrcorner"{anchor=center, pos=0.125, rotate=180}, draw=none, from=2-2, to=1-1]
\end{tikzcd}\]
Proposition \ref{prop:trivialization is epi} implies that $C\to C[S^{-1}]$ is an epimorphism.
\end{definition}

\begin{lemma}
\label{lemma:loc and col} Let $C$ be an $\omega$-category and  $S\subset \coprod_{n>0}C_n$ a sub-homotopy type of cells. The canonical morphism $C[S^{-1}]\to C[{\overline{S}}^{-1}]$ is an equivalence.
\end{lemma}

\begin{proof}
As \( C \to C[S^{-1}] \) and \( C \to C[\overline{S}^{-1}] \) are epimorphisms by proposition \ref{prop:trivialization is epi}, the morphism \( C[S^{-1}] \to C[\overline{S}^{-1}] \) is an equivalence if and only if there exists a morphism \( C[\overline{S}^{-1}] \to C[S^{-1}] \) under \( C \). In other words, we must show that for all \( a \in \overline{S} \), the morphism \( \Db_{|a|} \to C[S^{-1}] \) is a unit.

By the definition of \( \overline{S} \), it suffices to show that the subset of elements of \( \overline{S} \) such that \( \Db_{|a|} \to C[S^{-1}] \) is a unit is stable under composition, which is evident.
\end{proof}

\subsection{Gray operations for marked $\omega$-categories}

\begin{definition}
Let $A$ and $B$ be two $\omega$-categories, and $S$ and $T$ two sub homotopy types of cells of $A$ and $B$. We denote by $S\otimes T$ the sub homotopy types consisting of cells $$\Db_{n+m}\xrightarrow{\iota_{n,m}} \Db_{n}\otimes \Db_{m}\xrightarrow{a\otimes b} A\otimes B$$ for $a\in S_n$ and $b\in T_m$, where $\iota_{n,m}$ is the only non-trivial $(n+m)$-cell of the (strict) $\omega$-category $\Db_{n}\otimes \Db_{m}$.
\end{definition}

\begin{notation}
Let $A$ be an $\omega$-category. We denote by $A_{\geq 0}$ the homotopy type of cells of any dimension, and $A_{> 0}$ the homotopy type of all cells of positive dimension.
\end{notation}

\begin{lemma}
\label{lemma:technicalities composition Gray2} Let $A$ and $B$ be two $\omega$-categories. We have equivalences: $$\begin{array}{rcl} (A\otimes B)_{\geq 0}& \sim& \overline{A_{\geq 0}\otimes B_{\geq 0}}\\ (A\otimes B)_{>0}& \sim& \overline{A_{> 0}\otimes B_{\geq 0}\cup A_{\geq 0}\otimes B_{>0}}\\ \end{array}$$
\end{lemma}

\begin{proof}
The lemma 2.23 of \cite{Henry_an_inductive_model_structure_for_infini_categories} implies these equivalences in the case where $A,B$ are strict, and so in particular, when they are globular sums.

Let $A,B$ be any two $\omega$-categories. We have an obvious inclusion 
$$A_{\geq 0}\otimes B_{\geq 0}\subset (A\otimes B)_{\geq 0}$$
inducing an inclusion
$$\overline{A_{\geq 0}\otimes B_{\geq 0}}\subset (A\otimes B)_{\geq 0}.$$

Now, for any $a\to A$, $b\to B$, where $a$, $b$ are globular sums, we have the inclusion 
$$(a\otimes b)_{\geq 0}\subset \overline{a_{\geq 0}\otimes b_{\geq 0}}\subset \overline{A_{\geq 0}\otimes B_{\geq 0}}.$$
However, as $\otimes$ is colimit preserving, 
$$(A\otimes B)_{\geq 0}\sim \overline{\cup_{a:\Theta_{/A}, b:\Theta_{/B}}(a\otimes b)_{\geq 0}}$$
This then implies the other inclusion 
$$(A\otimes B)_{\geq 0}\subset \overline{A_{\geq 0}\otimes B_{\geq 0}}$$

We demonstrate similarly the second inclusion.
\end{proof}

\begin{lemma}
\label{lemma:technicalities composition Gray} Let $A$ and $B$ be strict $\omega$-categories, and $M\subset\coprod_{n>0} A_n$, $N\subset \coprod_{n>0}B$. Then $$\overline{\overline{M}\otimes N}\sim \overline{M\otimes N}\sim \overline{M\otimes \overline{N}}.$$
\end{lemma}

\begin{proof}
We will only show the equality $\overline{M \otimes N} = \overline{\overline{M} \otimes N}$. The equality $\overline{M \otimes N} = \overline{M \otimes \overline{N}}$ can be proved in the same way, and the last equality follows immediately by applying the result to $M$ and $\overline{N}$.

The evident inclusion $M \subset \overline{M}$ implies $\overline{M \otimes N} \subset \overline{\overline{M} \otimes N}$, so it is enough to show that $\overline{M} \otimes N \subset \overline{M \otimes N}$.

Let $K$ be the homotopy type of arrows $a$ in $A$ such that $a \otimes b \in \overline{M \otimes N}$ for all $b \in N$. We need to show that $K$ is closed under identity and composition to finish the proof.

If $a = \Ib_x$, then $a \otimes b = \Ib_{x \otimes b} \in \overline{M \otimes N}$. The homotopy type $K$ then includes all units. Let $a, a' \in A$ be two $k$-composable $n$-cells and let $b \in N$ be an arrow of dimension $m$ in $Y$, such that $a \otimes b$ and $a' \otimes b$ are in $\overline{M \otimes N}$. The two cells $a, a'$ are classified by a morphism $\Db_n \coprod_{\Db_i} \Db_n \to A$ and the cell $b$ by a morphism $\Db_m \to B$. All together, this induces a functor 
$$e:\left( \Db_n \coprod_{\Db_i} \Db_n \right) \otimes \Db_m \to A \otimes B.$$
In particular, this implies that we can reduce the problem to the case where $A$ and $B$ are strict, and we then apply lemma \cite[lemma 2.25]{Henry_an_inductive_model_structure_for_infini_categories} 
to demonstrate that $(a \circ_i a') \otimes b$ is in $\overline{M \otimes N}$. The homotopy type $K$ is then closed under composition, which concludes the proof.
\end{proof}

\begin{definition}
Let $(A,tA)$ and $(B,tB)$ be two marked $\omega$-categories. The \textit{Gray tensor product} of $(A,tA)$ and $(B,tB)$ is the marked $\omega$-category $$(A,tA)\otimes (B,tB):= (A\otimes B, \overline{tA\otimes B_{\geq 0}\cup A_{\geq 0}\otimes tB}) [(tA\otimes B_{>0})^{-1}].$$ By the description of colimits in marked $\omega$-categories, and by lemmas \ref{lemma:loc and col} and \ref{lemma:technicalities composition Gray}, this defines a cocontinuous bifunctor: $$\otimes:\ocatm\times \ocatm\to \ocatm.$$
\end{definition}

\begin{remark}
The definition of the Gray tensor product is strongly asymmetric. For example, the underlying $\omega$-category of $A^\flat\otimes B$ is $A\otimes B$, while we will show in proposition \ref{prop:sharp and times} that the underlying $\omega$-category of $A^\sharp\otimes B$ is $A\times B$.
\end{remark}

\begin{definition}
Let $A$, $B$ be two marked $\omega$-categories. We denote by $A^{B}$ the value on $A$ of the right adjoint of the functor $C\mapsto C\otimes B$, and we then have $A^{[0]}\sim A$.

The \textit{directed pullback} of a span $A\to B\leftarrow C$, denoted by $A\dirtimes{B}C$, is the limit of the diagram 
\[\begin{tikzcd}
	A & {B^{\{0\}}} & {B^{[1]^\sharp}} & {B^{\{1\}}} & C
	\arrow[from=1-1, to=1-2]
	\arrow[from=1-3, to=1-2]
	\arrow[from=1-3, to=1-4]
	\arrow[from=1-5, to=1-4]
\end{tikzcd}\]
\end{definition}

\begin{prop}
\label{prop:dualities marked otimes} Let $A,B$ be two marked $\omega$-categories. We have a canonical equivalence $$(A\otimes B)^{\circ} \sim A^{\circ} \otimes B^{\circ}.$$
\end{prop}

\begin{proof}
This directly follows from the construction of the Gray tensor product for marked $\omega$-categories and from form \ref{prop:otimes and dualities}.
\end{proof}

\begin{theorem}
\label{theo:Gray tensor product is monoidal on marked} The Gray tensor product induces a monoidal structure on $\ocatm$.
\end{theorem}

\begin{proof}
We denote by $\zocat_{2m}$ the category of bi-marked strict $\omega$-categories, defined as the pullback: $$\zocat_{2m}\sim \zocat_{m}\times_{\zocat}\zocat_{m}.$$ Objects of $\zocat_{2m}$ then correspond to the triplet $(A,t^0A,t^1A)$ where $(A,t^0A)$ and $(A,t^1A)$ are marked $\omega$-category. We will refer to $t^0A$ as the primary marking and $t^1A$ as the secondary marking.

We have a bifunctor: $$\otimes: \zocat_{2m}\times \zocat_{2m}\to \zocat_{2m}$$ sending $(A,t^0A,t^1A)$ and $(B,t^0B,t^1B)$ to $$(A\otimes B, \overline{t^0A\otimes B_{\geq 0}\cup A_{\geq 0}\otimes t^0B},\overline{t^0A\otimes B_{>0}\cup t^1A\otimes B_{\geq 0}\cup A_{\geq 0}\otimes t^1B}).$$

Suppose now we are given three bi-marked strict $\omega$-categories, $(A,t^0A,t^1A)$, $(B,t^0B,t^1B)$, and $(C,t^0C,t^1C)$. We want to show that the isomorphism $$(A\otimes B)\otimes C\sim A\otimes (B\otimes C)$$ promotes to an isomorphism $$(((A,t^0A,t^1A)\otimes (B,t^0B,t^1B))\otimes (C,t^0C,t^1C))\sim (A,t^0A,t^1A)\otimes ((B,t^0B,t^1B)\otimes (C,t^0C,t^1C)).$$ We then have to check that the primary and secondary markings coincide. By lemma \ref{lemma:technicalities composition Gray2} and \ref{lemma:technicalities composition Gray}, the primary marking corresponds to : $$\overline{t^0A\otimes B_{\geq 0}\otimes C_{\geq 0}\cup A_{\geq 0}\otimes t^0B\otimes C_{\geq 0}\cup A_{\geq 0}\otimes B_{\geq 0}\otimes t^0C}$$ and the secondary marking corresponds to $$\begin{array}{r} \overline{t^0A\otimes B_{\geq 0}\otimes C_{>0}\cup A_{\geq 0}\otimes t^0B\otimes C_{>0}\cup t^0A\otimes B_{>0}\otimes C_{\geq 0}}\\ \overline{\cup t^1A\otimes B_{\geq 0}\otimes C_{\geq 0}\cup A_{\geq 0}\otimes t^1B\otimes C_{\geq 0}\cup A_{\geq 0}\otimes B_{\geq 0}\otimes t^1C} \end{array}$$

The monoidal structure $\otimes$ on $\zocat$ then induces a monoidal structure $\otimes$ on $\zocat_{2m}$. We denote by $(\Theta_{\otimes})_{2m}$ the full subcategory of $\zocat_{2m}$ whose objects are the bimarked $\omega$-categories whose underlying $\zo$-categories belong to the category $\Theta_\otimes$ defined in the proof of theorem \ref{theo:monoidal}. 

By restriction, $\otimes$ induces a monoidal structure on $(\Theta_{\otimes})_{2m}$, and so by Day convolution (proposition 2.14 of \cite{glasman2013day}), on $\widehat{(\Theta_{\otimes})_{2m}}$. We furthermore have a localization 
\[\begin{tikzcd}
	{L:\widehat{(\Theta_{\otimes})_{2m}}} & {\ocatm:R}
	\arrow[""{name=0, anchor=center, inner sep=0}, shift left=2, from=1-1, to=1-2]
	\arrow[""{name=1, anchor=center, inner sep=0}, shift left=2, from=1-2, to=1-1]
	\arrow["\dashv"{anchor=center, rotate=-90}, draw=none, from=0, to=1]
\end{tikzcd}\]
where the left adjoint sends an object $(C,t^0C,t^1C)$ of $(\Theta_{\otimes})_{2m}$ to $(C,t^0C)[(t^1C)^{-1}]$. Moreover, by construction, we have a canonical commutative square 
\[\begin{tikzcd}
	{\widehat{(\Theta_{\otimes})_{2m}}\times \widehat{(\Theta_{\otimes})_{2m}}} & {\widehat{(\Theta_{\otimes})_{2m}}} \\
	{\ocatm\times\ocatm} & \ocatm
	\arrow["\otimes", from=1-1, to=1-2]
	\arrow["{L\times L}"', from=1-1, to=2-1]
	\arrow["L", from=1-2, to=2-2]
	\arrow["\otimes"', from=2-1, to=2-2]
\end{tikzcd}\]
and as the Gray tensor product preserves colimits, we can easily check that the canonical natural transformation $$L(RC\otimes RD)\to C\otimes D$$ is an equivalence.

Finally, by proposition 2.2.1.9 of \cite{LurieHigherAlgebra}, we can transfer the monoidal structure on $\widehat{(\Theta_{\otimes})_{2m}}$ to $\ocatm$.
\end{proof}

\begin{prop}
\label{prop:an other description of the suspension marked} Let $C$ be a marked $\omega$-category. There is an equivalence between $[C,1]$ and the colimit of the diagram: \[\begin{tikzcd} {\partial[1]^{\flat}} & {\partial[1]^{\flat}\otimes C^{\circ}} & {[1]^\flat\otimes C^{\circ}} \arrow[from=1-2, to=1-1] \arrow[from=1-2, to=1-3] \end{tikzcd}\]
\end{prop}

\begin{proof}
As all the functors appearing are colimit preserving, it is sufficient to construct this equivalence when $C$ is in $\Theta_t$. Proposition \ref{prop:an other description of the suspension} implies the equivalence at the level of the underlying $\omega$-category. We leave it to the reader to check that the markings agree.
\end{proof}

\begin{prop}
\label{prop:decomposition marked case} Let $C$ and $D$ be marked $\omega$-categories. There is a canonical equivalence between $[C,1]\otimes[D,1]$ and the colimit of the diagram $$[C,1]\vee[D,1]\leftarrow [C\otimes\{1\}\otimes D,1]\to [C\otimes[1]^\flat \otimes D,1]\leftarrow[C\otimes\{0\}\otimes D,1]\to [D,1]\vee[C,1].$$ There is a canonical equivalence between $[C,1]\otimes [1]^\sharp$ and the colimit of the diagram $$[C,1]\vee [1]^\sharp \leftarrow [C\otimes\{1\},1]\to [C\otimes[1]^\sharp,1]\leftarrow[C\otimes\{0\},1]\to [1]^\sharp\vee[C,1].$$
\end{prop}

\begin{proof}
This directly follows from the definition of the marked Gray tensor product and from proposition \ref{prop:decomposition}.
\end{proof}

\begin{lemma}
\label{lemma:cancel to get cartesian}
Let $C$ and $D$ be two $\omega$-categories. The canonical morphism $$\phi_{C,D}:(C\otimes D)[(C_{>0}\otimes D_{>0})^{-1}]\to C\times D$$ is an equivalence.
\end{lemma}

\begin{proof}
As $\phi_{\uvar,\uvar}$ commutes with limits, it is sufficient to demonstrate the result when $C$ and $D$ are globular sums $c,d$, and we will proceed by induction on $(|c|,|d|)$.

First, let us remark that $\phi_{[1],[1]}$ is an equivalence. By extension by colimits, this implies that $\phi_{[k],[l]}$ is an equivalence for any $k,l$. Suppose now the result is true for any globular sum of the stage $(n,m)$. In particular, this implies that $\phi_{C,D}$ is an equivalence for any $n$-category $C$ and $m$-category $D$. By the associativity of $\otimes$ and by lemma \ref{lemma:technicalities composition Gray2}, the induction hypothesis then implies that  
\begin{equation}
\label{eq:an eq}
\begin{array}{rcl}
((c\otimes [1])\otimes d)[\cup_{i+j+k=1}\left(c_{\geq 1-i}\otimes [1]_{\geq 1-j}\otimes d_{\geq 1-k}\right)^{-1} ]&\sim & (c \otimes [1])[(c_{>0}\otimes [1]_{>0})^{-1}]\times d\\
&\sim & c \times [1]\times d
\end{array}
\end{equation}
for any globular sums $c$ and $d$ such that $|c|=n-1$ and $|d|=m$. 

Proposition \ref{prop:decomposition} provides a canonical morphism $[c\otimes[1]\otimes d,1]\to [c,1]\otimes [d,1]$. By construction, this morphism sends a cell $[x\otimes y\otimes z,1]$ onto $[x,1]\otimes [z,1]$ when $y$ is the unique non-degenerate one-cell of $[1]$, and onto a unit if $|y|=0$ and $|x|+|z|>0$. 

The proposition \ref{prop:decomposition}, together with proposition \ref{prop:trivialization is epi} and lemma \ref{lemma:loc and col}, then implies that $([c,1]\otimes [d,1])[([c,1]_{>0}\otimes [d,1]_{>0})^{-1}]$ is the colimit of the diagram:
\[\begin{tikzcd}
	{[(c\otimes\{1\}\otimes d)[(c_{> 0}\otimes d_{>0})^{-1}],1]} & {[c,1]\vee[d,1]} \\
	{[(c\otimes[1]\otimes d)[\cup_{i+j=1}(c_{\geq i}\otimes [1]_{\geq j}\otimes d_{\geq i})^{-1}],1]} \\
	{[(c\otimes\{0\}\otimes d)[(c_{> 0}\otimes d_{>0})^{-1}],1]} & {[d,1]\vee[c,1]}
	\arrow[from=1-1, to=1-2]
	\arrow[from=1-1, to=2-1]
	\arrow[from=3-1, to=2-1]
	\arrow[from=3-1, to=3-2]
\end{tikzcd}\]
However, the equation \eqref{eq:an eq} implies that $(c\otimes[1]\otimes d)[\cup_{i+j=1}(c_{\geq i}\otimes [1]_{\geq j}\otimes d_{\geq i})^{-1}]$ is the colimit of the diagram
\[\begin{tikzcd}
	{c_0\times[1]} & {c_0\times[1]\times d} & {c\times [1]\times d} & {c\times[1]\times d_0} & {[1]\times d}
	\arrow[from=1-2, to=1-1]
	\arrow[from=1-2, to=1-3]
	\arrow[from=1-4, to=1-3]
	\arrow[from=1-4, to=1-5]
\end{tikzcd}\]
which is equivalent to $c\times d$ by \cite[proposition 2.2.1.59]{loubaton2024categorical}. 

Thus, $([c,1]\otimes [d,1])[([c,1]_{>0}\otimes [d,1]_{>0})^{-1}]$ is the colimit of the following diagram:
\[\begin{tikzcd}
	{[c\times d,1]} & {[c,1]\vee[d,1]} \\
	{[c\times d,1]} \\
	{[c\times d,1]} & {[d,1]\vee[c,1]}
	\arrow[from=1-1, to=1-2]
	\arrow[from=1-1, to=2-1]
	\arrow[from=3-1, to=2-1]
	\arrow[from=3-1, to=3-2]
\end{tikzcd}\]
and is then equivalent to $[c,1]\times [d,1]$ by \cite[proposition 2.2.1.29]{loubaton2024categorical}. As every globular sum of dimension $p$ is a colimit of globular sums of shape $[f,1]$ with $|f|=p-1$, this implies the case $(n,m+1)$. We demonstrate similarly the case $(n+1,m)$.
\end{proof}

\begin{prop}
\label{prop:sharp and times}
Let $C$ be an $\omega$-category and $D$ a marked $\omega$-category. The canonical morphism $$C^\sharp \otimes D \sim C^\sharp \times D$$ is an equivalence.
\end{prop}

\begin{proof}
By construction, $C^\sharp\otimes D$ is the colimit of the diagram
\[\begin{tikzcd}
	{\tau_0C\otimes D} & {\tau_0C\otimes D^\flat} & {(C\otimes D^\flat)[(C_{>0}\otimes D_{>0})^{-1}]} & {C\otimes\tau_0D^\flat} & {C^\sharp\otimes\tau_0D^\flat}
	\arrow[from=1-2, to=1-1]
	\arrow[from=1-2, to=1-3]
	\arrow[from=1-4, to=1-3]
	\arrow[from=1-4, to=1-5]
\end{tikzcd}\]
By lemma \ref{lemma:cancel to get cartesian}, $(C\otimes D^\flat)[(C_{>0}\otimes D_{>0})^{-1}]$ is equivalent to $C\times D^\flat$, which concludes the proof.
\end{proof}

\begin{prop}
\label{prop:otimes and epi}
Let $C, D$ be two marked $\omega$-categories. The morphism $C\otimes D\to C\times \tau_0^iD$ is an epimorphism.
\end{prop}

\begin{proof}
The proposition \ref{prop:trivialization is epi} implies that $D\to  \tau_0^iD$ is an equivalence. The result then follows from the fact that $C\otimes\uvar$ preserves colimit. 
\end{proof}

\section{Two-sided fibrations}
\label{section:Two-sided fibrations}

  \subsection{Biinitial morphisms}

\begin{definition}
Let $A$ and $B$ be marked $\omega$-categories. We denote by $\I_{A,B}$ the set of morphisms of $\ocatm_{/A\times B}$ of shape 
\[\begin{tikzcd}
	& {X\otimes [1]^\sharp} \\
	{X\otimes \{0\}} & A\times B
	\arrow[from=1-2, to=2-2]
	\arrow[from=2-1, to=1-2]
	\arrow[from=2-1, to=2-2]
\end{tikzcd}\]
such that $X\otimes [1]^\sharp\to A$ factors through $X$, and $\F_{A,B}$ the set of morphisms of $\ocatm_{/A\times B}$ of shape 
\[\begin{tikzcd}
	& {X\otimes [1]^\sharp} \\
	{X\otimes \{1\}} & A\times B
	\arrow[from=1-2, to=2-2]
	\arrow[from=2-1, to=1-2]
	\arrow[from=2-1, to=2-2]
\end{tikzcd}\]
such that $X\otimes [1]^\sharp\to B$ factors through $X$, where $X$ is either $\Db_n^\flat$ or $(\Db_n)_t$.

We denote by $\IF_{A, B}$ the set of morphisms over $A\times B$ that are either in $\I_{A,B}$ or in $\F_{A,B}$. A morphism is \textit{biinitial over $(A,B)$} if it is in $\widehat{\IF_{A,B}}$. Biinitial morphisms are then stable under colimits, transfinite composition, pushouts, left cancellation, and retracts.
\end{definition}

\begin{remark}
As $X\otimes [1]^\sharp\to X$ is an epimorphism by proposition \ref{prop:otimes and epi}, the fact that $X\otimes[1]^\sharp\to A$ or $X\otimes[1]^\sharp\to B$ indeed factors through $X$ is a property of this morphism and not a structure.
\end{remark}

\begin{remark}
\label{rem:stabilites of biinitial by duality}
Let $i$ be a morphism over $A\times B$.
The proposition \ref{prop:dualities marked otimes} implies that $i$ is biinitial over $(A,B)$ if and only if $i^\circ$ is biinitial over $(B^\circ,A^\circ)$.
\end{remark}

\begin{remark}
\label{rem:biinitial stable by composition}. Let $A\to A'$ and $B\to B'$ be two morphisms. The functor $\ocatm_{/A\times B}\to \ocatm_{/A'\times B'}$ sends $\widehat{\IF_{A,B}}$ onto $\widehat{\IF_{A',B'}}$.
\end{remark}

\begin{example}
\label{exe:the easiest example of initial and final morphism} By the stability of initial and final morphisms under colimits, for any marked $\omega$-category $C$, a morphism \[\begin{tikzcd} & {C\otimes [1]^\sharp} \\ {C\otimes \{0\}} & {A\times B} \arrow[from=1-2, to=2-2] \arrow[from=2-1, to=1-2] \arrow[from=2-1, to=2-2] \end{tikzcd}\] is biinitial as long as $C\otimes[1]^\sharp$ factors through $A$, and a morphism \[\begin{tikzcd} & {C\otimes [1]^\sharp} \\ {C\otimes \{1\}} & {A\times B} \arrow[from=1-2, to=2-2] \arrow[from=2-1, to=1-2] \arrow[from=2-1, to=2-2] \end{tikzcd}\] is biinitial as long as $C\otimes[1]^\sharp$ factors through $B$.
\end{example}

\begin{lemma}
\label{lemma:biinitial stable by otimes} If $i:C\to D$ is initial over $(A,B)$, then $E\otimes i$ is also a deformation retract over $(E\otimes A,E\otimes B)$, and so is $E\times i$ over $(E\times A,E\times B)$.
\end{lemma}

\begin{proof}
Straightforward.
\end{proof}

\begin{lemma}
\label{lemma:deformation retract stable by suspension} If $i:C\to D$ is biinitial over $(A,B)$, then $[i,1]$ is a right deformation retract over $([B,1],[A,1])$.
\end{lemma}

\begin{proof}
By proposition \ref{prop:an other description of the suspension marked}, the morphism $[i,1]:[C,1]\to [D,1]$ is the horizontal colimit of the diagram:
\[\begin{tikzcd}
	{\partial[1]^\flat} & {\partial[1]^\flat\otimes C^\circ} & {[1]^\flat\otimes C^\circ} \\
	{\partial[1]^\flat} & {\partial[1]^\flat\otimes D^\circ} & {[1]^\flat\otimes D^\circ} \\
	{(\partial[1]^\flat)\times (\partial[1]^\flat)} & {(\partial[1]^\flat\otimes B)\times (\partial[1]^\flat\otimes A)} & {([1]^\flat\otimes B)\times ([1]^\flat\otimes A)}
	\arrow[from=1-1, to=2-1]
	\arrow[from=1-2, to=1-1]
	\arrow[from=1-2, to=1-3]
	\arrow[from=1-2, to=2-2]
	\arrow[from=1-3, to=2-3]
	\arrow[from=2-1, to=3-1]
	\arrow[from=2-2, to=2-1]
	\arrow[from=2-2, to=2-3]
	\arrow[from=2-2, to=3-2]
	\arrow[from=2-3, to=3-3]
	\arrow[from=3-2, to=3-1]
	\arrow[from=3-2, to=3-3]
\end{tikzcd}\]
The result then follows from remark \ref{rem:stabilites of biinitial by duality} and from lemma \ref{lemma:biinitial stable by otimes}.
\end{proof}

\begin{cor}
\label{cor:when glob inclusion are final and initial} Let $(\Db_{n+1})_t\to A\times B$ be a morphism such that $(\Db_{n+1})_t\to A$ factors through $\Db_{n}^\flat$. If $n$ is even, $i_{n}^-:\Db_{n}^\flat\to (\Db_{n+1})_t$ is biinitial over $(A,B)$, and if $n$ is odd, $i_{n}^+:\Db_{n}^\flat\to (\Db_{n+1})_t$ is biinitial over $(A,B)$. Similarly, if $(\Db_{n+1})_t\to B$ factors through $\Db_{n}^\flat$, and $n$ is even, $i_{n}^+:\Db_{n}^\flat\to (\Db_{n+1})_t$ is biinitial over $(A,B)$, and if $n$ is odd, $i_{n}^-:\Db_{n}^\flat\to (\Db_{n+1})_t$ is biinitial over $(A,B)$.
\end{cor}

\begin{proof}
The case $n=0$ is trivial, and we can proceed by induction on $n$ using proposition \ref{lemma:deformation retract stable by suspension}.
\end{proof}

\begin{definition}
A \textit{marked trivialization over $(A,B)$} is a morphism in the smallest cocomplete class of morphisms that includes $\Ib_n:(\Db_{n+1})_t\to \Db_n^{\flat}$ for any \( n \). Marked trivialization are then stable under colimits, transfinite composition, pushouts, left cancellation, and retracts.
\end{definition}

\begin{example}
\label{exe:cotimes 1 to c is a trivialization} The proposition 3.2.1.8 of \cite{loubaton2024categorical} states that $C\otimes [1]^\sharp\to C$ is a marked trivialization.
\end{example}

\begin{prop}
A marked trivialization $i:C\to B$ over $A\times B$ is biinitial over $(A,B)$.
\end{prop}

\begin{proof}
As biinitial morphisms over $(A,B)$ are closed under colimits, it is sufficient to demonstrate that for any $n$, the morphism $\Ib_n$ is biinitial over $(A,B)$. According to corollary \ref{cor:when glob inclusion are final and initial}, there exists $\alpha\in\{-,+\}$ such that $i_n^\alpha$ is biinitial over $(A,B)$. As $\Ib_n$ is a retraction of this morphism, and as biinitial morphisms over $(A,B)$ are closed under left cancellation, $\Ib_n$ is biinitial over $(A,B)$.
\end{proof}

\begin{definition}
Let $A, B$ be two $\omega$-categories. Let $i:C\to D$ be a morphism over $A\times B$. A \textit{left Gray deformation retract structure over $(A,B)$} for $i$ is the data of a \textit{retract} $r:D\to C$ over $A$, a \textit{deformation} $\psi:D\otimes [1]^\sharp\to D$ over $A$\footnote{where we consider the morphism $D\otimes[1]\to A\otimes[1]\to A\otimes[0]\sim A$}, and equivalences over $A$: $$ri\sim id_C~~~~~\psi_{|D\otimes\{0\}}\sim ir~~~~~\psi_{|D\otimes\{1\}}\sim id_D~~~~~ \psi_{|C\otimes[1]^\sharp}\sim i\cst_C$$ By abuse of notation, such data will just be denoted by $(i,r,\psi)$. A morphism $i:C\to D$ over $A\times B$ is a \textit{left Gray deformation retract} if it admits a left deformation retract structure.

We define similarly the notion of \textit{right Gray deformation retract (structure) over $(A,B)$} by interchanging $0$ and $1$ and $A$ and $B$ in the previous definition.
\end{definition}

\begin{remark}
\label{rem:stabilites of Gray retract by duality}
Let $i$ be a morphism over $A\times B$.
The proposition \ref{prop:dualities marked otimes} implies that $i$ is a left Gray retract over $(A,B)$ if and only if $i^\circ$ is a right Gray retract over $(B^\circ,A^\circ)$.
\end{remark}

\begin{prop}
\label{prop:example of left Gray deformation} Suppose given a morphism $C\otimes[1]^\sharp\to A\times B$ such that $C\otimes[1]^\sharp\to A$ factors through $C$. Then $C\otimes\{0\}\to C\otimes[1]^\sharp$ is a left Gray deformation retract over $(A,B)$. Conversely, if $C\otimes[1]^\sharp\to B$ factors through $A$, $C\otimes\{1\}\to C\otimes[1]^\sharp$ is a right Gray deformation retract over $(A,B)$.
\end{prop}

\begin{proof}
We will show only the first assertion; the second follows by dualities. The retract is then given by $C\otimes[1]^\sharp\to C\otimes[0]\sim C$, and the deformation by the functor 
$$C\otimes[1]^\sharp\otimes[1]^\sharp\sim C\otimes ([1]\times[1])^\sharp\xrightarrow{C\otimes \min} C\otimes[1]^\sharp$$
where $\min:[1]\times[1]\to[1]$ is the functor sending $(i,j)$ to $\min(i,j)$. We recall that the first equivalence follows from proposition \ref{prop:sharp and times}. We leave it to the reader to check that this retract and deformation induce a left Gray deformation retract structure on $C\otimes\{0\}\to C\otimes[1]^\sharp$.
\end{proof}

\begin{prop}
\label{prop:canonical example of Gray deformation retract} For any marked $\omega$-category $B$, $B\to B^{[1]^\sharp}$ is both a left and right Gray deformation retract over $B\times B$.
\end{prop}

\begin{proof}
We will show that this morphism is a left Gray deformation retract; the other case is demonstrated similarly. The retraction is given by the evaluation at $0$, i.e., the morphism $B^{[1]^\sharp}\to B^{\{0\}}$, and the deformation by the morphism $$B^{[1]^\sharp}\otimes[1]^\sharp\to B^{[1]^\sharp}$$ induced by adjunction from the morphism $$B^{\min}:B^{[1]^\sharp}\to B^{([1]\times [1])^\sharp}\sim (B^{[1]^\sharp})^{[1]^\sharp}.$$
\end{proof}

\begin{prop}
\label{prop:left Gray deformation retract are initial} Left and right Gray deformation retracts over $(A,B)$ are biinitial over $(A,B)$.
\end{prop}

\begin{proof}
Let $ i:C\to D$ be a left Gray deformation retract over $(A,B)$, and $r,\phi$ the retraction and the deformation. We define $ D\otimes_C[1]^\sharp$ as the following pushout:
\[
\begin{tikzcd}[ampersand replacement=\&]
	{C\otimes[1]^\sharp} \& {D\otimes[1]^\sharp} \\
	C \& {D\otimes_C[1]^\sharp}
	\arrow[from=1-1, to=1-2]
	\arrow[from=1-1, to=2-1]
	\arrow[from=1-2, to=2-2]
	\arrow[from=2-1, to=2-2]
	\arrow["\lrcorner"{anchor=center, pos=0.125, rotate=180}, draw=none, from=2-2, to=1-1]
\end{tikzcd}
\]
Consider now the diagram:
\[
\begin{tikzcd}
	{C\otimes\{0\}} & {C\otimes[1]^\sharp} & C \\
	& {A\times B}
	\arrow["i", from=1-1, to=1-2]
	\arrow[from=1-1, to=2-2]
	\arrow[from=1-2, to=1-3]
	\arrow[from=1-2, to=2-2]
	\arrow[from=1-3, to=2-2]
\end{tikzcd}
\]
It follows from stability by left cancellation and from example \ref{exe:the easiest example of initial and final morphism} that $ C\otimes [1]^\sharp\to C$ is biinitial over $(A,B)$. By stability by pushout and composition, and using once again example \ref{exe:the easiest example of initial and final morphism}, so is $ D\otimes\{0\}\to D\otimes_C[1]^\sharp$. Eventually, the diagram
\[
\begin{tikzcd}[ampersand replacement=\&]
	C \& {D\otimes\{0\}} \&\& C \\
	{D\otimes\{1\}} \& {D\otimes_C [1]^\sharp} \& {D\otimes[1]^\sharp} \& D
	\arrow["i", from=1-1, to=1-2]
	\arrow["i"', from=1-1, to=2-1]
	\arrow["r", from=1-2, to=1-4]
	\arrow[from=1-2, to=2-2]
	\arrow["i", from=1-4, to=2-4]
	\arrow[from=2-1, to=2-2]
	\arrow[from=2-2, to=2-3]
	\arrow["\psi"', from=2-3, to=2-4]
\end{tikzcd}
\]
seen as a diagram in $ \ocatm_{/A\times B}$, expresses $ i$ as a retract of $ D\otimes \{0\}\to D\otimes_C [1]^\sharp$. The morphism $ i$ is then biinitial over $(A,B)$. The case of the right Gray deformation retract is proven similarly.
\end{proof}

\subsection{Definition of two-sided fibrations}

\begin{definition}
A span $X\to A\times B$ is a \textit{two-sided fibration} if for every biinitial morphism $i:C\to D$ over $(A,B)$, any square of shape
\[\begin{tikzcd}
	C & X \\
	D & {A\times B}
	\arrow[from=1-1, to=1-2]
	\arrow["i"', from=1-1, to=2-1]
	\arrow[from=1-2, to=2-2]
	\arrow[dashed, from=2-1, to=1-2]
	\arrow[from=2-1, to=2-2]
\end{tikzcd}\]
admits a unique lifting. A bifibration is then an $\IF_{A,B}$ local object of $\ocatm_{/A\times B}$. 

We denote by $\BCart(A,B)$ the subcategory of $\ocatm_{/A\times B}$ spanned by two-sided fibrations, and we then have a canonical adjunction
\begin{equation}
\label{eq:adjonction two sided fibration and morphism over a b}
\begin{tikzcd}
	{\Fb:\ocatm_{/A\times B}} & {\BCart(A,B):\iota}
	\arrow[""{name=0, anchor=center, inner sep=0}, shift left=2, from=1-1, to=1-2]
	\arrow[""{name=1, anchor=center, inner sep=0}, shift left=2, from=1-2, to=1-1]
	\arrow["\dashv"{anchor=center, rotate=-90}, draw=none, from=0, to=1]
\end{tikzcd}
\end{equation}
where the right adjoint is fully faithful. The functor $\Fb$ is called the \textit{two-sided replacement functor}.
\end{definition}

\begin{remark}
As two-sided fibrations are defined as local objects for some set of morphisms, they are stable under limits and retracts in $\BCart(A,B)$.
\end{remark}

\begin{prop}
\label{prop:bicart stable by pullback} Let $S:X\to A\times B$ be a two-sided fibration, and $i:A'\to A$, $j:B\to B'$ two morphisms. Then $$(i,j)^*S:X\to A'\times B'$$ is a two-sided fibration.
\end{prop}

\begin{proof}
This directly follows from the remark \ref{rem:biinitial stable by composition}.
\end{proof}

\begin{prop}
Let $X_i \to A_i \times B_i$ be a family of two-sided fibrations. Then $$\lim_i X_i \to \lim_i A_i \times \lim_i B_i$$ is a two-sided fibration.
\end{prop}

\begin{proof}
Straightforward.
\end{proof}

\begin{prop}
\label{prop:left fib over flat}
If $f:X\to A^\flat\times B^\flat$ is a two-sided fibration, then the canonical morphism $X^\flat\to X$ is an equivalence. Conversely, any morphism $X^\flat \to A^\flat\times B^\flat$ is a two-sided fibration.

In particular, if $X\to A\times B$ is a two-sided fibration, then for any $x:A$, $y:B$, the fiber $X_{a,b}$ is trivially marked.
\end{prop}

\begin{proof}
The first assertion is a consequence of the fact that marked trivializations are biinitial. The second assertion is a direct consequence of example \ref{exe:cotimes 1 to c is a trivialization} as $X^\flat \to A^\flat\times B^\flat$ obviously has the unique left lifting property against marked trivializations.
\end{proof}

\begin{prop}
\label{prop:cartesian fibration between arrow} Let $(p,q):X\to A\times B$ be a two-sided fibration, and let $x,y$ be two objects of $X$. Then, the morphism $p:\hom_X(x,y)\to \hom_B(qx,qy)\times \hom_A(px,py)$ is a two-sided fibration.
\end{prop}

\begin{proof}
This directly follows from the last assertion of lemma \ref{lemma:deformation retract stable by suspension}.
\end{proof}

\begin{definition}
A morphism $X\to A$ is a \textit{right cartesian fibration} if the morphism $X\to A\times 1$ is a two-sided fibration. Conversely, a morphism $X\to B$ is a \textit{left cartesian fibration} if the morphism $X\to 1\times B$ is a two-sided fibration.

We denote by $\RCart(A)$ (resp. $\LCart(B)$) the subcategory of $\ocatm_{/A}$ (resp. $\ocatm_{/B}$) whose objects are left (resp. right) cartesian fibrations.
\end{definition}

\begin{prop}
\label{prop:fibrations induced by bifibrations} Let $X\to A\times B$ be a two-sided fibration. Then $$A^\flat\times_AX\to B$$ is a left cartesian fibration, and $$X\times_BB^\flat\to A$$ is a right cartesian fibration.
\end{prop}

\begin{proof}
We will show only the first claim; the second is proven similarly. We have to show that the diagram of the shape 
\[\begin{tikzcd}
	{c\otimes\{0\}} & {A^\flat\times_AX} & X \\
	& {A^\flat\times B} & {A\times B} \\
	{c\otimes[1]^\sharp} & B
	\arrow[from=1-1, to=1-2]
	\arrow[from=1-1, to=3-1]
	\arrow[from=1-2, to=1-3]
	\arrow[from=1-2, to=2-2]
	\arrow["\lrcorner"{anchor=center, pos=0.125}, draw=none, from=1-2, to=2-3]
	\arrow[from=1-3, to=2-3]
	\arrow[from=2-2, to=2-3]
	\arrow["{\pi_B}", from=2-2, to=3-2]
	\arrow[from=3-1, to=3-2]
\end{tikzcd}\]
admits a unique lifting. Remark first that we can uniquely lift the morphism $c\otimes[1]^\sharp\to B$ to $A^\flat\times B$ as the projection $\pi_B$ has the unique left lifting property against marked trivialization, and as every morphism $c\otimes[1]^\sharp\to A^\flat$ uniquely factors through $c$ by example \ref{exe:cotimes 1 to c is a trivialization}. It is then sufficient to demonstrate that $A^\flat\times_AX\to A^\flat\times B$ is a two-sided fibration, which directly follows from proposition \ref{prop:bicart stable by pullback}.
\end{proof}

\begin{lemma}
\label{lemma: a partial biinitial} Suppose we are given a diagram 
\[\begin{tikzcd} {(\Db_n)_t} & {(\Db_n)_t\vee[1]^\sharp} \\ & {A\times B} \arrow["\triangledown", from=1-1, to=1-2] \arrow[from=1-1, to=2-2] \arrow[from=1-2, to=2-2] \end{tikzcd}\]
 where $\triangledown$ is the morphism preserving the endpoints. If the induced morphism $[1]^\sharp\to B$ factors through $[0]$, then $\triangledown$ is biinitial over $(A,B)$.
\end{lemma}

\begin{proof}
Remark that $\nabla$ is the following colimit of morphisms in $\ocat_{/A\times B}$: \[\begin{tikzcd} {[0]} & {\{0\}\times \{1\}} & {(\Db_n)_t\times\{1\}} \\ {[0]} & {\{0\}\times[1]^\sharp} & {(\Db_n)_t\times [1]^\sharp} \\ & {A\times B} \arrow[from=1-1, to=2-1] \arrow[from=1-2, to=1-1] \arrow[from=1-2, to=1-3] \arrow[from=1-2, to=2-2] \arrow[from=1-3, to=2-3] \arrow[from=2-1, to=3-2] \arrow[from=2-2, to=2-1] \arrow[from=2-2, to=2-3] \arrow[from=2-2, to=3-2] \arrow[from=2-3, to=3-2] \end{tikzcd}\] As all these morphisms are biinitial over $(A,B)$, so is their colimit.
\end{proof}

\begin{lemma}
\label{lemma:marking on left cartesian fibration} Let $(f,g):X\to A\times B$ be a left cartesian fibration. Then $X$ is equivalent to the marked $\omega$-category whose underlying $\omega$-category is $X^\natural$ and whose marking is freely generated by the marked cells in $X$ that are either over an object of $A$ or an object of $B$. More formally, we claim that the canonical morphism $$A^\flat\times_A X\coprod_{X^\flat}X\times_BB^\flat\to X$$ is an equivalence.
\end{lemma}

\begin{proof}
We will show by induction on $n$ and on all two-sided fibrations $(f,g):X\to A\times B$ that any marked $n$-cell $u$ in $X$ is a composite of marked $n$-cells which are either over a $0$-cell of $A$ or a $0$-cell of $B$.

The initialization is trivial, and suppose then the result is true at the stage $n$. Let $u$ be a marked $n$-cell in $X$ with $0$-source $a$ and $0$-target $b$. Applying the induction hypothesis to the bifibration $\hom_X(a,b)\to \hom_B(g(a),g(b))\times \hom_A(f(a),f(b))$, we can reduce to the case where $u$ is either over a $1$-cell of $A$ or over a $1$-cell of $B$. We will then suppose that we are in the first case, the other being treated similarly. This data then induces a diagram:
\[\begin{tikzcd}
	{(\Db_n)_t} && X \\
	{(\Db_n)_t\vee[1]^\sharp} & {[1]^\sharp\times (\Db_n)_t} & {A\times B}
	\arrow[from=1-1, to=1-3]
	\arrow["\triangledown"', from=1-1, to=2-1]
	\arrow["{(f,g)}", from=1-3, to=2-3]
	\arrow[from=2-1, to=2-2]
	\arrow["{f(v)\times g(u)}"', from=2-2, to=2-3]
\end{tikzcd}\]
which admits a unique lift by lemma \ref{lemma: a partial biinitial}. This then implies that the cell $u$ is a whiskering $u'\circ_0\overline{v}$ where $v$ is over an object of $B$ and $u'$ is over an object of $A$.
\end{proof}

\begin{cor}
\label{cor:morphism between is an equivalence when equivalence on fiberbifibrant case} Suppose we are given a morphism
\[\begin{tikzcd}
	X & Y \\
	{A^\sharp} & {B^\sharp}
	\arrow["\phi", from=1-1, to=1-2]
	\arrow[from=1-1, to=2-1]
	\arrow[from=1-1, to=2-2]
	\arrow[from=1-2, to=2-1]
	\arrow[from=1-2, to=2-2]
\end{tikzcd}\]
between two two-sided fibrations. Then $\phi$ is an equivalence if and only if for any $a:A$ and $b: B$, $$\phi_{a,b}:X_{a,b}\to Y_{a,b}$$ is an equivalence.
\end{cor}

\begin{proof}
This is obviously necessary. Suppose then that $\phi$ induces equivalences on fibers. As equivalences between left (resp. right) Cartesian fibrations are detected on fibers by \cite[corollary 3.2.2.8]{loubaton2024categorical}, this implies that for any $a:A$, $b:B$, the morphisms 
\[\begin{tikzcd}
	{X_b} & {Y_b} & {X_a} & {Y_a} \\
	& {A\times\{b\}} && {\{a\}\times B}
	\arrow["{\phi_b}", from=1-1, to=1-2]
	\arrow[from=1-1, to=2-2]
	\arrow[from=1-2, to=2-2]
	\arrow["{\phi_a}", from=1-3, to=1-4]
	\arrow[from=1-3, to=2-4]
	\arrow[from=1-4, to=2-4]
\end{tikzcd}\]
are equivalences. Using once again the fact that equivalences between left (resp. right) Cartesian fibrations are detected on fibers, and using proposition \ref{prop:fibrations induced by bifibrations} to apply this result, this implies that the two morphisms 
\[\begin{tikzcd}
	{A^\flat\times_AX} & {A^\flat\times_AY} & {X\times_BB^\flat} & {Y\times_BB^\flat} \\
	& A && B
	\arrow["{\phi_1}", from=1-1, to=1-2]
	\arrow[from=1-1, to=2-2]
	\arrow[from=1-2, to=2-2]
	\arrow["{\phi_2}", from=1-3, to=1-4]
	\arrow[from=1-3, to=2-4]
	\arrow[from=1-4, to=2-4]
\end{tikzcd}\]
are equivalences. By lemma \ref{lemma:marking on left cartesian fibration}, that concludes the proof.
\end{proof}

\begin{lemma}
\label{lemma:explicit factorization} For any $A$, the morphism $A^{[1]^\sharp}\to A\times A$ is a two-sided fibration.
\end{lemma}

\begin{proof}
Suppose we are given a diagram
\begin{equation}
\label{eq:square is explicit factorization}
\begin{tikzcd}
	{C\otimes\{0\}} & {A^{[1]^\sharp}} \\
	{C\otimes[1]^\sharp} & {A\times A}
	\arrow[from=1-1, to=1-2]
	\arrow[from=1-1, to=2-1]
	\arrow[from=1-2, to=2-2]
	\arrow["{(\pi_1,\pi_2)}"', from=2-1, to=2-2]
\end{tikzcd}
\end{equation}
such that $\pi_1$ factors through $C$. As we have a cocartesian square
\[\begin{tikzcd}
	{[1]^\sharp\times\{1\}} & {[1]^\sharp\times[1]^\sharp} \\
	{[0]} & {[1]^\sharp\vee[1]^\sharp}
	\arrow[from=1-1, to=1-2]
	\arrow[from=1-1, to=2-1]
	\arrow[from=1-2, to=2-2]
	\arrow[from=2-1, to=2-2]
	\arrow["\lrcorner"{anchor=center, pos=0.125, rotate=180}, draw=none, from=2-2, to=1-1]
\end{tikzcd}\]
the square \eqref{eq:square is explicit factorization} corresponds to a morphism 
$C\to A^{[2]^\sharp}$ 
and its unique lift is the induced morphism
$$C\to A^{[1]^\sharp}\xrightarrow{A^{\triangledown}} A^{[2]^\sharp}.$$

The case of the lifting property in a diagram
\[\begin{tikzcd}
	{C\otimes\{1\}} & {A^{[1]^\sharp}} \\
	{C\otimes[1]^\sharp} & {A\times A}
	\arrow[from=1-1, to=1-2]
	\arrow[from=1-1, to=2-1]
	\arrow[from=1-2, to=2-2]
	\arrow["{(\pi_1,\pi_2)}"', from=2-1, to=2-2]
\end{tikzcd}\]
such that $\pi_2$ factors through $C$ is similar.
\end{proof}

\begin{prop}
\label{prop:directed pullback i bifib} For any span $A\to B\leftarrow C$ of $\omega$-categories, the morphism $$A\dirtimes{B}C\to A\times C$$ is a two-sided fibration.
\end{prop}

\begin{proof}
This directly follows from lemma \ref{lemma:explicit factorization} and from the stability of two-sided fibration by pullback.
\end{proof}

\begin{prop}
\label{prop: another characterization} A span $(f,g):X\to A\times B$ is a two-sided fibration if and only if 
\begin{enumerate}
\item Every square of shape
\[\begin{tikzcd}
	{\{0\}} & X \\
	{[1]^\sharp} & {A\times B}
	\arrow[from=1-1, to=1-2]
	\arrow[from=1-1, to=2-1]
	\arrow[from=1-2, to=2-2]
	\arrow[dashed, from=2-1, to=1-2]
	\arrow[from=2-1, to=2-2]
\end{tikzcd}\]
where $[1]^\sharp\to A$ factors through $[0]$ admits a unique lifting.

Every square of shape
\[\begin{tikzcd}
	{\{1\}} & X \\
	{[1]^\sharp} & {A\times B}
	\arrow[from=1-1, to=1-2]
	\arrow[from=1-1, to=2-1]
	\arrow[from=1-2, to=2-2]
	\arrow[dashed, from=2-1, to=1-2]
	\arrow[from=2-1, to=2-2]
\end{tikzcd}\]
where $[1]^\sharp\to B$ factors through $[0]$ admits a unique lifting.
\item Every square of shape
\[\begin{tikzcd}
	{\Db_n^\flat} & X & {(\Db_n)_t} & X \\
	{\Db_n^\flat\vee[1]^\sharp} & {A\times B} & {(\Db_n)_t\vee[1]^\sharp} & {A\times B}
	\arrow[from=1-1, to=1-2]
	\arrow["\triangledown", from=1-1, to=2-1]
	\arrow[from=1-2, to=2-2]
	\arrow[from=1-3, to=1-4]
	\arrow[from=1-3, to=2-3]
	\arrow[from=1-4, to=2-4]
	\arrow[dashed, from=2-1, to=1-2]
	\arrow[from=2-1, to=2-2]
	\arrow[dashed, from=2-3, to=1-4]
	\arrow[from=2-3, to=2-4]
\end{tikzcd}\]
where $[1]^\sharp\to A$ factors through $[0]$ admits a unique lifting.

Every square of shape
\[\begin{tikzcd}
	{\Db_n^\flat} & X & {(\Db_n)_t} & X \\
	{[1]^\sharp\vee\Db_n^\flat} & {A\times B} & {[1]^\sharp\vee(\Db_n)_t} & {A\times B}
	\arrow[from=1-1, to=1-2]
	\arrow["\triangledown", from=1-1, to=2-1]
	\arrow[from=1-2, to=2-2]
	\arrow[from=1-3, to=1-4]
	\arrow[from=1-3, to=2-3]
	\arrow[from=1-4, to=2-4]
	\arrow[dashed, from=2-1, to=1-2]
	\arrow[from=2-1, to=2-2]
	\arrow[dashed, from=2-3, to=1-4]
	\arrow[from=2-3, to=2-4]
\end{tikzcd}\]
where $[1]^\sharp\to B$ factors through $[0]$ admits a unique lifting.

\item For any $x,y$ in $X$, the morphism:
$$\hom_X(x,y)\to \hom_B(g(x),g(y))\times \hom_A(f(x),f(y))$$
is a two-sided fibration.
\end{enumerate}
\end{prop}

\begin{proof}
Suppose first that $(f,g):X\to A\times B$ is a left cartesian fibration. It obviously fulfills the condition $(1)$. The lemma \ref{lemma: a partial biinitial} and its dual version imply that $(f,g)$ verifies the condition $(2)$. The condition $(3)$ is a consequence of proposition \ref{prop:cartesian fibration between arrow}.

Suppose now that $(f,g):X\to A\times B$ verifies the three conditions. The homotopy type of morphisms of $\ocatm_{/A\times B}$ having the unique right lifting property against $S$ is closed under composition and pushout. It then includes morphisms of the form 
\[\begin{tikzcd}
	{\Db_n^\flat\otimes\{0\}} & {\Db_n^\flat\otimes[1]^\sharp} \\
	& {A\times B}
	\arrow[from=1-1, to=1-2]
	\arrow[from=1-1, to=2-2]
	\arrow[from=1-2, to=2-2]
\end{tikzcd}\]
where $\Db_n^\flat\otimes[1]^\sharp\to A$ factors through $\Db_n^\flat$ as, by proposition \ref{prop:decomposition marked case}, this morphism is the composition of the following sequence of morphisms over $A\times B$:
\[\begin{tikzcd}
	{\{0\}} & {\Db_n} \\
	{[1]^\sharp} & {\Db_n\vee[1]^\sharp} & {[\Db_{n-1}\otimes\{1\},1]} \\
	{[\Db_{n-1}\otimes\{0\},1]} & \bullet & {[\Db_{n-1}\otimes[1]^\sharp,1]} \\
	{[1]^\sharp \vee \Db_n} & {\Db_n\otimes[1]^\sharp}
	\arrow[from=1-1, to=1-2]
	\arrow[from=1-1, to=2-1]
	\arrow[from=1-2, to=2-2]
	\arrow[from=2-1, to=2-2]
	\arrow["\lrcorner"{anchor=center, pos=0.125, rotate=180}, draw=none, from=2-2, to=1-1]
	\arrow[from=2-2, to=3-2]
	\arrow["\triangledown"', from=2-3, to=2-2]
	\arrow[from=2-3, to=3-3]
	\arrow[from=3-1, to=3-2]
	\arrow[from=3-1, to=4-1]
	\arrow["\lrcorner"{anchor=center, pos=0.125, rotate=90}, draw=none, from=3-2, to=2-3]
	\arrow[from=3-2, to=4-2]
	\arrow[from=3-3, to=3-2]
	\arrow[from=4-1, to=4-2]
	\arrow["\lrcorner"{anchor=center, pos=0.125, rotate=180}, draw=none, from=4-2, to=3-1]
\end{tikzcd}\]

We show similarly that $(f,g)$ is local with respect to morphisms
\[\begin{tikzcd}
	{(\Db_n)_t\otimes\{0\}} & {(\Db_n)_t\otimes[1]^\sharp} \\
	& {A\times B}
	\arrow[from=1-1, to=1-2]
	\arrow[from=1-1, to=2-2]
	\arrow[from=1-2, to=2-2]
\end{tikzcd}\]
where $(\Db_n)_t\otimes[1]^\sharp\to A$ factors through $(\Db_n)_t$. The span $(f,g)$ is then local with respect to $\I_{A,B}$, and we demonstrate similarly that it is local with respect to $\F_{A,B}$. The span $(f,g)$ is then a two-sided fibration.
\end{proof}

\subsection{$(\omega,1)$-double marked categories}

\begin{definition}
A simplicial object $D_\bullet:\Delta^{op}\to \ocatm$ is a \textit{$(\omega,1)$-double marked category} if for any integer $k$, the canonical map: $$D_k\to D_1 \times_{D_0}... \times_{D_0}D_1$$ is an equivalence. We denote by $\oDcatm$ the full subcategory of $\Fun(\Delta^{op},\ocatm)$ spanned by $(\omega,1)$-double marked categories.
\end{definition}

\begin{definition}
Let $K$ be a category and $C$ a marked $\omega$-category. We denote by $\langle C,K\rangle_\bullet:\Delta^{op}\to \ocatm$ the $(\omega,1)$-double marked category whose value on $n$ is given by $$\langle C,K\rangle_n:= C\times K_n$$ where $K_n:=\hom_{\cat}([n],K)$. This association defines a cocontinuous functor: $$\langle \uvar,\uvar\rangle:\cat \times \ocatm \to (\omega,1)\Dcatm$$
\end{definition}

\begin{remark}
\label{rem univ of langle omega1}
By some manipulation on localization of presheaf category, we can show that the adjunction induced by left Kan extension,
\[\begin{tikzcd}
	{\langle\uvar,\uvar\rangle_!:\widehat{\Delta\times \ocatm}} & { (\omega,1)\Dcatm}
	\arrow[""{name=0, anchor=center, inner sep=0}, shift left=2, from=1-1, to=1-2]
	\arrow[""{name=1, anchor=center, inner sep=0}, shift left=2, from=1-2, to=1-1]
	\arrow["\dashv"{anchor=center, rotate=-90}, draw=none, from=0, to=1]
\end{tikzcd}\]
is a localization. Any colimit preserving functor $ F:(\omega,1)\Dcatm\to X$ is then a restriction of the colimit preserving functor $\widehat{\Delta\times \ocatm}\to X$ sending $[n],C$ onto $F(\langle C,[n]\rangle)$.
\end{remark}

\begin{construction}
Let $D_\bullet$ be a $(\omega,1)$-double marked category and $f:A\to D_0$ a morphism. We denote by $f^*D_\bullet$ the $(\omega,1)$-double marked category whose value on $n$ fits in the pullback:
\[\begin{tikzcd}
	{f^*D_n} & {D_n} \\
	{\prod_{k\leq n }A} & {\prod_{k\leq n }D_{[0]}}
	\arrow[from=1-1, to=1-2]
	\arrow[from=1-1, to=2-1]
	\arrow["\lrcorner"{anchor=center, pos=0.125}, draw=none, from=1-1, to=2-2]
	\arrow[from=1-2, to=2-2]
	\arrow[from=2-1, to=2-2]
\end{tikzcd}\]
\end{construction}

\begin{definition}
A $(\omega,1)$-double marked category $D_\bullet$ is accompanied if 
\begin{enumerate}
\item $D_0 \sim D_0^\sharp$.
\item The canonical morphism $D_1 \to D_0^\sharp \times D_0^\sharp$ is a two-sided fibration.
\end{enumerate}
\end{definition}

\begin{remark}
If $D_\bullet$ is an accompanied $(\omega,1)$-double marked category and $f:A\to D_0$ is a morphism, then $f^*D$ is accompanied.
\end{remark}

\begin{lemma}
\label{lemma:characterization equivalence accompanied marked case.} A morphism $\phi:D_\bullet\to C_\bullet$ between accompanied $(\omega,1)$-double marked categories is an equivalence if and only if: \begin{enumerate} \item The induced morphism $C_0\to D_0$ is an equivalence. \item For any $x,y$ in $C_0$, the induced morphism $(C_1)_{x,y}\to (D_1)_{\phi(x),\phi(y)}$ is an equivalence. \end{enumerate}
\end{lemma}

\begin{proof}
This is a direct consequence of the fact that equivalences between two-sided fibrations are detected on fibers, as proven in corollary \ref{cor:morphism between is an equivalence when equivalence on fiberbifibrant case}.
\end{proof}

\subsection{Double initial morphisms}

\begin{definition}
We denote by $\DI$ the set of morphisms of shape $$\langle X\otimes\{0\},[0]\rangle\to \langle X\otimes[1]^\sharp,[0]\rangle~~~~~~\mbox{or} ~~~~~~\langle X,\{0\}\rangle\to \langle X\otimes[1]^\sharp,[1]\rangle$$ for $X$ being either $\Db_n$ or $(\Db_n)_t$ and $k$ an integer. A morphism is double initial if it is in $\widehat{\DI}$. Double initial morphisms are then stable under colimits, transfinite compositions, pushouts, left cancellation, and retracts.
\end{definition}

\begin{example}
\label{exemple:intiial and double initial}
If $i:C\to D$ is an initial morphism between $1$-categories, then for any marked $\omega$-category $X$, 
$$\langle X,C\rangle \to \langle X,D\rangle$$
is double initial.
\end{example}

\begin{lemma}
\label{lemma:the easiest example of double initial morphism} 
For any $(\omega,1)$-double marked categories $D_\bullet$, the morphism $$D_\bullet\times \langle [0],\{0\}\rangle\to D_\bullet\times \langle [0],[1]\rangle$$ is double initial.
\end{lemma}

\begin{proof}
As the cartesian product commutes with colimits, we can reduce to the case where $D_\bullet$ is of shape $\langle X,[n]\rangle$. As we have $\langle X,[n]\rangle\times \langle [0],\{0\}\rangle\sim \langle X, [n]\times\{0\}\rangle$ and $\langle X,[n]\rangle\times \langle [0],[1]\rangle\sim \langle X, [n]\times[1]\rangle$, this follows from the fact that $[n]\times\{0\}\to [n]\times[1]$ is an initial object between $1$-categories and from example \ref{exemple:intiial and double initial}.
\end{proof}

\begin{definition}
A \textit{left retract structure} for a morphism $i_\bullet:C_\bullet\to D_\bullet$ is the data of a \textit{retract} $r_\bullet:D_\bullet\to C_\bullet$, a \textit{deformation} $\psi_\bullet:D_\bullet\otimes \langle [0],[1]\rangle\to D$, and equivalences: $$r_\bullet i_\bullet\sim id_{C_\bullet}~~~~~(\psi_\bullet)_{|D_\bullet\times\{0\}}\sim i_\bullet r_\bullet ~~~~~(\psi_\bullet)_{|D_\bullet \times\{1\}}\sim id_{D_\bullet}~~~~~ (\psi_\bullet)_{|C_\bullet\times\langle [0],[1]\rangle}\sim i_\bullet \cst_{C_\bullet}$$ A morphism $i:C\to D$ over $A\times B$ is a \textit{left deformation retract} if it admits a left deformation retract structure. By abuse of notation, such data will just be denoted by $(i,r,\psi)$.
\end{definition}

\begin{prop}
\label{prop:left Gray deformation of double stuff retract are initial} Left deformation retracts are double initial.
\end{prop}

\begin{proof}
Let $ i:C_\bullet \to D_\bullet$ be a left deformation retract. We define $ D_\bullet\times_C\langle [0],[1]\rangle$ as the following pushout:
\[
\begin{tikzcd}
	{C_\bullet\times\langle [0],[1]\rangle} & {D_\bullet\times\langle [0],[1]\rangle} \\
	{C_\bullet} & {D_\bullet\times_{C_\bullet} \langle [0],[1]\rangle}
	\arrow[from=1-1, to=1-2]
	\arrow[from=1-1, to=2-1]
	\arrow[from=1-2, to=2-2]
	\arrow[from=2-1, to=2-2]
	\arrow["\lrcorner"{anchor=center, pos=0.125, rotate=180}, draw=none, from=2-2, to=1-1]
\end{tikzcd}
\]
Consider now the diagram:
\[
\begin{tikzcd}
	{C_\bullet\times\{0\}} & {C_\bullet\times\langle [0],[1]\rangle} \\
	& {C_\bullet}
	\arrow["{i_\bullet}", from=1-1, to=1-2]
	\arrow[equals, from=1-1, to=2-2]
	\arrow[from=1-2, to=2-2]
\end{tikzcd}
\]
It follows from stability by left cancellation and from lemma \ref{lemma:the easiest example of double initial morphism} that $ C_\bullet\times\langle [0],[1]\rangle\to C_\bullet$ is double initial. By stability by pushout and composition, and using once again example \ref{exe:the easiest example of initial and final morphism}, so is $ D_\bullet\times\{0\}\to D_\bullet\times_C\langle [0],[1]\rangle$.

Eventually, the diagram
\[
\begin{tikzcd}
	{C_\bullet} & {D_\bullet\times\{0\}} && {C_\bullet} \\
	{D_\bullet\times\{1\}} & {D_\bullet\times_C\langle [0],[1]\rangle} & {D_\bullet\times\langle [0],[1]\rangle} & {D_\bullet}
	\arrow["{i_\bullet}", from=1-1, to=1-2]
	\arrow["i"', from=1-1, to=2-1]
	\arrow["{r_\bullet}", from=1-2, to=1-4]
	\arrow[from=1-2, to=2-2]
	\arrow["i", from=1-4, to=2-4]
	\arrow[from=2-1, to=2-2]
	\arrow[from=2-2, to=2-3]
	\arrow["{\psi_\bullet}"', from=2-3, to=2-4]
\end{tikzcd}
\]
expresses $ i_\bullet$ as a retract of $ D_\bullet\times \{0\}\to D_\bullet\times_{C_\bullet}\langle [0],[1]\rangle$. The morphism $ i_\bullet$ is then double initial.
\end{proof}

\begin{construction}
Let $D_\bullet$ be an $(\omega,1)$-double category and $x$ an object of $D_0$. We denote by $(D_\bullet)_{x/}$ the $(\omega,1)$-double marked category defined by the formula: \[\begin{tikzcd} {(D_n)_{x/}} & {D_{n+1}} \\ {\{x\}} & {D_{\{0\}}} \arrow[from=1-1, to=1-2] \arrow[from=1-1, to=2-1] \arrow["\lrcorner"{anchor=center, pos=0.125}, draw=none, from=1-1, to=2-2] \arrow[from=1-2, to=2-2] \arrow[from=2-1, to=2-2] \end{tikzcd}\]
\end{construction}

\begin{lemma}
\label{lemma:double initial example} The morphism $\{x\} \to (D_\bullet)_{x/}$ is double initial.
\end{lemma}

\begin{proof}
We will show that this morphism admits a left deformation retract structure, which will conclude the proof by proposition \ref{prop:left Gray deformation of double stuff retract are initial}. The retraction is given by the unique morphism $(D_\bullet)_{x/}\to \{x\}$ and the deformation is $$\phi:(D_\bullet)_{x/}\times \langle[0],[1]\rangle\to (D_\bullet)_{x/}$$ induced for any $n$ by the morphism $$D_{n+1}\times [1]\to D_{n+1}$$ sending $u:\Db_{n+1}$ and $\phi:[n+1]\to [1]$ onto $d(\phi)^* u$ where $d(\phi):[n+1]\to [n+1]$ is the simplicial map whose value on $k$ is $k\times \phi(k)$.
\end{proof}

\subsection{Left fibration in $(\omega,1)$-double marked categories}

\begin{definition}
\label{defi:left fib in double} 
A \textit{left fibration} is a morphism having the unique left lifting property against $\DI$. Unfolding the definition, a morphism $E_\bullet\to D_\bullet$ is a left fibration if 
\begin{enumerate}
\item $E_0\to D_0$ is a left cartesian fibration.
\item for any integer $k$, the canonical square
\[\begin{tikzcd}
	{E_k} & {E_{\{0\}}} \\
	{D_k} & {D_{\{0\}}}
	\arrow[from=1-1, to=1-2]
	\arrow[from=1-1, to=2-1]
	\arrow["\lrcorner"{anchor=center, pos=0.125}, draw=none, from=1-1, to=2-2]
	\arrow[from=1-2, to=2-2]
	\arrow[from=2-1, to=2-2]
\end{tikzcd}\]
is cartesian.
\end{enumerate}

We denote by $\LFib(D_\bullet)$ the subcategory of $(\omega,1)\Dcatm_{/D_\bullet}$ spanned by left fibrations, and we then have a canonical adjunction
\begin{equation}
\label{eq:adjonction left fibration of double}
\begin{tikzcd}
	{\Fb:(\omega,1)\Dcatm_{/D_\bullet}} & {\LFib(D_\bullet):\iota}
	\arrow[""{name=0, anchor=center, inner sep=0}, shift left=2, from=1-1, to=1-2]
	\arrow[""{name=1, anchor=center, inner sep=0}, shift left=2, from=1-2, to=1-1]
	\arrow["\dashv"{anchor=center, rotate=-90}, draw=none, from=0, to=1]
\end{tikzcd}
\end{equation}
where the right adjoint is fully faithful. The functor $\Fb$ is the \textit{left fibrant replacement functor}.
\end{definition}

\begin{prop}
\label{prop:explicit factorization double case} Let $D_\bullet$ be an accompanied $(\omega,1)$-double category. The canonical morphism $$(D_\bullet)_{x/} \to D_\bullet $$ is the left fibrant replacement of $\{x\} \to D_\bullet$.
\end{prop}

\begin{proof}
The lemma \ref{lemma:double initial example} implies that $\{x\} \to (D_\bullet)_{x/}$ is double initial. Moreover, we have by construction a cartesian square:
\[\begin{tikzcd}
	{(D_0)_{x/}} & {D_1} \\
	{\{x\} \times D_0} & {D_0 \times D_0}
	\arrow[from=1-1, to=1-2]
	\arrow[from=1-1, to=2-1]
	\arrow["\lrcorner"{anchor=center, pos=0.125}, draw=none, from=1-1, to=2-2]
	\arrow[from=1-2, to=2-2]
	\arrow[from=2-1, to=2-2]
\end{tikzcd}\]
As $D_\bullet$ is accompanied, the right vertical map is a two-sided cartesian fibration, which implies that the left one is a left cartesian fibration. Moreover, as $D_\bullet$ verifies the Segal condition, we have a Cartesian square:
\[\begin{tikzcd}
	{(D_n)_{x/}} & {(D_{0})_{x/}} \\
	{D_n} & {D_0}
	\arrow[from=1-1, to=1-2]
	\arrow[from=1-1, to=2-1]
	\arrow["\lrcorner"{anchor=center, pos=0.125}, draw=none, from=1-1, to=2-2]
	\arrow[from=1-2, to=2-2]
	\arrow[from=2-1, to=2-2]
\end{tikzcd}\]
These two facts imply that $(D_\bullet)_{x/} \to D_\bullet$ is a left fibration of $(\omega,1)$-double marked categories, which concludes the proof.
\end{proof}

\subsection{Marked $(\omega,1)$-effectivity}

\begin{definition}
A \textit{$(\omega,1)$-filtration} is a morphism $A_0\to A_1$ between $\omega$-categories. A $(\omega,1)$-filtration is \textit{surjective} if it is a filtration such that $A_0\to A_1$ is $0$-surjective.

We denote by $\Filt_{\omega,1}$ the category $\Fun([1],\ocat)$ and $\Filts_{\omega,1}$ the full subcategory of surjective $(\omega,1)$-filtrations.
\end{definition}

\begin{construction}
\label{cons:cechnerve marked case} Given an $(\omega,1)$-filtration $A_\bullet$, its \textit{directed marked Čech nerve}, denoted by $\C_{\bullet}(A_0 \to A_1)$, is the $(\omega,1)$-double marked category whose value on $n$ fits in the cartesian square:
\[\begin{tikzcd}
	{\C_{n}(A_0 \to A_1)} & {(A_1^{\sharp})^{[n]^\sharp}} \\
	{\prod_{k\leq n}A_0^\sharp} & {\prod_{k\leq n}A_1^\sharp}
	\arrow[from=1-1, to=1-2]
	\arrow[from=1-1, to=2-1]
	\arrow["\lrcorner"{anchor=center, pos=0.125}, draw=none, from=1-1, to=2-2]
	\arrow[from=1-2, to=2-2]
	\arrow[from=2-1, to=2-2]
\end{tikzcd}\]
Given a $(\omega,1)$-double marked category $D_{\bullet}$, its \textit{realization} is the $(\omega,1)$-filtration $|D|_\bullet:=D_0^\natural\to |D|_1^{\natural}$ where $$|D|_1:= \coend_{\Delta} D_n \otimes [n]^\sharp.$$

These two functors fit into an adjunction:
\begin{equation}
\label{cons:adj betwn highly cech marked case}
\begin{tikzcd}
	{|\uvar|:(\omega,1)\Dcat_m} & {\Filt_{\omega,1}:\C}
	\arrow[""{name=0, anchor=center, inner sep=0}, shift left=2, from=1-1, to=1-2]
	\arrow[""{name=1, anchor=center, inner sep=0}, shift left=2, from=1-2, to=1-1]
	\arrow["\dashv"{anchor=center, rotate=-90}, draw=none, from=0, to=1]
\end{tikzcd}
\end{equation}
\end{construction}

\begin{remark}
By construction, we have 
$$\C_1(A_0\to A_1):= A_0 \dirtimes{A_1} A_0.$$
\end{remark}

\begin{remark}
\label{remark: an easy description of relation case 1} It can be useful to remark that the functor $|\uvar|$ is the restriction of the colimit preserving functor $$|\uvar|:\Fun(\Delta^{op},\ocatm)\to \Filt_{\omega,1}$$ defined by the formula: $$|\langle C,n\rangle| := \coprod_k C^\natural \otimes\{k\}\to (C\otimes[n]^\sharp)^\natural$$
\end{remark}

\begin{lemma}
\label{lemma:cech is accompagnated}
Let $A_\bullet$ be a $(\omega,1)$-filtration. The $(\omega,1)$-double marked category $\C_\bullet(A)$ is accompanied.
\end{lemma}

\begin{proof}
We have to check that for any $f:A\to B$, the morphism $\C_1(f)\to \C_0(f)\times \C_0(f)$ is a two-sided fibration, which directly follows from proposition \ref{prop:directed pullback i bifib}.
\end{proof}

\begin{lemma}
\label{lemma:realization is 0 surj} Let $D_\bullet$ be a $(\omega,1)$-double  marked category. The $(\omega,1)$-filtration $|D|_\bullet$ is surjective.
\end{lemma}

\begin{proof}
This directly follows from the remark \ref{remark: an easy description of relation case 1} as $0$-surjective morphisms are stable under colimits and as $ \coprod_k C^\natural \to (C\otimes[n]^\sharp)^\natural$ obviously induces an equivalence on objects.
\end{proof}

\begin{construction}
We can easily check that $ |\uvar|_1$ sends $ \langle X \otimes \{0\}, [0] \rangle \to \langle X \otimes [1]^\sharp, [0] \rangle$ and $ \langle X, \{0\} \rangle \to \langle X, [1] \rangle$ to $ X \otimes \{0\} \to X \otimes [1]^\sharp$. As this functor preserves colimits, it sends double initial morphisms to initial morphisms, and then induces, for any $(\omega,1)$-double marked category $D_\bullet$, a derived adjunction:

\begin{equation}
\label{eq: adjunction of double bar}
\begin{tikzcd}
	{|\uvar|_1^{D}:\LFib(D_\bullet)} & {\LCart(|D_\bullet|^\sharp):\N_{D}}
	\arrow[""{name=0, anchor=center, inner sep=0}, shift left=2, from=1-1, to=1-2]
	\arrow[""{name=1, anchor=center, inner sep=0}, shift left=2, from=1-2, to=1-1]
	\arrow["\dashv"{anchor=center, rotate=-90}, draw=none, from=0, to=1]
\end{tikzcd}
\end{equation}

Where the left adjoint is natural in $ D_\bullet:\oDcatm$ and the right adjoint is natural in $ D_\bullet:\oDcatm^{op}$.
\end{construction}

\begin{prop}
\label{prop:double bar is eq} 
For any $(\omega,1)$-double marked category $D_\bullet$, the functor $|\uvar|_1^{D}: \LFib(D_\bullet)\to \LCart(|D_\bullet|^\sharp)$ is an equivalence of categories.
\end{prop}

\begin{proof}
We will show that the right adjoint $\N_{D}$ is an equivalence of categories. The explicit description of the left fibration of a $(\omega,1)$-double marked category given in definition \ref{defi:left fib in double} implies that $$ \LFib(D_\bullet) \sim \cocoend_\Delta \Fun([n], \LCartc(D_n)).$$ The functor $\N_{D}$ is equivalent to the composite: $$\LCart(\coend_\Delta D_n \otimes [n]^\sharp)^\sharp \to \cocoend_\Delta \LCart((D_n \otimes [n]^\sharp)^\sharp) \to \cocoend_\Delta \Fun([n], \LCartc(D_n))$$ where the right arrow is induced by the family of right adjoints in the adjunctions: \[\begin{tikzcd} {\Fun([k], \LCartc(D_n))} & {\LCart((D_n \otimes [k]^\sharp)^\sharp)} \arrow[""{name=0, anchor=center, inner sep=0}, shift left=2, from=1-1, to=1-2] \arrow[""{name=1, anchor=center, inner sep=0}, shift left=2, from=1-2, to=1-1] \arrow["\dashv"{anchor=center, rotate=-90}, draw=none, from=0, to=1] \end{tikzcd}\] defined in \cite[construction 4.1.3.1]{loubaton2024categorical}. Moreover, proposition 4.1.3.28 of \cite{loubaton2024categorical} implies that these adjunctions are an adjoint equivalence, and corollary 4.1.2.17 of \cite{loubaton2024categorical} states that $\LCart(\uvar)$ sends colimits to limits. As a consequence, the functor $\N_{D}$ is an equivalence, which concludes the proof.
\end{proof}

\begin{lemma}
\label{lemma:equivalence on fiber} Let $D_\bullet$ be an accompanied $(\omega,1)$-double category, and $x,y$ two objects of $D_0$. The canonical morphism $$(D_1)_{(x,y)}^\natural\to \hom_{|D|_1^\natural}(|x|_1,|y|_1)$$ is an equivalence.
\end{lemma}

\begin{proof}
Proposition \ref{prop:explicit factorization double case} states that $(D_\bullet)_{x/}$ is the fibrant replacement of $\{x\}\to D_\bullet$, which implies it is then sent to the fibrant replacement of $\{|x|_1\}\to |D|_1^\sharp$. By \cite[remark 3.2.2.14]{loubaton2024categorical}, this fibrant replacement is $(|D|_1^\sharp)_{|x|_1/}\to |D|_1^\sharp$. 

We then have $$|(D_\bullet)_{x/}|_1^{D}\sim (|D|_1^\sharp)_{|x|_1/}$$
and by proposition \ref{prop:double bar is eq}, 
$$(D_\bullet)_{x/}\sim \N_{D}((|D|_1^\sharp)_{|x|_1/}).$$
By evaluating  on $[0]$, we then have a cartesian square:
\[\begin{tikzcd}
	{(D_1)_x} & {(|D_1|^{\sharp})^{[1]^{\sharp}}} \\
	{\{x\}\times D_0^\sharp} & {D_0^\sharp\times D^\sharp_0}
	\arrow[from=1-1, to=1-2]
	\arrow[from=1-1, to=2-1]
	\arrow["\lrcorner"{anchor=center, pos=0.125}, draw=none, from=1-1, to=2-2]
	\arrow[from=1-2, to=2-2]
	\arrow[from=2-1, to=2-2]
\end{tikzcd}\]
Taking the fiber over $y$, we get the desired equivalence.
\end{proof}

\begin{theorem}
\label{theo:eff of n1 dcat marked} 
The realization of a $(\omega,1)$-double category is always a surjective $(\omega,1)$-filtration, and the directed marked Čech nerve of a $(\omega,1)$-filtration is always an accompanied $(\omega,1)$-double marked category.

The realization and the directed marked Čech nerve induce inverse equivalences:
$$|\uvar|:(\omega,1)\Dcatmc\sim \Filts_{\omega,1}:\C.$$
\end{theorem}

\begin{proof}
The first assertion is lemmas \ref{lemma:realization is 0 surj} and \ref{lemma:cech is accompagnated}. It remains to show that the realization and the directed marked Čech nerve induce the given equivalences. Let $ D_\bullet$ be an accompanied category. By construction, we have an equivalence $ \C_0(|D_\bullet|) \sim D_0$. The unit of the adjunction induces a diagram :

 \[
\begin{tikzcd}
	{D_1} & {\C_1(|D_\bullet|)} \\
	{D_0} & {D_0}
	\arrow["{\nu_1}", from=1-1, to=1-2]
	\arrow[from=1-1, to=2-1]
	\arrow[from=1-1, to=2-2]
	\arrow[from=1-2, to=2-1]
	\arrow[from=1-2, to=2-2]
\end{tikzcd}
\]

The lemma \ref{lemma:equivalence on fiber} implies that the unit is an equivalence on fibers, and as these two spans are bifibrations, the corollary \ref{cor:morphism between is an equivalence when equivalence on fiberbifibrant case} implies that $ \nu_1$ is an equivalence. As $ D_\bullet$ and $ \C_\bullet(|D|)$ are Segal objects, this implies the desired equivalence.

Suppose now we are given a $ 0$-surjection $ p:C \to D$. Lemma \ref{lemma:realization is 0 surj} implies that $ |\C(p)|: C \to |\C(p)|_1^{\natural}$ is $ 0$-surjective, and so is $ \eta:|\C_\bullet(p)|_1^{\natural} \to D$ by left cancellation. Let $ x, y$ be two objects of $|\C_\bullet(p)|_1^\natural$, and let $ \tilde{x}, \tilde{y}$ be two objects of $C$ that are sent to $ x$ and $ y$. This induces a diagram :
\[\begin{tikzcd}
	& {(\C_1(p)^\natural)_{(x,y)}} & \\
	{\hom_{|\C_\bullet(p)|_1^\natural}(x,y)} && {\hom_D(\eta(x),\eta(y))}
	\arrow[from=1-2, to=2-1]
	\arrow["\sim", from=1-2, to=2-3]
	\arrow[from=2-1, to=2-3]
\end{tikzcd}\]
However, the lemma \ref{lemma:equivalence on fiber} implies that the upper left morphism is an equivalence, and so is the lower one. The morphism $ \eta:|\C_\bullet(p)|_1^{\natural} \to D$ is then fully faithful. As we already know that it is $ 0$-surjective, this implies that it is an equivalence. As a consequence, we get $ |\C_\bullet(p)| \sim p$.
\end{proof}

\section{Effectivity of accompanied $(n,m)$-double categories}
\label{section:Effectivity of accompanied $(n,m)$-double categories}
\subsection{$(n,m)$-Double categories}

\begin{definition}
\label{defi:of nmdouble cate}
Let $n,m\in \Nb\cup\{\omega\}$. A \textit{$(n,m)$-double category} is a functor $D_\bullet:\Theta_m^{op}\to \ncat$ such that for any globular sum $a$, the canonical morphism $D_a\to D_{\Sp_a}$ is an equivalence. We denote by $(n,m)\Dcat$ the category of $(n,m)$-double categories\footnote{We decided to call $(n,m)$-double categories the $m$-categories in the $n$-categories to stick with the convention of denoting $(n,m)$-categories the $m$-categories in $n$-homotopy types.}.
\end{definition}

\begin{definition}
Let $K$ be an $m$-category and $C$ an $n$-category. We denote by $\langle C,K\rangle_\bullet:\Theta_m^{op}\to \ncat$ the $(n,m)$-double category whose value on $a$ is given by: $$\langle C,K\rangle_a:= C\times K_a$$ This association defines a cocontinuous functor: $$\ncat \times m\mbox{-$\cat$} \to (n,m)\Dcat$$
\end{definition}

\begin{remark}
\label{rem univ of langle nm}
By some manipulation on localization of presheaf category, we can show that the adjunction induced by left Kan extension,
\[\begin{tikzcd}
	{\langle\uvar,\uvar\rangle_!:\widehat{\Theta_m\times \ncat}} & { (n,m)\Dcat}
	\arrow[""{name=0, anchor=center, inner sep=0}, shift left=2, from=1-1, to=1-2]
	\arrow[""{name=1, anchor=center, inner sep=0}, shift left=2, from=1-2, to=1-1]
	\arrow["\dashv"{anchor=center, rotate=-90}, draw=none, from=0, to=1]
\end{tikzcd}\]
is a localization. Any colimit preserving functor $ F: (n,m)\Dcat\to X$ is then a restriction of the colimit preserving functor $\widehat{\Theta_m\times \ncat}\to X$ sending $a,C$ onto $F(\langle C,a\rangle)$.
\end{remark}

\begin{definition}
Let $ D_\bullet$ be a $(n,m)$-double category and $(x,y)$ two objects of $ D_{0}$. We denote by $ \hom^v_{D}(x,y)_\bullet$ the $(n-1,m)$-double category whose value on $ a:\Theta$ fits in the pullback:
\[
\begin{tikzcd}
	{\hom^v_{D}(x,y)_\bullet} && {D_a^{[1]}} \\
	{\{x\}\times\{y\}} & {D_0\times D_0} & {D_a\times D_a}
	\arrow[from=1-1, to=1-3]
	\arrow[from=1-1, to=2-1]
	\arrow["\lrcorner"{anchor=center, pos=0.125}, draw=none, from=1-1, to=2-2]
	\arrow[from=1-3, to=2-3]
	\arrow[from=2-1, to=2-2]
	\arrow[from=2-2, to=2-3]
\end{tikzcd}
\]
and $ \hom^h_{D}(x,y)_\bullet$ the $(n,m-1)$-double category whose value on $ a:\Theta$ fits in the pullback:
\[
\begin{tikzcd}
	{\hom^h_{D}(x,y)_a} & {D_{[a,1]}} \\
	{\{x\}\times\{y\}} & {D_0\times D_0}
	\arrow[from=1-1, to=1-2]
	\arrow[from=1-1, to=2-1]
	\arrow["\lrcorner"{anchor=center, pos=0.125}, draw=none, from=1-1, to=2-2]
	\arrow[from=1-2, to=2-2]
	\arrow[from=2-1, to=2-2]
\end{tikzcd}
\]

Eventually, given $ k\leq n$ and $ l\leq m$, we define $ \tau_{k,l}D_\bullet$ as the $(k,l)$-double category whose value on $ a\in \Theta_l$ is 
$$ \tau_{k,l}D_a := \tau_k(D_a). $$
\end{definition}

\subsection{Accompanied $(n,1)$-double categories}

We will start by recalling some results on $(1,1)$-double categories proven by Ruit in \cite{Ruit2025}.

\begin{notation}
Let $D_\bullet$ be a $(1,1)$-double category. The $1$-cell of $D_0$ will be called a \textit{vertical $1$-cell} of $D_\bullet$, and we will use the notation $f:a\to b$ to signify that $\pi_0^{-,v}(g)\sim a$ and $\pi_0^{+,v}(g)\sim b$. The objects of $D_1$ are called \textit{horizontal $1$-cells}. When we write $f:a\proto b$, this means that $f$ is a vertical $1$-cell such that $\pi_0^{-,h}(f)\sim a$ and $\pi_0^{+,h}(f)\sim b$. The $1$-cells of $D_1$ will be pictured as squares, as shown below:
\[
\begin{tikzcd}
	\bullet & \bullet \\
	\bullet & \bullet
	\arrow["{\pi_0^{-,v}(\alpha)}", from=1-1, to=1-2]
	\arrow["{\pi_0^{-,h}(\alpha)}"', from=1-1, to=2-1]
	\arrow["\alpha"{description}, draw=none, from=1-1, to=2-2]
	\arrow["{\pi_0^{+,h}(\alpha)}", from=1-2, to=2-2]
	\arrow["{\pi_0^{+,v}(\alpha)}"', from=2-1, to=2-2]
\end{tikzcd}
\]
Finally, we will denote by $\uvar\circ^v_0\uvar$ the $0$-composition in each of the $1$-categories $C_n$, and $\uvar\circ^h_0\uvar$ the canonical functor: $$D_1\times_{D_0}D_1\sim D_2\xrightarrow{D_{d^1}}D_1$$
\end{notation}

\begin{definition}
Let $ D_\bullet$ be a $(1,1)$-double category.

A vertical $1$-cell $f:a\to b$ is \textit{companionable} if there exists a horizontal $1$-cell $\overline{f}:a\proto b$ called \textit{a companion of $f$}, two squares
\[\begin{tikzcd}
	a & a && a & b \\
	a & b && b & b
	\arrow["\shortmid"{marking}, equals, from=1-1, to=1-2]
	\arrow[equals, from=1-1, to=2-1]
	\arrow["{\psi_f}"{description}, draw=none, from=1-1, to=2-2]
	\arrow["f", from=1-2, to=2-2]
	\arrow["{\overline{f}}", "\shortmid"{marking}, from=1-4, to=1-5]
	\arrow["f"', from=1-4, to=2-4]
	\arrow["{\phi_f}"{description}, draw=none, from=1-4, to=2-5]
	\arrow[equals, from=1-5, to=2-5]
	\arrow["{\overline{f}}"', "\shortmid"{marking}, from=2-1, to=2-2]
	\arrow["{|}"{marking, allow upside down}, "\shortmid"{marking}, equals, from=2-4, to=2-5]
\end{tikzcd}\]
that are called respectively the \textit{companionship unit} and the \textit{companionship counit}, and equivalences:
\[
\begin{tikzcd}
	a & a & b & a & b \\
	a & b & b & a & b
	\arrow["\shortmid"{marking}, equals, from=1-1, to=1-2]
	\arrow[equals, from=1-1, to=2-1]
	\arrow["{\overline{f}}", "\shortmid"{marking}, from=1-2, to=1-3]
	\arrow["\psi_f"{description}, draw=none, from=1-2, to=2-1]
	\arrow["f"', from=1-2, to=2-2]
	\arrow["\phi_f"{description}, draw=none, from=1-3, to=2-2]
	\arrow[""{name=0, anchor=center, inner sep=0}, equals, from=1-3, to=2-3]
	\arrow["{\overline{f}}", "\shortmid"{marking}, from=1-4, to=1-5]
	\arrow[""{name=1, anchor=center, inner sep=0}, equals, from=1-4, to=2-4]
	\arrow["{Ib^v_{\overline{f}}}"{description}, draw=none, from=1-4, to=2-5]
	\arrow[equals, from=1-5, to=2-5]
	\arrow["{\overline{f}}"', "\shortmid"{marking}, from=2-1, to=2-2]
	\arrow["\shortmid"{marking}, equals, from=2-2, to=2-3]
	\arrow["{\overline{f}}"', "\shortmid"{marking}, from=2-4, to=2-5]
	\arrow["\sim"{description}, draw=none, from=0, to=1]
\end{tikzcd}
\]
and 
\[
\begin{tikzcd}
	a & a \\
	a & b & a & a \\
	b & b & b & b
	\arrow["\shortmid"{marking}, equals, from=1-1, to=1-2]
	\arrow[equals, from=1-1, to=2-1]
	\arrow["f", from=1-2, to=2-2]
	\arrow["{\overline{f}}", "\shortmid"{marking}, from=2-1, to=2-2]
	\arrow["f"', from=2-1, to=3-1]
	\arrow["\psi_f"{description}, draw=none, from=2-2, to=1-1]
	\arrow["\phi_f"{description}, draw=none, from=2-2, to=3-1]
	\arrow[""{name=0, anchor=center, inner sep=0}, equals, from=2-2, to=3-2]
	\arrow["\shortmid"{marking}, equals, from=2-3, to=2-4]
	\arrow[""{name=1, anchor=center, inner sep=0}, "f"', from=2-3, to=3-3]
	\arrow["{\Ib^h_f}"{description}, draw=none, from=2-3, to=3-4]
	\arrow["f", from=2-4, to=3-4]
	\arrow["\shortmid"{marking}, equals, from=3-1, to=3-2]
	\arrow["\shortmid"{marking}, equals, from=3-3, to=3-4]
	\arrow["\sim"{description}, draw=none, from=0, to=1]
\end{tikzcd}
\]
A $(1,1)$-double category $D_\bullet$ is \textit{accompanied} if every vertical cell is companionable. We denote by $(1,1)\Dcatc$ the full subcategory of $(1,1)\Dcat$ whose objects are accompanied $(1,1)$-double categories.
\end{definition}

\begin{prop}[Ruit] \label{prop:unicity of companion 1} Let $f:a\to b$ be a companionable vertical $1$-cell. The homotopy type of companionship units of $f$ as well as the homotopy type of companionship units of $f$ is contractible.
\end{prop}

\begin{proof}
This is \cite[theorem 2.5.13]{Ruit2025}.
\end{proof}

\begin{lemma}
\label{lemma:companion stable by comp}
Let $D_\bullet$ be a $(1,1)$-double category, and $f:a\to b$, $g:b\to x$ two companionable $1$-cells. The $1$-cell $g\circ_0^v f$ is then companionable.
\end{lemma}

\begin{proof}

We leave it to the reader to check that a companion is given by $\overline{f}\circ_0^v\overline{g}$, and the companionship unit and counit are given by the composite:
\[\begin{tikzcd}
	a & a & a && a & b & c \\
	a & b & b && b & b & c \\
	a & b & c && c & c & c
	\arrow["\shortmid"{marking}, equals, from=1-1, to=1-2]
	\arrow[equals, from=1-1, to=2-1]
	\arrow["{\psi_f}"{description}, draw=none, from=1-1, to=2-2]
	\arrow["\shortmid"{marking}, equals, from=1-2, to=1-3]
	\arrow["f", from=1-2, to=2-2]
	\arrow["{\Ib^h_{\overline{f}}}"{description}, draw=none, from=1-2, to=2-3]
	\arrow["f", from=1-3, to=2-3]
	\arrow["{\overline{f}}", "\shortmid"{marking}, from=1-5, to=1-6]
	\arrow["f"', from=1-5, to=2-5]
	\arrow["{\phi_f}"{description}, draw=none, from=1-5, to=2-6]
	\arrow["{\overline{g}}", "\shortmid"{marking}, from=1-6, to=1-7]
	\arrow[equals, from=1-6, to=2-6]
	\arrow["{\Ib^v_g}"{description}, draw=none, from=1-6, to=2-7]
	\arrow[equals, from=1-7, to=2-7]
	\arrow["{\overline{f}}"', "\shortmid"{marking}, from=2-1, to=2-2]
	\arrow[equals, from=2-1, to=3-1]
	\arrow["{\Ib^v_f}"{description}, draw=none, from=2-1, to=3-2]
	\arrow["\shortmid"{marking}, equals, from=2-2, to=2-3]
	\arrow[equals, from=2-2, to=3-2]
	\arrow["{\phi_f}"{description}, draw=none, from=2-2, to=3-3]
	\arrow["g", from=2-3, to=3-3]
	\arrow["\shortmid"{marking}, equals, from=2-5, to=2-6]
	\arrow["g"', from=2-5, to=3-5]
	\arrow["{\Ib^h_{\overline{g}}}"{description}, draw=none, from=2-5, to=3-6]
	\arrow["{\overline{g}}", "\shortmid"{marking}, from=2-6, to=2-7]
	\arrow["g"', from=2-6, to=3-6]
	\arrow["{\psi_f}"{description}, draw=none, from=2-6, to=3-7]
	\arrow[equals, from=2-7, to=3-7]
	\arrow["{\overline{f}}"', "\shortmid"{marking}, from=3-1, to=3-2]
	\arrow["{\overline{g}}"', "\shortmid"{marking}, from=3-2, to=3-3]
	\arrow["\shortmid"{marking}, equals, from=3-5, to=3-6]
	\arrow["\shortmid"{marking}, equals, from=3-6, to=3-7]
\end{tikzcd}\]
\end{proof}

\begin{definition}
\label{defi:cartesian cell}
Let $D_\bullet$ be a $(1,1)$-double category. A square $\alpha:s\proto t$ in $D_1$ is \textit{bicartesian} if $s$ and $t$ are companionable, and if it is of the shape:
\[\begin{tikzcd}
	a & b & c & c \\
	b & b & c & d
	\arrow["{\overline{s}}", "\shortmid"{marking}, from=1-1, to=1-2]
	\arrow["s"', from=1-1, to=2-1]
	\arrow["{\phi_s}"{description}, draw=none, from=1-1, to=2-2]
	\arrow[""{name=0, anchor=center, inner sep=0}, "f", "\shortmid"{marking}, from=1-2, to=1-3]
	\arrow[equals, from=1-2, to=2-2]
	\arrow["\shortmid"{marking}, equals, from=1-3, to=1-4]
	\arrow[equals, from=1-3, to=2-3]
	\arrow["{\psi_t}"{description}, draw=none, from=1-3, to=2-4]
	\arrow["t", from=1-4, to=2-4]
	\arrow["\shortmid"{marking}, equals, from=2-1, to=2-2]
	\arrow[""{name=1, anchor=center, inner sep=0}, "f"', "\shortmid"{marking}, from=2-2, to=2-3]
	\arrow["{\overline{t}}"', "\shortmid"{marking}, from=2-3, to=2-4]
	\arrow["{\Ib _vf}"{description}, draw=none, from=0, to=1]
\end{tikzcd}\]
for some horizontal cell $f$.
\end{definition}

\begin{lemma} \label{lemma:stability of cartesian cell by comp1} Let $D_\bullet$ be a $(1,1)$-double category. Bicartesian cells are stable under horizontal and vertical compositions and contain horizontal identities of companionable $1$-cells.
\end{lemma}

\begin{proof}
The stability by horizontal composition and the fact that identities of companionable $1$-cells are bicartesian cells is straightforward. The stability by horizontal composition is easily verified using the explicit description of the (co)unit companionship of a composite of vertical arrows given in the proof of lemma \ref{lemma:companion stable by comp}.
\end{proof}

\begin{definition}
Let $D_\bullet$ be a $(1,1)$-double category. We denote by $t^cD_0$ the marking on $D_0$ consisting of companionable $1$-cells. We denote by $t^cD_{1}$ the marking on $D_1$ consisting of bicartesian cells. Eventually, we denote by $t^cD_n$ the marking on $D_{[n]}$ defined as $$t^cD_{[n]}:= t^cD_1\times_{t^cD_0}t^cD_1\times...\times_{t^cD_0}t^cD_1.$$ As lemma \ref{lemma:stability of cartesian cell by comp1} states that bicartesian cells are stable under horizontal composition, this implies that $(D_{[n]},t^cD_{[n]})$ is a simplicial marked category, which we will denote as $(D,t^cD)_\bullet$.
\end{definition}

\begin{theorem}[Ruit] \label{theo:jaco} Let $D_\bullet$ be a $(1,1)$-double category. The morphism $$(D,t^cD)_1\to(D,t^cD)_0\times(D,t^cD)_0$$ is a two-sided fibration.
\end{theorem}

\begin{proof}
The theorem 2.6.8 of \cite{Ruit2025} and its dual version imply that a cocartesian lift of $(f:a\proto b,\Ib^v_b:b\to b,t:c\to d)$ is given by the composition of 
\[\begin{tikzcd}
	b & c & c \\
	b & c & d
	\arrow[""{name=0, anchor=center, inner sep=0}, "f", "\shortmid"{marking}, from=1-1, to=1-2]
	\arrow[equals, from=1-1, to=2-1]
	\arrow["\shortmid"{marking}, equals, from=1-2, to=1-3]
	\arrow[equals, from=1-2, to=2-2]
	\arrow["{\psi_t}"{description}, draw=none, from=1-2, to=2-3]
	\arrow["t", from=1-3, to=2-3]
	\arrow[""{name=1, anchor=center, inner sep=0}, "f"', "\shortmid"{marking}, from=2-1, to=2-2]
	\arrow["{\overline{t}}"', "\shortmid"{marking}, from=2-2, to=2-3]
	\arrow["{\Ib _vf}"{description}, draw=none, from=0, to=1]
\end{tikzcd}\]
and a cartesian lift of $(f:a\proto b,s:a\to b\Ib^v_c:c\to c)$ is given by the composition of
\[\begin{tikzcd}
	a & b & c \\
	b & b & c
	\arrow["{\overline{s}}", "\shortmid"{marking}, from=1-1, to=1-2]
	\arrow["s"', from=1-1, to=2-1]
	\arrow["{\phi_s}"{description}, draw=none, from=1-1, to=2-2]
	\arrow[""{name=0, anchor=center, inner sep=0}, "f", "\shortmid"{marking}, from=1-2, to=1-3]
	\arrow[equals, from=1-2, to=2-2]
	\arrow[equals, from=1-3, to=2-3]
	\arrow["\shortmid"{marking}, equals, from=2-1, to=2-2]
	\arrow[""{name=1, anchor=center, inner sep=0}, "f"', "\shortmid"{marking}, from=2-2, to=2-3]
	\arrow["{\Ib _vf}"{description}, draw=none, from=0, to=1]
\end{tikzcd}\]
Eventually, bicartesian cells in the sense of \ref{defi:cartesian cell} exactly correspond to vertical compositions of these cocartesian and cartesian lifts.
\end{proof}

The goal of this section is to generalize the previous theorem to a $(n,1)$-double category.

\begin{notation}
Let $D_\bullet$ be a $(n,1)$-double category. We will denote by $\pi_k^{-,v},\pi_k^{+,v}:(C_l)_m\to (C_l)_k$ the functors sending a $m$-cell of $C_l$ to its $k$-source and target. We will denote by $\pi_0^{-,v},\pi_0^{+,v}:C_1\to C_0$ the two functors induced by the simplicial map $\{0\}\to [1]$ and $\{1\}\to [1]$. When we write $f:a\proto b$, this means that $f$ is a cell of $C_1$ such that $\pi_0^{-,h}(f)\sim a$ and $\pi_0^{+,h}(f)\sim b$. Finally, we will denote by $\uvar\circ^v_k\uvar$ the $k$-compositions in each of the $n$-categories $C_m$, and $\uvar\circ^h_0\uvar$ the canonical functor: $$D_1\times_{D_0}D_1\sim D_2\xrightarrow{D_{d^1}}D_1$$
\end{notation}

\begin{definition}
\label{defi:accompange}
Let $n\in \Nb\cup\{\omega\}$, and let $D_\bullet$ be a $(n,1)$-double category. We define by induction on $0<k<n$ the notion of \textit{companionable $k$-cell} of $D_0$. A $1$-cell $\alpha$ in $D_0$ is \textit{companionable} if the corresponding vertical $1$-cell of $\tau_{1,1}D_\bullet$ is companionable. A $k$-cell $\alpha$ in $D_0$ is \textit{companionable} if the corresponding $(k-1)$-cell of $\hom_D^v(\pi_0^{-,v}(\alpha),\pi_0^{-,v}(\alpha))_{\bullet^{op}}$ is companionable.

We denote by $(n,1)\Dcatc$ the full subcategory of $(n,1)\Dcat$ whose objects are accompanied $(n,1)$-double categories.
\end{definition}

\begin{remark}
\label{rem:accom closed}
Let $C_\bullet\to D_\bullet$ be a morphism that is an equivalence when evaluated on $[0]$. As companion structures are preserved by any morphism of $(n,1)$-double categories, if $C_\bullet$ is accompanied, so is $D_\bullet$.
\end{remark}

\begin{remark}
Let $n\in \mathbb{N}\cup\{\omega\}$ and let $D_\bullet$ be a $(n,1)$-double category. If $n<\omega$, the $(n,1)$-double category $D_\bullet$ is accompanied if and only if 
\begin{enumerate}
\item $\tau_{1,1}D_\bullet$ is accompanied.
\item for any $x,y\in D_0$, the $(n-1,1)$-double category $\hom^v_D(x,y)_{\bullet^{op}}$ is accompanied.
\end{enumerate}
If $n=\omega$, the $(n,1)$-double category $D_\bullet$ is accompanied if and only if for any $k>0$, $\tau_{k,1}(D_\bullet)$ is accompanied.
\end{remark}

\begin{lemma}
\label{lemma:stability of companion by comp2}Let $D_\bullet$ be a $(n,1)$-double category. Let $\alpha$ be a $k$-cell of $D_0$ with $k>1$, and $f:a\to b$ a $1$-cell of $D$ such that $f$ and $\alpha$ are composable. Then, if $\alpha$ is companionable, so is $\alpha\circ^v_0 f$.
\end{lemma}

\begin{proof}
The $1$-cell $f$ induces a functor: $$f_!:\hom_D^v(\pi_0^{-,v} \alpha,a)_{\bullet^{op}}\to \hom_D^v(\pi_0^{-,v} \alpha,b)_{\bullet^{op}}$$ that sends $\alpha$ to $\alpha\circ_0 f$. As a morphism of $(n,1)$-double categories preserves companion structure, this concludes the proof.
\end{proof}

\begin{prop}
\label{prop:comp stable by comp} 
Companionable vertical cells are closed under composition.
\end{prop}

\begin{proof}
Remark first that units obviously are companionable. We will then use the remark \ref{remark:other description of cloture by composition} and its notation to demonstrate the result. We will proceed by induction. The case $k=0$ follows from lemmas \ref{lemma:stability of cartesian cell by comp1} and \ref{lemma:stability of companion by comp2}. The recursive step is demonstrated using the fact that given a $(n,1)$-double category $D_\bullet$ and an integer $m>0$, a $m$-cell $\alpha$ in $D_0$ is companionable if and only if the corresponding $(n-1)$-cell of $\hom^v_{D}(\pi_0^{-,v}\alpha,\pi_0^{+,v}\alpha)_{\bullet^{op}}$ is companionable.
\end{proof}

\begin{definition}
Let $ D_\bullet$ be an $(n,1)$-double category, and let $ f:a \proto b$ and $ f':a' \proto b'$ be two objects of $ D_1$. We denote by $ \hom^v_D(f,g)_\bullet$ the $(n-1,1)$-double category whose value on $ k$ is: \[
\hom^v_D(f,g)_k := \coprod_{\phi:[k]\to [1]} \hom_{D_k}( D_{\phi}(f), D_{\phi}(f')).
\]
\end{definition}

\begin{remark}
If $D_\bullet$ is an $(n,1)$-double category and $f:a \proto b$ and $f':a' \proto b'$ are two objects of $D_1$, then $$\hom^v_D(f,f')_0 := \hom_{D_0}(a,b) \coprod \hom_{D_0}(a',b')$$ and $$\hom^v_D(f,g)_1 := \hom_{D_1}(\Ib^h_a,\Ib^h_b) \coprod \hom_{D_1}(f,f') \coprod \hom_{D_1}(\Ib^h_{a'},\Ib^h_{b'}).$$
\end{remark}

\begin{definition}
Let $n\in \Nb\cup\{\omega\}$, and let $D_\bullet$ be a $(n,1)$-double category. We define by induction on $0<k<n$ the notion of \textit{bicartesian $k$-cell} of $D_1$. A $1$-cell $\alpha$ in $D_1$ is \textit{bicartesian} if the corresponding square of $\tau_{1,1}D_\bullet$ is bicartesian. A $k$-cell $\alpha$ in $D_1$ is \textit{bicartesian} if the corresponding $(k-1)$-cell of $\hom^v_D(\pi_0^{-,v}(\alpha),\pi_0^{-,v}(\alpha))_{\bullet^{op}}$ is bicartesian.
\end{definition}

\begin{lemma}
\label{lemma:stability of cartesian cell by comp2} Let $D_\bullet$ be a $(n,1)$-double category. Let $\alpha$ be a $k$-cell of $D_1$ for $k>1$, and $\beta:f\to g$ a $1$-cell of $D_1$ such that $\beta$ and $\alpha$ are $0$-composable. Then, if $\alpha$ is a bicartesian cell, so is $\alpha\circ^v_0 \beta$.
\end{lemma}

\begin{proof}
The $1$-cell $f$ induces a functor: $$\beta_!:\hom_D^v(\pi_0^{-,v} \alpha,f)_{\bullet^{op}}\to \hom_D^v(\pi_0^{-,v} \alpha,g)_{\bullet^{op}}$$ that sends $\alpha$ to $\alpha\circ_0 \beta$. As a morphism of $(n,1)$-double categories preserves bicartesian cell, this concludes the proof.
\end{proof}

\begin{prop}
\label{prop:cartesian cells are a marking} Let $D_\bullet$ be a $(n,1)$-double category. Bicartesian cells are stable under vertical and horizontal composition and contain horizontal units of companionable cells.
\end{prop}

\begin{proof}
The stability by horizontal composition of bicartesian cells and the fact that horizontal and vertical units of companionable cells are bicartesian cells is straightforward.

We will then use the remark \ref{remark:other description of cloture by composition} and its notation to demonstrate the stability by vertical composition. We will proceed by induction on $k$. The case $k=0$ follows from lemmas \ref{lemma:stability of cartesian cell by comp1} and \ref{lemma:stability of cartesian cell by comp2}. The recursive step is demonstrated using the fact that given a $(n,1)$-double category $D_\bullet$ and an integer $m>0$, a $m$-cell $\alpha$ in $D_1$ is a bicartesian cell if and only if the corresponding $(n-1)$-cell of $\hom^v_{D}(\pi_0^{-,v}\alpha,\pi_0^{+,v}\alpha)_\bullet$ is a bicartesian cell.
\end{proof}

\begin{definition}
Let $D_\bullet$ be a $(n,1)$-double category. We denote by $t^cD_0$ the marking on $D_0$ consisting of companionable cells. We denote by $t^cD_{1}$ the marking on $D_1$ consisting of bicartesian cells. Eventually, we denote by $t^cD_n$ the marking on $D_{[n]}$ defined as $$t^cD_{[n]}:= t^cD_1\times_{t^cD_0}t^cD_1\times...\times_{t^cD_0}t^cD_1.$$ As proposition \ref{prop:cartesian cells are a marking} states that bicartesian cells are stable under horizontal composition, this implies that $(D_{[n]},t^cD_{[n]})$ is a simplicial marked $n$-category, which we will denote as $(D,t^cD)_\bullet$.
\end{definition}

\begin{remark}
\label{rem:foncteur preserves accompanied marking} Any functor $f:C_\bullet\to D_\bullet$ verifies  $f(tC_\bullet)\subset tD_\bullet$.
\end{remark}

\begin{definition}
Let $1 < k \leq n$ be two integers. The functor $[\Db_{k-1},\uvar]:\Delta\to \Theta_n$ sending $[n]$ onto $[\Db_{k-1},n]$ induces a morphism $$\intw_k:\ncat\to \cat$$ sending an $n$-category $C$ to the category whose objects are the same and morphisms are $k$-cells.
\end{definition}

\begin{lemma}
\label{lemma:lifting of db} Let $1 \leq k \leq n < \omega$, and let $D_\bullet$ be a $(n,1)$-double category. Any diagram of shape \[ \begin{tikzcd} {\Db_k^\flat} & {(C_1,t^{c}C_1)} \\ {[1]^\sharp \vee \Db_k^\flat} & {(C_0,t^{c}C_0) \times (C_0,t^{c}C_0)} \arrow[from=1-1, to=1-2] \arrow[from=1-1, to=2-1] \arrow[from=1-2, to=2-2] \arrow[dashed, from=2-1, to=1-2] \arrow[from=2-1, to=2-2] \end{tikzcd} \] where the first projection $[1] \vee \Db_n \to C_0$ factors through $\Db_n$, admits a unique lift.
\end{lemma}

\begin{proof}
We have a canonical functor $\tau_{1,1}D_\bullet\to \intw_kD_\bullet$. As functors of $(n,1)$-double categories preserve companionable cells and bicartesian cells, this implies that a lift in the diagram of the statement of the lemma corresponds to a lift in the diagram:
\[\begin{tikzcd}
	{[1]^\flat} & {(\intw_kC_1,t^{c}\intw_kC_1)} \\
	{[1]^\sharp\vee[1]^\flat} & {(\intw_kC_0,t^{c}\intw_kC_0)\times (\intw_kC_0,t^{c}\intw_kC_0)}
	\arrow[from=1-1, to=1-2]
	\arrow[from=1-1, to=2-1]
	\arrow[from=1-2, to=2-2]
	\arrow[dashed, from=2-1, to=1-2]
	\arrow[from=2-1, to=2-2]
\end{tikzcd}\]
However, the theorem \ref{theo:jaco} implies that the left-hand morphism is a two-sided fibration, and by proposition \ref{prop: another characterization}, lifts exist and are unique.
\end{proof}

\begin{lemma}
\label{lemma:lifting of db2} Let $1\leq k\leq n<\omega$, and let $D_\bullet$ be a $(n,1)$-double category. Any diagram of shape
\[\begin{tikzcd}
	{(\Db_k)_t} & {(C_1,t^{c}C_1)} \\
	{[1]^\sharp\vee(\Db_k)_t} & {(C_0,t^{c}C_0)\times (C_0,t^{c}C_0)}
	\arrow[from=1-1, to=1-2]
	\arrow[from=1-1, to=2-1]
	\arrow[from=1-2, to=2-2]
	\arrow[dashed, from=2-1, to=1-2]
	\arrow[from=2-1, to=2-2]
\end{tikzcd}\]
where the first projection $[1]\vee\Db_n\to C_0$ factors through $\Db_n$, admits a unique lift.
\end{lemma}

\begin{proof}
Suppose first that $k=1$. In this case, up to replacing $D_\bullet$ by $\tau_{1,1}D_\bullet$, we can assume that $D_\bullet$ is a $(1,1)$-double category and then apply the theorem \ref{theo:jaco}. 

Suppose now that $k>1$. We denote by $\intw'_kD_n$ the subcategory of $\intw_kD_n$ whose objects are the same and whose morphisms are $n$-cells of $D_n$ that belong to $t^{c}D_n$. Note in particular that the canonical functor $\tau_{1,1}D_\bullet\to \intw_kD_\bullet$ factors through $\intw'_kD_n$. As functors of $(n,1)$-double categories preserve comparable cells and bicartesian cells, this implies that a lift in the diagram of the statement of the lemma corresponds to a lift in the diagram:
\[\begin{tikzcd}
 {[1]^\flat} & {(\intw'_kC_1,t^{c}\intw'_kC_1)} \\
 {[1]^\sharp\vee[1]^\flat} & {(\intw'_kC_0,t^{c}\intw'_kC_0)\times (\intw'_kC_0,t^{c}\intw'_kC_0)}
 \arrow[from=1-1, to=1-2]
 \arrow[from=1-1, to=2-1]
 \arrow[from=1-2, to=2-2]
 \arrow[dashed, from=2-1, to=1-2]
 \arrow[from=2-1, to=2-2]
\end{tikzcd}\]
However, the theorem \ref{theo:jaco} implies that the left-hand morphism is a two-sided fibration, and by proposition \ref{prop: another characterization}, lifts exist and are unique.
\end{proof}

\begin{lemma}
\label{lemma: accomoagnated implies marjed} Let $ n \in \mathbb{N} \cup \{\omega\}$ and let $ D_\bullet$ be an $(n,1)$-double category. The morphism $$(D, t^cD)_1 \to (D, t^cD)_0 \times (D, t^cD)_0$$ is a two-sided fibration.
\end{lemma}

\begin{proof}
We will proceed by induction on $n$. The case $n=1$ is a consequence of theorem \ref{theo:jaco}. Suppose then the result is true at the stage $n-1<\omega$, and let $D_\bullet$ be a $(n,1)$-double category. The theorem \ref{theo:jaco} applied to $\tau_{1,1}D_\bullet$ implies that the span of the statement then verifies the condition $(1)$ of proposition \ref{prop: another characterization}. The condition $(2)$ is implied by lemmas \ref{lemma:lifting of db} and \ref{lemma:lifting of db2} and their dual versions. Eventually, let $f:a\proto b$ and $f':a'\proto b'$ be two objects of $D_1$. The induction hypothesis applied to the $(n-1,1)$-double categories $\hom^v_{D}(f,g)_{\bullet^{op}}$ and the cartesian square
\[\begin{tikzcd}
	{\hom_{(D,t^cD)_1}(f,g)} & {(\hom^v_{D}(f,g),t^c\hom^v_{D}(f,g))_1} \\
	{\hom_{(D,t^cD)_0}(a,a')\times \hom_{(D,t^cD)_0}(b,b')} & {(\hom^v_{D}(f,g),t^c\hom^v_{D}(f,g))_0\times (\hom^v_{D}(f,g),t^c\hom^v_{D}(f,g))_0}
	\arrow[from=1-1, to=1-2]
	\arrow[from=1-1, to=2-1]
	\arrow[from=1-2, to=2-2]
	\arrow[from=2-1, to=2-2]
\end{tikzcd}\]
implies that 
$$\hom_{(D,t^cD)_1}(f,g)\to\hom_{(D,t^cD)_0}(b,b')\times \hom_{(D,t^cD)_0}(a,a')$$
is a two-sided fibration.
The span of the statement then verifies the condition $(3)$ of proposition \ref{prop: another characterization}, which implies that it is a two-sided fibration.

Suppose now that $n=\omega$. Let $i$ be a morphism of $\IF_{ (D,t^cD)_0, (D,t^cD)_0}$. By construction, the domain and the codomain of $i$ are $k$-categories for some $k$. The span $$(D,t^cD)_1\to (D,t^cD)_0\times (D,t^cD)_0$$ is local with respect to $i$ if and only if $$\tau_{k,1}(D,t^cD)_1\to \tau_{k,1}(D,t^cD)_0\times \tau_{k,1}(D,t^cD)_0$$ is local with respect to $i$, which is true by the induction hypothesis.
\end{proof}

We then have the following corollary of Ruit theorem:

\begin{theorem}
\label{theo:generaljaco} Let $n\in \Nb\cup\{\omega\}$ and let $D_\bullet$ be a $(n,1)$-double category. $$(D_1,t^{c}D_1)\to(D_0,t^{c}D_0)\times (D_0,t^{c}D_0)$$ is a two-sided fibration. Furthermore, if $D_\bullet$ is a marked $(n,1)$-double category such that $D_0:=D_0^{\sharp}$ and $$D_1\to D_0\times D_0$$ is a two-sided fibration, then $D^\natural$ is accompanied.
\end{theorem}

\begin{proof}
The first assertion is lemma \ref{lemma: accomoagnated implies marjed}. It remains to prove the second assertion, and to this end, let $D_\bullet$ be a marked $(n,1)$-double category fulfilling the desired condition. We will show that any $1$-cell of $D_0$ admits a companion. By the proposition \ref{prop:cartesian fibration between arrow} and by the definition of companionable $n$-cell, this will imply the result by an easy induction.

Let $f$ be a $1$-cell of $D_0$. By proposition \ref{prop: another characterization}, there exists an unique lift $l$ appearing in the following diagram:
\[\begin{tikzcd}
	{[1]^\sharp} && {D_1} \\
	{[1]^\sharp\vee[1]^\sharp} && {D_0^\sharp\times D_0^\sharp}
	\arrow["{\Ib_f}", from=1-1, to=1-3]
	\arrow[from=1-1, to=2-1]
	\arrow[from=1-3, to=2-3]
	\arrow["l"', dashed, from=2-1, to=1-3]
	\arrow["{(\pi_0f,f)\vee(f,\pi_0)}"', from=2-1, to=2-3]
\end{tikzcd}\]
This lift $l$ then corresponds to the data of
\[\begin{tikzcd}
	a & a \\
	a & b & a & a \\
	b & b & b & b
	\arrow["\shortmid"{marking}, equals, from=1-1, to=1-2]
	\arrow[equals, from=1-1, to=2-1]
	\arrow["\psi"{description}, draw=none, from=1-1, to=2-2]
	\arrow["f", from=1-2, to=2-2]
	\arrow["{\overline{f}}"', "\shortmid"{marking}, from=2-1, to=2-2]
	\arrow["f", from=2-1, to=3-1]
	\arrow["\phi"{description}, draw=none, from=2-1, to=3-2]
	\arrow[""{name=0, anchor=center, inner sep=0}, equals, from=2-2, to=3-2]
	\arrow["\shortmid"{marking}, equals, from=2-3, to=2-4]
	\arrow[""{name=1, anchor=center, inner sep=0}, "f", from=2-3, to=3-3]
	\arrow["{\Ib^h_f}"{description}, draw=none, from=2-3, to=3-4]
	\arrow["f", from=2-4, to=3-4]
	\arrow["\shortmid"{marking}, equals, from=3-1, to=3-2]
	\arrow["\shortmid"{marking}, equals, from=3-3, to=3-4]
	\arrow["\sim"{description}, draw=none, from=0, to=1]
\end{tikzcd}\]
Furthermore, $\phi \#_1 \psi$ is a marked $1$-cell of $D_1$ which lies in the fiber of $D_1$ over $\{a\}\times \{b\}$. By proposition \ref{prop:left fib over flat} $1$-cell is then an equivalence, and we then have
\[\begin{tikzcd}
	a & a & b & a & b \\
	a & b & b & a & b
	\arrow["\shortmid"{marking}, equals, from=1-1, to=1-2]
	\arrow[equals, from=1-1, to=2-1]
	\arrow["\psi"{description}, draw=none, from=1-1, to=2-2]
	\arrow["{\overline{f}}"', "\shortmid"{marking}, from=1-2, to=1-3]
	\arrow["f", from=1-2, to=2-2]
	\arrow["\phi"{description}, draw=none, from=1-2, to=2-3]
	\arrow[""{name=0, anchor=center, inner sep=0}, equals, from=1-3, to=2-3]
	\arrow["{\overline{f}}", "\shortmid"{marking}, from=1-4, to=1-5]
	\arrow[""{name=1, anchor=center, inner sep=0}, equals, from=1-4, to=2-4]
	\arrow["{\Ib^v_{\overline{f}}}"{description}, draw=none, from=1-4, to=2-5]
	\arrow[equals, from=1-5, to=2-5]
	\arrow["{\overline{f}}"', "\shortmid"{marking}, from=2-1, to=2-2]
	\arrow["\shortmid"{marking}, equals, from=2-2, to=2-3]
	\arrow["{\overline{f}}"', "\shortmid"{marking}, from=2-4, to=2-5]
	\arrow["\sim"{description}, draw=none, from=0, to=1]
\end{tikzcd}\]
\end{proof}

We directly obtain the following corollary:

\begin{cor}
\label{cor:marking and accompanied} The functor $$(\uvar)^\natural:(\omega,1)\Dcatm\to (\omega,1)\Dcat$$ induces an equivalence: $$(\uvar)^\natural:(\omega,1)\Dcatmc\sim (\omega,1)\Dcatc$$
\end{cor}

\subsection{Highly directed  Čech nerve}

\begin{definition}
Let $n,m$ be two elements of $\mathbb{N}\cup\{\omega\}$. A \textit{$(n,m)$-filtration} is a sequence 
$$A_0\to A_1\to \dots \to A_m$$
if $m<\omega$, or a infinite sequence 
$$A_0\to A_1\to \dots$$
if $m=\omega$, and where for any integer $i\leq m$, 
$A_i$ is an $(n+i)$-category. A $(n,m)$-filtration is \textit{surjective} if for any $i<n$, $A_i\to A_{i+1}$ is $i$-surjective.

We denote by $\Filt_{n,m}$ the subcategory of $\Fun([m],\npmcat)$ whose objects correspond to $(n,m)$-filtrations and $\Filts_{n,m}$ the full subcategory of surjective $(n,m)$-filtrations.
\end{definition}

\begin{remark}
For any $k$, $\Filt_{\omega,k}$ is simply the functor category $\Fun([k],\ocat)$.
\end{remark}

\begin{construction}
\label{cons:cechnerve}
Given a $(n,m)$-filtration $A_\bullet$, its \textit{Čech nerve}, denoted by $\C_{\bullet}(A)$, is the $(n,m)$-double category whose value on $a$  is defined by induction on $|a|$ by setting  $\C_{[0]}(A):= A_0$ and $\C_a(A)$ for $|a|>0$ as the pullback: 
\[\begin{tikzcd}
	{\C_a(A)} & {A^{a}} \\
	{\C_{\tau_{|a|-1}(a)}(A)} & {A^{\tau_{|a|-1}(a)}}
	\arrow[from=1-1, to=1-2]
	\arrow[from=1-1, to=2-1]
	\arrow["\lrcorner"{anchor=center, pos=0.125}, draw=none, from=1-1, to=2-2]
	\arrow[from=1-2, to=2-2]
	\arrow[from=2-1, to=2-2]
\end{tikzcd}\]

Given a $(n,m)$-double category $D_{\bullet}$, its realization, denoted by $|D|_\bullet$, is the $(n,m)$-filtration defined by the formula
$$|D|_k := \coend_{\Theta_k} D_a \otimes a.$$
In particular, we have $|D|_0 := D_0$.

These two functors fit into an adjunction
\begin{equation}
\label{cons:adj betwn highly cech}
\begin{tikzcd}
	{|\uvar|:(n,m)\Dcat} & {\Filt_{n,m}:\C}
	\arrow[""{name=0, anchor=center, inner sep=0}, shift left=2, from=1-1, to=1-2]
	\arrow[""{name=1, anchor=center, inner sep=0}, shift left=2, from=1-2, to=1-1]
	\arrow["\dashv"{anchor=center, rotate=-90}, draw=none, from=0, to=1]
\end{tikzcd}
\end{equation}
\end{construction}

\begin{remark}
Let $A_\bullet$ be a $(n,m)$-filtration. The $n$-category $\C(A)_1$ is $A_0 \dirtimes{A_1} A_0$, the $n$-category $\C(A)_2$ is $A_0 \overset{\to}{\underset{\overset{\to}{A_1 \underset{A_2}{\times} A_1}}{\times}} A_0$, and so on.
\end{remark}

\begin{remark}
\label{remark: an easy description of relation case nm} It can be useful to note that the functor $|\uvar|$ is the restriction of the colimit preserving functor $$|\uvar|:\Fun(\Theta_m^{op},\ncat)\to \Filt_{n,m}$$ defined by the formula: $$|\langle C,a\rangle| := C\otimes \tau_0 a\to C\otimes \tau_1 a\to ...\to C\otimes \tau_k a\to... $$
\end{remark}

\begin{prop}
\label{prop:image of real}
Let $D_\bullet$ be a $(n,m)$-double category. Then $|D|_\bullet$ belongs to $\Filts_{n,m}$.
\end{prop}

\begin{proof}
This directly follows from remark \ref{remark: an easy description of relation case nm} and from proposition \ref{prop:otimes and surjection} that implies that $C\otimes\tau_ka\to C\otimes \tau_{k+1}a$ is $k$-surjective.
\end{proof}

\begin{definition}
\label{defi:nmaccomp}
We define by induction on $(n,m)$ the notion of \textit{accompanied $(n,m)$-double categories} as follows: 
\begin{enumerate} 
\item Every $(0,m)$-double category and $(n,0)$-double category is accompanied. 
\item If $n<\omega$ or $m<\omega$, a $(n,m)$-double category is accompanied if $\tau_{1,1}D_\bullet$ is accompanied, and for any objects $x,y$ in $D_0$, $\hom^v_{D}(x,y)_{\bullet^{\circ}}$ and $\hom^h_{D}(x,y)_{\bullet}$ are accompanied. 
\item A $(\omega,\omega)$-double category $D_\bullet$ is accompanied if for any $n<\omega$ and $m<\omega$, $\tau_{n,m}D_\bullet$ is accompanied. 
\end{enumerate}
\end{definition}

\begin{remark}
\label{rem:autre def ofthe accompagnability}
Let $n,m\in \mathbb{N}\cup\{\omega\}$, and let $D_\bullet$ be a $(n,m)$-double category. If $m<\omega$, then the $(n,m)$-double category $D_{\bullet}$ is accompanied if and only if 
\begin{enumerate}
\item $\tau_{n,1}D_\bullet$ is accompanied.
\item for any $x,y\in D_0$, the $(n,m-1)$-double category $\hom^h_D(x,y)_{\bullet}$ is accompanied.
\end{enumerate}
If $m=\omega$, the $(n,m)$-double category $D_\bullet$ is accompanied if and only if for any $k>0$, $\tau_{n,k}(D_\bullet)$ is accompanied.
\end{remark}

\begin{remark}
\label{rem:if D is accompagned, so is Dan}
Let $D_\bullet$ be an accompanied $(n,m)$-double category. By assumption, $D_{[[0],\bullet]}$ is accompanied. Furthermore, for any $a$ in $\Theta_{m-1}$, the morphism $D_{[[0],\bullet]} \to D_{[a,\bullet]}$ is an equivalence when evaluated on $[0]$. By remark \ref{rem:accom closed}, this implies that $D_{[a,\bullet]}$ is accompanied.
\end{remark}

\begin{lemma}
\label{lemma:characterization equivalence accompanied} Let $ n \in \Nb \cup \{\omega\}$, $m\in \Nb$ and let $\phi:C_\bullet\to D_\bullet$ be a morphism between accompanied $(n,m)$-double categories. Then $\phi$ is an equivalence if and only if the following hold: 
\begin{enumerate}
\item the induced morphism $C_0\to D_0$ is an equivalence.
\item for any $x,y\in C_0$, the induced morphism $\hom^h_C(x,y)_\bullet\to \hom^h_D(\phi(x),\phi(y))_\bullet$ is an equivalence.
\end{enumerate}
\end{lemma}

\begin{proof}
The remark \ref{rem:if D is accompagned, so is Dan} implies that for any $a$ in $\Theta_{m-1}$, the $(n,1)$-double category $C_{[a,\bullet]}$ and $D_{[a,\bullet]}$ are accompanied. A morphism $\phi$ satisfies the two conditions of the statement if and only if the functors $\phi_a:C_{[a,\bullet]}\to D_{[a,\bullet]}$ induce equivalences when evaluated on $[0]$ and on fibers for any $a$ in $\Theta_{m-1}$. Moreover, $\phi$ is an equivalence if and only if all the $\phi_a$ are equivalences. The result then follows from lemma \ref{lemma:characterization equivalence accompanied marked case.}.
\end{proof}

\begin{construction}
Let $A_\bullet$ be an object of $\Filt_{n,m}$, and $x,y$ two objects of $A_0$. We denote by $\hom_A(x,y)_{\bullet}$ the $(n,m-1)$-filtration: $$\hom_{A_1}(f_1(x),f_1(y)) \to \hom_{A_2}(f_2(x),f_2(y)) \to \ldots \to \hom_{A_k}(f_k(x),f_k(y)) \to \ldots$$
\end{construction}

\begin{lemma}
\label{lemma:hom and cech}
Let $A_\bullet$ be an object of $\Filt_{n,m}$, and $x,y$ two objects of $A_0$. The canonical morphism $$\hom^h_{\C(A)}(x,y)_\bullet\to \C_\bullet(\hom_A(x,y))$$ is an equivalence.
\end{lemma}

\begin{proof}
This is a consequence of proposition \ref{prop:otimes and an}.
\end{proof}

\begin{construction}
Let $m' \geq m$. We denote by $\tau_{m,n}:\Filt_{n,m'}\to \Filt_{n,m}$ the functor sending a sequence $k<m' \mapsto A_k$ to $k<m \mapsto A_k$. Remark that we have canonical equivalences $$|\tau_{n,m}\uvar|\sim \tau_{n,m}|\uvar|~~~~ \C(\tau_{n,m}\uvar)\sim\tau_{n,m}\C(\uvar)$$
\end{construction}

\begin{prop}
\label{prop:filtration induces accomp} For every $(n,m)$-filtration $A_{\bullet}$, the $(n,m)$-double category $\C_\bullet(A)$ is accompanied.
\end{prop}

\begin{proof}
We proceed by induction on $m$. The case $m=1$ follows from theorem \ref{theo:eff of n1 dcat}.

Suppose the result is true at the stage $(m-1)$ and let $A_\bullet$ be a $(n,m)$-filtration. Remark that we have a canonical equivalence
$$\tau_{n,1}\C_\bullet(A)\sim \C_\bullet(A_0\to A_1)$$
and it is accompanied by theorems \ref{theo:generaljaco} and \ref{theo:eff of n1 dcat marked}. The induction hypothesis and lemma \ref{lemma:hom and cech} imply that $\hom^h_{\C(A)}(x,y)_\bullet$ is accompanied. By remark \ref{rem:autre def ofthe accompagnability}, this implies that $\C_\bullet(A)$ is accompanied.

If $m=\omega$, the result follows from the equivalence
$$\C(\tau_{n,m}\uvar)\sim\tau_{n,m}\C(\uvar)$$
and the already proven case $m<\omega$.
\end{proof}

\subsection{$(n,1)$-Effectivity}
\label{section n1 eff}
\begin{theorem}
\label{theo:eff of n1 dcat}
The realization of a $(n,1)$-double category is always a surjective $(n,1)$-filtration, and the directed  Čech nerve of a $(n,1)$-filtration is always an accompanied $(n,1)$-double category.

The realization and the directed marked Čech nerve induce inverse equivalences:
$$|\uvar|:(n,1)\Dcatc\sim \Filts_{n,1}:\C.$$
\end{theorem}

\begin{proof}
It is sufficient to demonstrate the results when $n=\omega$. In this case, this follows from the equivalence between marked accompanied $(\omega,1)$-double categories and accompanied $(\omega,1)$-double categories of corollary \ref{cor:marking and accompanied} and from theorem \ref{theo:eff of n1 dcat marked}.
\end{proof}

\begin{cor}
\label{cor:unit is 1 fully faithful} Let $A_\bullet$ be an object of $\Filt_{n,1}$. The canonical morphism $$|\C(A)|_1\to A_1$$ is fully faithful.
\end{cor}

\begin{proof}
Let $A_0\to A_1$ be an object of $\Filt_1(\npcat)$. Let $\tilde{A}_1$ be the unique $(n+1)$-category fitting in a commutative square
\[\begin{tikzcd}
	{A_0} & {\tilde{A}_1} \\
	{A_0} & {A_1}
	\arrow[two heads, from=1-1, to=1-2]
	\arrow[equals, from=1-1, to=2-1]
	\arrow[hook, from=1-2, to=2-2]
	\arrow[from=2-1, to=2-2]
\end{tikzcd}\]
where the top arrow is $0$-surjective and the left vertical arrow is fully faithful.

We can easily check that the assignment $(A_0\to A_1)\mapsto (A_0\to \tilde{A}_1)$ is the right adjoint of the canonical inclusion of $\Filt_{n,1}$ into $\Filts_{n,1}$. However, the theorem \ref{theo:eff of n1 dcat} implies that $|\C\uvar|$ is also a right adjoint of this inclusion. The two canonical morphisms $|\C(A)|_1\to A_1$ and $\tilde{A}_1\to A_1$ then coincide, which concludes the proof.
\end{proof}

\begin{cor}
\label{cor: naturality} Let $D_\bullet$ be an accompanied $(n,1)$-double category and $f:A\to D_0$ a morphism. The canonical morphism $$|f^*D|_1\to |D|_1$$ is fully faithful. In particular, if $f$ is essentially surjective, this morphism is an equivalence.
\end{cor}

\begin{proof}
This directly follows from the fact that $D_\bullet$ is the Čech nerve of its realization by theorem \ref{theo:eff of n1 dcat} and from the corollary \ref{cor:unit is 1 fully faithful}.
\end{proof}

\begin{remark}
\label{rem:realization is companion}
Let $D_\bullet$ be a companionable $(n,1)$-double category. We then have $ D_0\overset{\to}{\underset{|D|_1}{\times}} D_0\sim D_1$, which induces for any $x,y$ in $D_0$, a morphism $$\hom(|x|_1,|y|_1)\sim (D_1)_{x,y}.$$ We can then easily check that the induced morphism $$\hom_{D_0}(x,y)\to \hom_{|D|_1}(|x|_1,|y|_1)\sim (D_1)_{x,y}$$ sends a $1$-arrow of $D_0$ to its companion.
\end{remark}

\begin{remark}
Let $D_\bullet$ be an accompanied $(1,1)$-double category. Let $\iota:\tau_0D_0\to D_0$ be the inclusion of the maximal sub homotopy type. The $2$-category $|\iota^*D|_1$ is the $2$-category freely generated by the flagged $2$-category $\Hor(D)$, and the corollary \ref{cor: naturality} then implies that $|D|$ corresponds to a morphism $$D_0\to \Fb\Hor(D_\bullet).$$
\end{remark}

\begin{cor}
The inclusion $\iota:(n,1)\Dcatc\to (n,1)\Dcat$ is part of a triplet of adjunctions:
\[\begin{tikzcd}
	{ (n,1)\Dcat} & {(n,1)\Dcatc:\iota}
	\arrow[""{name=0, anchor=center, inner sep=0}, "\Comp"', shift right=3, from=1-1, to=1-2]
	\arrow[""{name=1, anchor=center, inner sep=0}, "\Fb", shift left=3, from=1-1, to=1-2]
	\arrow[""{name=2, anchor=center, inner sep=0}, from=1-2, to=1-1]
	\arrow["\dashv"{anchor=center, rotate=-90}, draw=none, from=1, to=2]
	\arrow["\dashv"{anchor=center, rotate=-90}, draw=none, from=2, to=0]
\end{tikzcd}\]
Where given a $(n,1)$-double category $D_\bullet$, $\Fb(D_\bullet):= \C_\bullet(|D|).$
\end{cor}

\begin{prop}
\label{prop:comp stable by colimit} The category $\nDcatc$ is closed under colimits. Moreover, the canonical functor $\nDcatc \to \nDcat$ preserves them.
\end{prop}

\begin{proof}
Let $F:I\to \nDcat$ be a functor such that for any $i:I$, $F(i)$ is accompanied. We want to show that $\colim_IF$ is accompanied. Remark first that the colimit preserving functor $ev_{[0]}:\Fun(\Delta^{op},\ncat)\to \ncat$ sends $\langle X,\Sp_n\rangle\to \langle X,[n]\rangle$ to equivalence. This then implies that the induced functor $ev_{[0]}(n,1)\Dcat\to \ncat$ is colimit preserving.

Furthermore, as functors of $(n,1)$-double categories preserve companions, the hypothesis implies that every cell lying in the image of $F(i)\to \colim_IF$ for some $i$ has companions. However, proposition \ref{prop:comp stable by comp} implies that the space of cells having companions is closed under composition. By example \ref{exe:canonical exe of cloture by comp}, every cell of non-negative dimension of $(\colim_IF)_0$ admits a companion, and the $(n,1)$-double category $\colim_IF$ is then accompanied.
\end{proof}

\begin{proof}
The adjunction $\Fb \dashv \iota$ comes from the adjunction
\[\begin{tikzcd}
	{|\uvar|: (n,1)\Dcat} & {\Filt_1^{\twoheadrightarrow}(\npcat):\C}
	\arrow[""{name=0, anchor=center, inner sep=0}, shift left=2, from=1-1, to=1-2]
	\arrow[""{name=1, anchor=center, inner sep=0}, shift left=2, from=1-2, to=1-1]
	\arrow["\dashv"{anchor=center, rotate=-90}, draw=none, from=0, to=1]
\end{tikzcd}\]
and the theorem \ref{theo:eff of n1 dcat}. The existence of the adjunction $\iota \dashv \Comp$ follows from proposition \ref{prop:comp stable by colimit} that states that $\iota$ preserves colimits.
\end{proof}

\begin{cor}
\label{cor:cech preserves colimits} Let $\{A_i\to B_i\}_{i\in I}$ be a family of $0$-surjective morphisms in $\Filt_{n,1}$ indexed by a category $I$. Then we have an equivalence $$\colim_{i:I} \C_\bullet(A_i\to B_i) \sim \C_\bullet(\colim_{i:I}A_i\to \colim_{i:I}B_i).$$
\end{cor}

\begin{proof}
Remark that $(\colim_{i:I}A_i\to \colim_{i:I}B_i)$ is the colimit of the functor $I\to \Filts_{\omega,1}$ induced by the family of morphisms. The result then follows from theorem \ref{theo:eff of n1 dcat} and from proposition \ref{prop:comp stable by colimit}, which states that colimits in $(n,1)\Dcatc$ coincide with colimits in $(n,1)\Dcat$.
\end{proof}

\subsubsection*{Consequence on $n$-surjective morphisms}

\begin{cor}
\label{cor:characterization of n surjective}
Let $f:C\to D$ be a morphism. Then $f$ is $n$-surjective if and only if it is surjective of objects and it is locally $(n-1)$-surjective, i.e. for any pair of objects $x,y$ of $C$, $$\hom_C(x,y)\to \hom_D(f(x),f(y))$$ is $(n-1)$-surjective.
\end{cor}

\begin{cor}
\label{cor:stability of n surjective by some pullback}
Suppose we are given a diagram
\begin{equation}
\label{eq:stability of n surjective by some pullback}
\begin{tikzcd}
	A & B & C \\
	{A'} & B & {C'}
	\arrow[from=1-1, to=1-2]
	\arrow["f"', from=1-1, to=2-1]
	\arrow["id_B",equals, from=1-2, to=2-2]
	\arrow[from=1-3, to=1-2]
	\arrow["g", from=1-3, to=2-3]
	\arrow[from=2-1, to=2-2]
	\arrow[from=2-3, to=2-2]
\end{tikzcd}
\end{equation}
If $f$ and $g$ are $n$-surjective, so is $f\times_{id_B}C:A\times_BC\to A'\times_BC'$.
\end{cor}

\begin{proof}[Proof of corollaries \ref{cor:characterization of n surjective} and \ref{cor:stability of n surjective by some pullback}]
We will proceed in three steps. We will first show that a morphism $f$ that is surjective on objects and that induces for any $x,y$, a $(n-1)$-surjection $$\hom_C(x,y)\to \hom_D(f(x),f(y))$$ is $n$-surjective.
We will then show that this implies the case $n$ of the corollary \ref{cor:characterization of n surjective}, and we will then deduce that $n$-surjective morphisms are surjective on objects and are locally $(n-1)$-surjective.
\vspace{0.5cm}

Suppose then first that $f$ is surjective on objects and $$\hom_C(x,y)\to \hom_D(f(x),f(y))$$ is $(n-1)$-surjective.
We then have to show that any diagram of shape 
\begin{equation}
\label{eq:square ff and surj}
\begin{tikzcd}
	C & A \\
	D & B
	\arrow["k", from=1-1, to=1-2]
	\arrow["f"', from=1-1, to=2-1]
	\arrow["g", from=1-2, to=2-2]
	\arrow["l"', from=2-1, to=2-2]
\end{tikzcd}
\end{equation}
where $g$ is $(n+1)$-fully faithful admits a unique lift.

By the naturality of the factorization in $0$-surjective morphisms followed by fully faithful functors and by the stability by right cancellation of $(n+1)$-fully faithful morphisms, we can reduce to the case where the two morphisms $k,l$ in \eqref{eq:square ff and surj} are surjective on objects.

We now choose an inclusion $\pi_0C\to C$. The assumption implies that the morphisms $\pi_0C\to A$, $\pi_0C\to D$, and $\pi_0C\to B$ are $0$-surjective. The theorem \ref{theo:eff of n1 dcat} then implies that lifts in the square \eqref{eq:square ff and surj} correspond to lifts in the square:
\begin{equation}
\label{eq:square ff and surj2}
\begin{tikzcd}
	{\C_\bullet(\pi_0C\to C)} & {\C_\bullet(\pi_0C\to A)} \\
	{\C_\bullet(\pi_0C\to D)} & {\C_\bullet(\pi_0C\to B)}
	\arrow["k", from=1-1, to=1-2]
	\arrow["f"', from=1-1, to=2-1]
	\arrow["g", from=1-2, to=2-2]
	\arrow["l"', from=2-1, to=2-2]
\end{tikzcd}
\end{equation}
Lifts in this square exist and are unique if and only if the induced square by the evaluation on $0$ and $1$ admits unique lifts. Evaluated on $0$, this square is just composed of identities. Evaluated on $1$, it corresponds to the square
\[\begin{tikzcd}
	{\coprod_{x,y\in \pi_0(C)}\hom_C(x,y)} & {\coprod_{x,y\in \pi_0(A)}\hom_C(k(x),k(y))} \\
	{\coprod_{x,y\in \pi_0(D)}\hom_C(f(x),f(y))} & {\coprod_{x,y\in \pi_0(B)}\hom_C(gk(x),gk(y))}
	\arrow[from=1-1, to=1-2]
	\arrow[from=1-1, to=2-1]
	\arrow[from=1-2, to=2-2]
	\arrow[from=2-1, to=2-2]
\end{tikzcd}\]
The hypothesis implies that the left-hand morphism is $(n-1)$-surjective, and the right one is $n$-fully faithful. This square then admits a unique lift which concludes the first step of the proof.

Suppose now given a diagram of shape \eqref{eq:stability of n surjective by some pullback}. We proceed by induction on $k\leq n$. The case $0$ is obvious thanks to proposition \ref{prop: characterization of 1 surjective2}. Suppose the result proven at the stage $k<n$. The functor $f\times_{id_b}g$ is then $0$-surjective on objects and the induction hypothesis implies that it is locally a $(k-1)$-surjection. By the previous step, it is then a $k$-surjection. 

Eventually, let $f:C\to D$ be $n$-surjective. As $f$ is $0$-surjective, it is surjective on objects by proposition \ref{prop: characterization of 1 surjective2}. We choose an inclusion $\pi_0C\to C$. For any $m$, we consider the factorization 
$$\C_m(\pi_0C\to C)\to  E_m\to \C_m(\pi_0C\to D)$$
into a $(n-1)$-surjective morphism followed by a $n$-fully faithful morphism. As fully faithful functors are stable by pullback, and as  $\C_0(\pi_0C\to C)\to \C_0(\pi_0C\to D)\sim C$, the case $(n-1)$ of the corollary \ref{cor:stability of n surjective by some pullback} then implies that $E_m$ is a Segal object. Eventually, as we have a morphism $\C_\bullet(\pi_0C\to C)\to  E_\bullet$ inducing an equivalence on objects and as $\C_0(\pi_0C\to C)$ is companionable, so is $E_\bullet$ by remark \ref{rem:accom closed}. By theorem \ref{theo:eff of n1 dcat}, this then implies that $C\to |E_\bullet|_1$ is locally $(n-1)$-surjective and $|E_\bullet|_1\to D$ is $(n+1)$-fully faithful. As the composite of these two morphisms is by assumption $n$-surjective, $|E_\bullet|_1\sim D$, and $f$ is then locally $(n-1)$-surjective, which concludes the proof.
\end{proof}

\begin{cor}
\label{cor:characterization of n surjective pullback} $n$-surjective morphisms are closed under pullback.
\end{cor}

\begin{proof}
A direct induction on $n$ using the corollary \ref{cor:characterization of n surjective} implies that a functor $f:C\to D$ is $n$-surjective if and only if any diagram of shape 
\[\begin{tikzcd}
	{\partial \Db_k} & C \\
	{\Db_k} & D
	\arrow[from=1-1, to=1-2]
	\arrow[from=1-1, to=2-1]
	\arrow[from=1-2, to=2-2]
	\arrow[dashed, from=2-1, to=1-2]
	\arrow[from=2-1, to=2-2]
\end{tikzcd}\]
admits a (a priori non-unique) lift. This directly implies the stability of $n$-surjective morphisms by pullback.
\end{proof}

\begin{cor}
\label{cor:yet another characterization of fully faithful functor}
Let $f:A\to B$ be a functor. Then $f$ is $(n+1)$-surjective if and only if it is surjective on objects and if $$A^{[1]}\to A\overset{\to}{\underset{B}{\times}}A$$ is $n$-surjective. 

The morphism $f$ is $(n+1)$-fully faithful if and only if $$A^{[1]}\to A\overset{\to}{\underset{B}{\times}}A$$ is $n$-fully faithful.
\end{cor}

\begin{proof}
We will denote by $\alpha$ the functor $A^{[1]}\to A\overset{\to}{\underset{B}{\times}}A$. Remark that we have a cartesian square:
\[\begin{tikzcd}
	{\hom_A(a,b)} & {A^{[1]}} \\
	{\hom_{B}(f(a),f(b))} & { A\overset{\to}{\underset{B}{\times}}A} \\
	{\{a\}\times \{b\}} & {A\times A}
	\arrow[from=1-1, to=1-2]
	\arrow[from=1-1, to=2-1]
	\arrow[from=1-2, to=2-2]
	\arrow[""{name=0, anchor=center, inner sep=0}, from=2-1, to=2-2]
	\arrow[from=2-1, to=3-1]
	\arrow[from=2-2, to=3-2]
	\arrow[""{name=1, anchor=center, inner sep=0}, from=3-1, to=3-2]
	\arrow["\lrcorner"{anchor=center, pos=0.125}, draw=none, from=1-1, to=0]
	\arrow["\lrcorner"{anchor=center, pos=0.125}, draw=none, from=2-1, to=1]
\end{tikzcd}\]
By the definition of a $n$-fully faithful functor, it implies that if $\alpha$ is $n$-fully faithful then $f$ is $(n+1)$-fully faithful. Similarly, by corollary \ref{cor:characterization of n surjective},  if $f$ is surjective on objects and $\alpha$ is $n$-surjective, then $f$ is $(n+1)$-surjective.

It remains to show that if $f$ is $(n+1)$-surjective (resp. $(n+1)$-fully faithful) then $\alpha$ is $n$-surjective (resp. $n$-fully faithful). Let $u:a\to a'$ and $v:b\to b'$ be two elements of $A^{[1]}$. The proposition \ref{prop:decomposition} implies that we have 
a cartesian square:

\[\begin{tikzcd}
	{\hom_{A^{[1]}}(u,v)} & {\hom_A(a,b')^{[1]}} \\
	{\hom_{A\overset{\to}{\underset{B}{\times}}A}(u,v)} & {\hom_A(a,b')\overset{\to}{\underset{\hom_B(f(a),f(b'))}{\times}}\hom_A(a,b')} \\
	\begin{array}{c} \hom_A(a,b)\times \hom_A(a',b')\\ \end{array} & {\hom_A(a,b')\times \hom_A(a,b')}
	\arrow[from=1-1, to=1-2]
	\arrow[from=1-1, to=2-1]
	\arrow[from=1-2, to=2-2]
	\arrow[""{name=0, anchor=center, inner sep=0}, from=2-1, to=2-2]
	\arrow[from=2-1, to=3-1]
	\arrow[from=2-2, to=3-2]
	\arrow[""{name=1, anchor=center, inner sep=0}, "{v_!\times u_!}"', from=3-1, to=3-2]
	\arrow["\lrcorner"{anchor=center, pos=0.125}, draw=none, from=1-1, to=0]
	\arrow["\lrcorner"{anchor=center, pos=0.125}, draw=none, from=2-1, to=1]
\end{tikzcd}\]

We leave it to the reader to check that an obvious induction on $n$ enables us to conclude the proof.
\end{proof}

\subsubsection*{Square functors and variations}
We now collect some results on the square and pair-square functors considered in the appendix of \cite{Gaitsgory_A_study_on_DAG}.

\begin{definition}
The \textit{square} functor is defined as 
$$\begin{array}{rrcl} \Sq^{n+1}:& \npcat &\to&(n,1)\Dcat\\ & D&\mapsto &\C_\bullet(\tau_{n} D \to D)\\ 
\end{array}$$
Unfolding the notation, for any pair of integers $n,m$ and for any $2$-category $D$, $$\mbox{$(\Sq^2_{n}(D))_m$}:=\Hom([m]\otimes[n],D).$$ The functor $\Sq^2$ is then the square functor considered in the appendix of \cite{Gaitsgory_A_study_on_DAG}.
 
We denote by $\npcat^{\Pair}$ the subcategory of $\Filts_{\omega,1}$ whose objects correspond to surjections $C\to D$ that are monomorphisms for $k$-cells for any $k\leq n$. The \textit{pair square} functor is defined as $$\begin{array}{rrcl} \Sq^{n,\Pair}:& \npcat^{\Pair}&\to&(n,1)\Dcat\\ & C\to D&\mapsto &\C_\bullet(C\to D)\\ \end{array}$$
\end{definition}

\begin{definition}
A $(n,1)$-double category $D_\bullet$ is \textit{complete} if the canonical map $$D_0\to D_{E_{eq}}$$ is an equivalence. We denote by $(n,1)\Dcat_{cplt}$ the subcategory of $(n,1)\Dcat$ whose objects are complete $(n,1)$-double categories.
\end{definition}

\begin{lemma}
\label{lemma:horizontal completeness} Let $D_\bullet$ be an accompanied $(n,1)$-double category. Then $D_\bullet$ is complete if and only if $$D_0 \to |D_{1}|$$ is a monomorphism on $k$-cells for any $k \leq n$.
\end{lemma}

\begin{proof}
Remark that $D_\bullet$ is complete if and only if it is local with respect to $\langle \Db_k,E_{eq}\rangle\to \langle\Db_k,[0]\rangle$ for any $k\leq n$. By theorem \ref{theo:eff of n1 dcat}, $D_\bullet$ is the Čech nerve of its realization, and lifts in the diagram
\[
\begin{tikzcd}
	{\langle \Db_k,E_{eq}\rangle} & {D_{\bullet}} \\
	{\langle\Db_k,[0]\rangle} & {[0]}
	\arrow[from=1-1, to=1-2]
	\arrow[from=1-1, to=2-1]
	\arrow[from=1-2, to=2-2]
	\arrow[from=2-1, to=2-2]
\end{tikzcd}
\]
are then equivalent to lifts in the diagram
\[
\begin{tikzcd}
	{(\Db_k\coprod\Db_k\to \Db_k)} & {(D_0\to |D|_1)} \\
	{(\Db_k\to \Db_k)} & {([0]\to [0])}
	\arrow[from=1-1, to=1-2]
	\arrow[from=1-1, to=2-1]
	\arrow[from=1-2, to=2-2]
	\arrow[from=2-1, to=2-2]
\end{tikzcd}
\]
which are equivalent to lifts in the diagram
\[
\begin{tikzcd}
	{\Db_k\coprod\Db_k} & {D_0} \\
	{\Db_k} & {|D|_1}
	\arrow[from=1-1, to=1-2]
	\arrow[from=1-1, to=2-1]
	\arrow[from=1-2, to=2-2]
	\arrow[from=2-1, to=2-2]
\end{tikzcd}
\]
However, lifts in the previous diagram exist and are unique if and only if $D_0\to |D_{1}|$ induces a monomorphism on $k$-cells, which concludes the proof.
\end{proof}

\begin{remark}
Combined with the remark \ref{rem:realization is companion}, the last lemma then implies that in an accompanied $(1,1)$-double category, the assignment $f\mapsto \overline{f}$ induces a monomorphism between vertical $1$-cells and horizontal $1$-cells.
\end{remark}

\begin{lemma}
\label{lemma:all hori are companion} Let $D_\bullet$ be an accompanied $(1,1)$-double category. Then every horizontal $1$-cell of $D_\bullet$ is a companion of some $f$ if and only if $$D_0\to |D|_1$$ is surjective on $1$-cells.
\end{lemma}

\begin{proof}
This directly follows from remark \ref{rem:realization is companion}.
\end{proof}

\begin{definition}
The functor:
$$
\begin{array}{rrcl}
\Delta^n&\to&\ncat\\
([k_i])_{i\leq n}&\mapsto&[k_1]\otimes...\otimes[k_n]\\
\end{array}$$
induces an adjunction
\[\begin{tikzcd}
	{\widehat{\Delta^n}} & {\ncat:\Cube^n}
	\arrow[""{name=0, anchor=center, inner sep=0}, shift left=2, from=1-1, to=1-2]
	\arrow[""{name=1, anchor=center, inner sep=0}, shift left=2, from=1-2, to=1-1]
	\arrow["\dashv"{anchor=center, rotate=-90}, draw=none, from=0, to=1]
\end{tikzcd}\]
The right adjoint $\Cube^n$ is called the \textit{cube functor}.
\end{definition}

\begin{prop}
\label{prop:ff of squares} For any $n$, the functors $\Sq^{n+1}$, $\Sq^{n+1,\Pair}$, and $\Cube^{n+1}$ are fully faithful.
\end{prop}

\begin{proof}
The functor $\Sq^n$ is the composite:
$$\npcat\to \Filts_{n,1}\sim (n,1)\Dcatc\to (n,1)\Dcat$$
where the first functor sends $C$ onto $\tau_{n-1}C\to C$, 
and the functor $\Sq^{n+1,\Pair}$ is the composite 
$$\npcat^{\Pair}\to \Filts_{n,1}\sim (n,1)\Dcatc\to (n,1)\Dcat$$
As all these functors are fully faithful, so are $\Sq^{n+1}$ and $\Sq^{n+1,\Pair}$. Note now that $\Cube^2$ is the composite: 
$$2\mbox{-$\cat$}\xrightarrow{\Sq^2}(n,1)\Dcat\to \widehat{\Delta^2}$$
and is then fully faithful. 
Furthermore, $\Cube^{n+1}$ is the composite: $$\npcat\xrightarrow{\Sq^{n+1}}(n,1)\Dcat\to \Fun(\Delta^{op},\ncat)\xrightarrow{\Fun(\Delta^{op},\Cube^n)}{ \widehat{\Delta^{n+1}}}$$ and an obvious induction on $n$ implies that $\Cube^{n+1}$ is fully faithful.
\end{proof}

\begin{prop}
\label{prop:characterization of the image.} The functor $\Sq^2:2\mbox{-$\cat$}\to \widehat{\Delta\times \Delta}$ is fully faithful and its image consists of accompanied and complete $(1,1)$-double categories where each horizontal cell is the companion of a vertical cell. The functor $\Sq^{n,\Pair}: 2\mbox{-$\cat$}^{\Pair}\to \widehat{\Delta\times \Delta}$ is fully faithful and its image consists of accompanied and complete $(n,1)$-double categories.
\end{prop}

\begin{proof}
This follows from lemmas \ref{lemma:horizontal completeness} and \ref{lemma:all hori are companion}.
\end{proof}

\begin{remark}
\label{rem:gait}
The proposition \ref{prop:ff of squares} implies theorems 4.1.3, 4.3.5, and 4.6.3 of \cite{Gaitsgory_A_study_on_DAG}, and proposition \ref{prop:characterization of the image.} implies theorem 5.2.3 of \cite{Gaitsgory_A_study_on_DAG}. 
A proof of the fully faithfulness of the square functor was already given by Abellán in \cite{Abellan2023}.
\end{remark}

\begin{theorem}
\label{theo:square commutes with cil}
The functors $$\mbox{$\Sq^{2}$}:2\mbox{-$\cat$}\to (1,1)\Dcat_{cplt}$$ preserve colimits.
\end{theorem}

\begin{proof}
As a complete category can be defined as a local object with respect to a set of maps, we have an adjunction
\[\begin{tikzcd}
	{\Fb:(1,1)\Dcat} & {(1,1)\Dcat_{cplt}:\iota}
	\arrow[""{name=0, anchor=center, inner sep=0}, shift left=2, from=1-1, to=1-2]
	\arrow[""{name=1, anchor=center, inner sep=0}, shift left=2, from=1-2, to=1-1]
	\arrow["\dashv"{anchor=center, rotate=-90}, draw=none, from=0, to=1]
\end{tikzcd}\]
Let $\eta:\widehat{\Theta_{2}}\to \Filts_{1,1}$ be the functor that sends $X$ to $\tau_1 X\to X$.
We then define $\alpha$ as the composite
$$\widehat{\Theta_{2}}\xrightarrow{\eta} \Filts_{1,1}\xrightarrow{\C} (1,1)\Dcat \xrightarrow{\Fb} (1,1)\Dcat_{cplt}$$
Note that all these functors preserve colimits, and thus so does $F$. Furthermore, $\eta$ sends $(\Wseg)_{2}$ and $E_{eq}\to [0]$ to equivalences. We are now willing to show that $\eta$ sends $[E_{eq},1]\to [[0],1]$ to an equivalence.

Note that $\C\circ \alpha$ sends this morphism to $\C_{\bullet}(\partial \Db_2\to \Db_1)$ which is, by corollary \ref{cor:cech preserves colimits}, the morphism 
$$\Sq([1])\coprod_{\langle [0],[1]\rangle}\Sq([1])\to \Sq([1])$$
By corollary 2.5.16 of \cite{Ruit2025}, this morphism is sent by $\Fb$ to an equivalence. 

The colimit-preserving functor $\alpha$ then sends every morphism of $\W_2$ to an equivalence. It then induces a colimit-preserving functor 
$$2\mbox{-$\cat$}\to (1,1)\Dcat_{cplt}$$
which by construction coincides with the functor $\Sq^2$.
\end{proof}

\subsection{$(n,m)$-Effectivity}

\begin{lemma}
\label{lemma:technical |||3bis}
Let $D_\bullet$ be a $(\omega,m+1)$-double category. We have a natural equivalence:
$$|[n]\mapsto |D_{[\bullet,n]}|_m|_1\sim |D_\bullet|_{m+1}\sim (|a\mapsto (|D_{[a,\bullet]}|_1)^{\circ}|_{m})^{\circ}$$
\end{lemma}

\begin{proof}
As all these functors commute with colimits, proposition \ref{prop:delta theta and delta} and remark \ref{rem univ of langle nm} imply that it is sufficient to construct these equivalences for the $(\omega,m)$-double category of the shape $\langle C,[a,k]\rangle$ for $C$ an $\omega$-category, $a$ an element of $\Theta_m$, and $k$ an integer.

In this case,  $|[n]\mapsto |\langle C,[a,k]\rangle_{[\bullet,n]}|_m|_1$ is the colimit of the span:
\[\begin{tikzcd}
	{\coprod_{l\leq k}C\otimes  \{l\}} & {\coprod_{l\leq k}C\otimes a\otimes \{l\}} & {C\otimes a\otimes [k]}
	\arrow[from=1-2, to=1-1]
	\arrow[from=1-2, to=1-3]
\end{tikzcd}\]
and $(|a\mapsto (|\langle C,[a,k]\rangle_{[a,\bullet]}|_1)^{\circ}|_{m})^{\circ}$ is the colimit of the span:
\[\begin{tikzcd}
	{\coprod_{l\leq k}((C\otimes \{l\})^\circ)^\circ} & {\coprod_{l\leq k}((C\otimes \{l\})^\circ \otimes a)^\circ} & {((C\otimes [k])^\circ \otimes a)^\circ}
	\arrow[from=1-2, to=1-1]
	\arrow[from=1-2, to=1-3]
\end{tikzcd}\]
By propositions \ref{prop:otimes and dualities}, \ref{prop:otimes and an} and \ref{prop:an other description of the suspension}, these two functors are canonically equivalent to $C\otimes [a,k]$, which concludes the proof by remark \ref{remark: an easy description of relation case nm}.
\end{proof}

\begin{construction}
Let $D_\bullet$ be a $(\omega,m)$-double category and $f:A\to D_0$ a morphism. We denote by $f^*D_\bullet$ the $(\omega,m)$-double category whose value on $a$ fits in the pullback
\[\begin{tikzcd}
	{f^*D_a} & {D_a} \\
	{\prod_{[0]\to a}A} & {\prod_{[0]\to a}D_{[0]}}
	\arrow[from=1-1, to=1-2]
	\arrow[from=1-1, to=2-1]
	\arrow[from=1-2, to=2-2]
	\arrow[from=2-1, to=2-2]
\end{tikzcd}\]
\end{construction}

\begin{lemma}
\label{lemma:lumpen} Let $D_\bullet$ be an accompanied $(\omega,m)$-double category, and $p:X\to D_0$ a $0$-surjection. Then for any $k>0$, the canonical morphism $|p^*D|_k\to |D|_k$ is an equivalence.
\end{lemma}

\begin{proof}
Up to replacing $D_\bullet$ by $\tau_{\omega,k}D_\bullet$, we can suppose $m=k$. The remark \ref{rem:if D is accompagned, so is Dan} states that  $D_{[a,\bullet]}$ is accompanied for any $a$. By lemma \ref{lemma:technical |||3bis} and corollary \ref{cor: naturality}, this induces 
equivalences
$$|f^*D_\bullet|_m\sim (|a\mapsto (|f^*D_{[a,\bullet]}|_{1})^{\circ}|_{m})^{\circ}\sim (|a\mapsto (|D_{[a,\bullet]}|_{1})^{\circ}|_{m})^{\circ}\sim |f^*D_\bullet|_m.$$
\end{proof}

\begin{lemma}
\label{lemma: Dan accompanied} Let $D_\bullet$ be an accompanied $(\omega,m)$-double category such that $D_0$ is a homotopy type. Then for any $n$, $D_{[\bullet,n]}$ is an accompanied $(\omega,m-1)$-double category.
\end{lemma}

\begin{proof}
By assumption, for any $x,y$ in $D_0$, the $(\omega,m-1)$-double category $(D_{[\bullet,1]})_{x,y}$ is accompanied. As $D_0$ is a homotopy type, every cell of $D_{[a,1]}$ is in one of the $(D_{[a,1]})_{x,y}$ for some $x,y$. Unfolding the definition, this implies that the $(n,m-1)$-double category $a\mapsto D_{[a,1]}$ is accompanied. Furthermore, as $D_0$ is a homotopy type, the $(n,m-1)$-double category $a\mapsto D_0$ is accompanied. As accompanied $(\omega,m-1)$-double categories are stable under pullback, this implies that $D_{[\bullet,n]}$ is accompanied for any $n$.
\end{proof}

\begin{theorem}
\label{theo:eff of nm dcat} Let $ n,m \in \Nb \cup \{\omega\}$. The realization of a $(n,m)$-double category is always a surjective $(n,m)$-filtration, and the directed marked Čech nerve of a $(n,m)$-filtration is always an accompanied $(n,m)$-double marked category.

The realization and the directed marked Čech nerve induce inverse equivalences:
$$|\uvar|:(n,m)\Dcatc\sim \Filts_{n,m}:\C.$$
\end{theorem}

\begin{proof}
The first assertion is propositions \ref{prop:image of real} and \ref{prop:filtration induces accomp}. By construction, it is sufficient to demonstrate the second claim for $n=\omega$ and $m<\omega$. We will then proceed by induction on $m$. The case $m=1$ is theorem \ref{theo:eff of n1 dcat}, and we then suppose the result is true at the stage $m$.
\vspace{0.5cm}
Let $D_\bullet$ be an accompanied $(\omega,m+1)$-double category. By the induction hypothesis, the canonical morphism 
$D_{[a,n]}\to \C_{[a,n]}(|D|)$ is an equivalence
for any $|a|<m$. We want to show that this morphism is an equivalence for a globular sum $a$ of dimension $m$.
 By lemmas \ref{lemma:characterization equivalence accompanied} and \ref{lemma:lumpen}, we can reduce to the case where $D_0$ is a homotopy type.  The lemma \ref{lemma: Dan accompanied} then implies that $D_{[\bullet,n]}$ is accompanied for any $n$, and by the induction hypothesis, we have a cartesian square:
\begin{equation}
\label{eq:presque fini}
\begin{tikzcd}
	{D_{[a,n]}} & {(|D_{[\bullet,n]}|_m)^a} \\
	{D_{[\tau_{m-1}a,n]}} & {(|D_{[\bullet,n]}|_m)^{\tau_{m-1}a}}
	\arrow[from=1-1, to=1-2]
	\arrow[from=1-1, to=2-1]
	\arrow["\lrcorner"{anchor=center, pos=0.125}, draw=none, from=1-1, to=2-2]
	\arrow[from=1-2, to=2-2]
	\arrow[from=2-1, to=2-2]
\end{tikzcd}
\end{equation}

By corollary \ref{cor:stability of n surjective by some pullback}, and as $|D_{[\bullet,0]}|$ is constant, the $(n,m)$-filtration 
$$|D_{[\bullet,1]}|\times_{|D_{[\bullet,0]}|}\times....\times_{|D_{[\bullet,0]}|}|D_{[\bullet,1]}|$$
is surjective. As its Čech nerve is $D_{[\bullet,n]}$, the induction hypothesis implies that the canonical morphism
$$|D_{[\bullet,n]}|\to |D_{[\bullet,1]}|\times_{|D_{[\bullet,0]}|}\times....\times_{|D_{[\bullet,0]}|}|D_{[\bullet,1]}|$$
is an equivalence.
The assignment $([n]\mapsto |D_{[\bullet,n]}|_{m})$ is then a Segal object. As evaluated on $[0]$, it is the homotopy type $D_0$, $([n]\mapsto |D_{[\bullet,n]}|_{m})$ is an accompanied $(\omega,1)$-double category. 
By the theorem \ref{theo:eff of n1 dcat} and the equivalence 
$$|D_\bullet|_{m+1}\sim |[n]\mapsto |D_{[\bullet,n]}|_m|_1$$ 
given by lemma \ref{lemma:technical |||3bis},  we get a cartesian square
\[\begin{tikzcd}
	{|D_{[\bullet,n]}|_m} & {(|D|_{m+1})^{[n]}} \\
	{\prod_{k\leq n}D_0^{\{k\}}} & {\prod_{k\leq n}(|D|_{m+1})^{\{k\}}}
	\arrow[from=1-1, to=1-2]
	\arrow[from=1-1, to=2-1]
	\arrow["\lrcorner"{anchor=center, pos=0.125}, draw=none, from=1-1, to=2-2]
	\arrow[from=1-2, to=2-2]
	\arrow[from=2-1, to=2-2]
\end{tikzcd}\]
By proposition \ref{prop:otimes and an}, and using the fact that $D_0$ is a homotopy type, this implies that the outer and the lower square in the diagram
\[\begin{tikzcd}
	{(|D_{[\bullet,n]}|_m)^{a}} & {(|D|_{m+1})^{[a,n]}} \\
	{(|D_{[\bullet,n]}|_m)^{\tau_{m-1}a}} & {(|D|_{m+1})^{[\tau_{m-1} a,n]}} \\
	{\prod_{k\leq n}D_0^{\{k\}}} & {\prod_{k\leq n}(|D|_{m+1})^{\{k\}}}
	\arrow[from=1-1, to=1-2]
	\arrow[from=1-1, to=2-1]
	\arrow[from=1-2, to=2-2]
	\arrow[from=2-1, to=2-2]
	\arrow[from=2-1, to=3-1]
	\arrow[from=2-2, to=3-2]
	\arrow[from=3-1, to=3-2]
\end{tikzcd}\]
are pullbacks. By left cancellation, so is the top one. Combined with the square \eqref{eq:presque fini}, we get a cartesian square:
\[\begin{tikzcd}
	{D_{[a,n]}} & {(|D|_{m+1})^{[a,n]}} \\
	{D_{[\tau_{m-1}a,n]}} & {(|D|_{m+1})^{[\tau_{m-1} a,n]}}
	\arrow[from=1-1, to=1-2]
	\arrow[from=1-1, to=2-1]
	\arrow["\lrcorner"{anchor=center, pos=0.125}, draw=none, from=1-1, to=2-2]
	\arrow[from=1-2, to=2-2]
	\arrow[from=2-1, to=2-2]
\end{tikzcd}\]
This then implies the equivalence
$$D_{[a,n]}\sim\C_{[a,n]}(|D|)$$
In particular, $D_\bullet\to \C_\bullet(|D|)$ is an equivalence when evaluated on globes, and so it is an equivalence.

\vspace{0.5cm}

Now let $A_\bullet$ be a surjective $(\omega,m+1)$-filtration. The induction hypothesis implies that $|\C(A)|_k\to A_k$ is an equivalence for any $k\leq m$, and it then remains to demonstrate the case $m+1$. Note that we have a square:
\begin{equation}
\label{eq:lastsquare}
\begin{tikzcd}
	{|\C(A)|_m} & {|\C(A)|_{m+1}} \\
	{A_m} & {A_{m+1}}
	\arrow[from=1-1, to=1-2]
	\arrow["\sim"', from=1-1, to=2-1]
	\arrow[from=1-2, to=2-2]
	\arrow[from=2-1, to=2-2]
\end{tikzcd}
\end{equation}
where the two horizontal morphisms are $m$-surjective, and so is the right vertical one by left cancellation. Note that the already proven equivalence  $\C_\bullet(A)\to \C_\bullet(|\C(A)|)$ implies that for any pair of morphisms $(f,g):\partial\Db_{m+1}\to |\C(A)|_m\sim A_m$, the $\omega$-category  $\hom_{|\C(A)|_{m+1}}(\overline{f},\overline{g})$ is the pullback of the cospan
\[\begin{tikzcd}
	{\{(f,g)\}} & {A_{m+1}^{\partial\Db_{m+1}}} & {A_{m+1}^{\Db_{m+1}}}
	\arrow[from=1-1, to=1-2]
	\arrow[from=1-3, to=1-2]
\end{tikzcd}\] 
and is then equivalent to $ \hom_{A_{m+1}}(\tilde{f},\tilde{g})$, where $(\overline{f},\overline{g})$ and $(\tilde{f},\tilde{g})$ are the images of $(f,g)$. As the horizontal arrows of the square \eqref{eq:lastsquare} are $m$-surjective, the corollary \ref{cor:characterization of n surjective} implies that the two morphisms
$$\Fun(\partial \Db_{m+1}, |\C(A)|_m)\to  \Fun(\partial \Db_{m+1}, |\C(A)|_{m+1})~~~\Fun(\partial \Db_{m+1},A_m)\to  \Fun(\partial \Db_{m+1},A_{m+1})$$
are surjective on objects.

This then implies that $|\C(A)|_{m+1}\to A_{m+1}$ is $(m+1)$-fully faithful. As we already know that this morphism is $m$-surjective, it is an equivalence. We then have $|\C(A)|_\bullet\sim A_\bullet$, which concludes the proof.
\end{proof}

\begin{cor}
\label{cor:unit is k fully faithful} Let $A_\bullet$ be an object of $\Filt_{n,m}$. For any $k \leq m$, the canonical morphism $$|\C(A)|_k \to A_k$$ is $k$-fully faithful.
\end{cor}

\begin{proof}
Let $A_\bullet$ be an object of $\Filt_{n,m}$. We define the unique sequence $\tilde{A}_\bullet$ fitting in a diagram \[\begin{tikzcd} {A_0} & {\tilde{A}_1} & {\tilde{A}_2} & {...} \\ {A_0} & {A_1} & {A_2} & {...} \arrow["0", two heads, from=1-1, to=1-2] \arrow[equals, from=1-1, to=2-1] \arrow["1", two heads, from=1-2, to=1-3] \arrow["1", hook, from=1-2, to=2-2] \arrow["2", two heads, from=1-3, to=1-4] \arrow["2", hook, from=1-3, to=2-3] \arrow[from=2-1, to=2-2] \arrow[from=2-2, to=2-3] \arrow[from=2-3, to=2-4] \end{tikzcd}\] where arrows labeled as $\overset{k}{\twoheadrightarrow}$ are $k$-surjective, and arrows labeled as $\overset{k}{\hookrightarrow}$ are $k$-fully faithful. We can easily check that the assignment $A_\bullet\mapsto \tilde{A}_\bullet$ is the right adjoint of the canonical inclusion of $\Filt_{n,m}$ into $\Filts_{n,m}$. However, the theorem \ref{theo:eff of nm dcat} implies that $|\C\uvar|$ is also a right adjoint of this inclusion. The two canonical morphisms $|\C(A)|_\bullet\to A_\bullet|$ and $\tilde{A}_\bullet\to A$ then coincide, which concludes the proof.
\end{proof}

\bibliography{biblio}{}
\bibliographystyle{alpha}
\end{document}